\documentclass[twoside]{article}
\usepackage{color}
\usepackage[utf8]{inputenc}
\usepackage[T1]{fontenc}
\usepackage{natbib}
\usepackage[colorlinks,citecolor=blue,linkcolor=blue,urlcolor=blue]{hyperref}
\usepackage{url}
\usepackage{booktabs}
\usepackage{amsfonts}
\usepackage{nicefrac}
\usepackage{microtype}
\usepackage{xcolor}
\usepackage{graphicx}
\usepackage{wrapfig}
\usepackage[rightcaption]{sidecap}
\sidecaptionvpos{figure}{c}
\usepackage{bbm}
\usepackage{latexsym, epsfig, amssymb, amsmath, amsthm, graphicx, mathrsfs,mathtools}
\usepackage{enumerate}
\usepackage{multirow}
\usepackage{multicol}
\usepackage{enumitem,kantlipsum}
\usepackage{epstopdf}
\usepackage{caption}
\usepackage{braket}
\usepackage{subcaption}
\usepackage{adjustbox}
\usepackage{algorithmic,algorithm,bm}
\usepackage{tikz}
\usepackage[most]{tcolorbox}
\usepackage{appendix}
\usepackage{titletoc}

\usetikzlibrary{tikzmark,calc}
\usepackage{soul}
\definecolor{darkred}{RGB}{150,0,0}
\definecolor{darkgreen}{RGB}{0,150,0}
\definecolor{darkblue}{RGB}{0,0,200}
\hypersetup{colorlinks=true, linkcolor=darkred, citecolor=blue, urlcolor=darkblue}
\newtheorem{thm}{Theorem}

\newtheorem{defn}[thm]{Definition}
\newtheorem{cor}[thm]{Corollary}
\newtheorem{exm}{Example}

\newtheorem{rem}[thm]{Remark}

\newtheorem{lem}[thm]{Lemma}
\newtheorem{assumption}{Assumption}

\newcommand{\mc}{\mathcal}
\newcommand{\m}[1]{{\bf{#1}}}

\renewcommand{\mc}[1]{\ensuremath{\mathcal{#1}}} 	

\newcommand{\mb}[1]{{\mathbb{#1}}}

\newcommand{\norm}[1]{\left\lVert#1\right\rVert}
\DeclareMathOperator*{\argmax}{arg\,max}
\DeclareMathOperator*{\argmin}{arg\,min}

\usepackage[accepted]{aistats2024}

\begin{document}
\runningauthor{Davoud Ataee Tarzanagh, Parvin Nazari, Bojian Hou, Li Shen, Laura Balzano}
\twocolumn[

\aistatstitle{Online Bilevel Optimization: Regret Analysis of Online Alternating Gradient Methods
}
\aistatsauthor{ 
Davoud Ataee Tarzanagh \And Parvin Nazari \And  Bojian Hou }
\aistatsaddress{University of Pennsylvania \And  Amirkabir University of Technology \And University of Pennsylvania}
\aistatsauthor{ 
Li Shen \And  Laura Balzano }
\aistatsaddress{University of Pennsylvania  \And University of Michigan}
]
\begin{abstract} 
This paper introduces \textit{online bilevel optimization} in which a sequence of time-varying bilevel problems is revealed one after the other. We extend the known regret bounds for online single-level algorithms to the bilevel setting. Specifically, we provide new notions of \textit{bilevel regret}, develop an online alternating time-averaged gradient method that is capable of leveraging smoothness, and give regret bounds in terms of the path-length of the inner and outer minimizer sequences.
\end{abstract}
\section{Introduction}\label{sec:intro}
Bilevel optimization (BO) is rapidly evolving due to its wide array of applications in modern machine learning problems, including meta-learning~\citep{bertinetto2018meta}, hyperparameter optimization~\citep{feurer2019hyperparameter}, neural network architecture search~\citep{liu2018darts}, data
hypercleaning~\citep{shaban2019truncated}, and reinforcement learning~\citep{wu2020finite}. 
A fundamental assumption in BO, which has been adopted by almost all of the relevant literature~\citep{franceschi2017forward,ghadimi2018approximation,ji2021bilevel}, is that the \textit{inner} and \textit{outer} cost functions do \textit{not} change throughout the horizon over which we seek to optimize. This offline setting may not be suitable to model temporal changes in today's machine learning problems, such as online actor-critic~\citep{vamvoudakis2010online,zhou2020online}, online meta-learning~\citep{finn2019online}, strategic dynamic regression \citep{harris2021stateful}, and sequential decision-making, for which the objective functions are time-varying and are not available to the decision-maker \textit{a priori}. To address these challenges, this paper introduces an \textit{online bilevel optimization} (OBO) setting in which a sequence of bilevel problems is revealed one after the other, and it studies computationally tractable notions of \textit{bilevel regret} minimization. 
\subsection{Background:  Online  Single-Level Optimization}
In online single-level optimization, the setup resembles a game between a learner and an adversary \citep{Hazan16a}.  In each of the repeated decision rounds ($t \in [T]:=\{1,\dots,T\}$), the learner predicts $\m{x}_t \in \mc{X} \subset \mb{R}^{d_1}$, an element within a convex decision set. Simultaneously, the adversary selects a loss function $f_t: \mc{X} \rightarrow \mb{R}$, and the learner observes $f_t(\m{x})$, incurring a loss of $f_t(\m{x}_t)$.  In the \textit{non-static} setting \citep{besbes2015non}, the learner's performance is measured through its single-level dynamic regret 
\begin{equation}\label{eqn:worst:dynamic:regret}
\mbox{D-Reg}_T := \sum^T_{t=1} (f_t(\m{x}_t)-f_t(\m{x}_t^*)),
\end{equation}
where $\m{x}_t^*\in \argmin_{\m{x}\in\mc{X}}f_t(\m{x})$.

In the case of static regret~\citep{zinkevich2003online}, $\m{x}^\ast_t$ is replaced by $\m{x}^\ast \in \argmin_{\m{x} \in \mc{X}} \sum_{t=1}^T f_t(\m{x})$, i.e.,
\begin{equation}\label{eqn:static:regret}
    \mbox{S-Reg}_T := \sum^T_{t=1} (f_t(\m{x}_t)-f_t(\m{x}^*)).
\end{equation}
The static regret \eqref{eqn:static:regret} assumes that the comparators do not change over time. This assumption can be unrealistic in many practical online problems, ranging from motion imagery formation to network analysis, where the underlying environment is dynamic. The parameters $\{\m{x}^{*}_t\}_{t=1}^{T}$ could correspond to frames in a video or the weights of edges in a social network and, by nature, are variable \citep{hall2015online}.

\subsection{Stackelberg Game and Online Bilevel Optimization}\label{sec:obo}
Bilevel optimization, also known as the Stackelberg leader-follower model, involves two players whose choices impact each other's outcomes. One player, the \textit{leader}, possesses knowledge of the other player's objective function, enabling her to predict the follower's choice accurately. Consequently, the leader optimizes her own objective while factoring in the follower's expected response. 
In contrast, the \textit{follower} is only aware of her own objective and must consider how it is influenced by the leader's decisions \citep{stackelberg1952theory}.

\noindent \textbf{Online Bilevel Optimization}: Let  $\m{x}_t \in \mc{X} \subset \mb{R}^{d_1}$ and $f_t: \mc{X} \times \mb{R}^{d_2}\rightarrow \mb{R}$ denote the decision variable and the objective function for the leader, respectively; similarly define $\m{y}_t \in \mb{R}^{d_2}$ and $g_t : \mc{X} \times \mb{R}^{d_2}\rightarrow \mb{R}$ for the follower \footnote{For simplicity of analysis, we use $\mb{R}^{d_2}$ as the follower's decision set.}. In each round $t\in[T]$, knowing the decision $\m{x}_{t-1}$ of the leader and the objective function $g_{t-1}$ of the follower, the follower has to select $\m{y}_{t} \in \mb{R}^{d_2}$ in an attempt to minimize  $g_t(\m{x}_t, \m{y})$ using the information from rounds $t-1, t-2, \ldots, 0$. Being aware of the follower's selection, the leader then moves by selecting $ \m{x}_t \in \mc{X}$ in an attempt to minimize the \textit{bilevel dynamic regret}, defined as:
\begin{subequations}\label{reg:bilevel:dynamic}
\begin{align}\label{eqn:fundiff}
\hspace{-.3cm}\textnormal{BD-Reg}_T
     & := \sum_{t=1}^T (f_t(\m{x}_t, \m{y}_t^*(\m{x}_t))-f_t(\m{x}^*_{t},  \m{y}^*_t(\m{x}^*_{t}))), 
\end{align}
where 
\begin{equation}\label{eqn:obl}
    \begin{split}
 &\m{y}^*_t(\m{x}) \in \argmin_{\m{y} \in  \mb{R}^{d_2}} g_t(\m{x}, \m{y}),~~~\textnormal{and}\\
 &~\m{x}^*_t \in \argmin_{\m{x} \in \mc{X}} f_t\left(\m{x}, \m{y}^*_t(\m{x})\right).        
    \end{split}
\end{equation}
\end{subequations}
Our objective is to design online algorithms with a sublinear bilevel regret, i.e.,  $\textnormal{BD-Reg}_T =o(T)$.

We also study the framework of regret minimization where $\m{x}^\ast_t$ in \eqref{reg:bilevel:dynamic} is replaced by $\m{x}^\ast \in \argmin_{\m{x} \in \mc{X}} \sum_{t=1}^T f_t(\m{x}, \m{y}^*_t(\m{x}))$. In this case, the goal of the leader is to generate a sequence of decisions $\{\m{x}_t\}_{t=1}^T$ so that the following regret can be minimized:
\begin{align}\label{reg:bilevel:static}
 \hspace{-.3cm} \textnormal{BS-Reg}_T:=\sum_{t=1}^T \big(f_t(\m{x}_t, \m{y}^*_t(\m{x}_t))-f_t(\m{x}^*,\m{y}^*_t(\m{x}^*))\big).
\end{align}
Note that the above regret is \textit{not} fully static, as the inner optima  $\{\m{y}^*_t\}_{t=1}^T$ are changing over $T$. 

In general, it is impossible to achieve a sublinear dynamic regret bound due to the arbitrary fluctuations in the time-varying functions~\citep{besbes2015non}. Existing single-level analysis shows that it is indeed possible to bound the dynamic regret in terms of certain regularities of the comparator sequence \citep{zinkevich2003online,besbes2015non}. Hence, in order to achieve sublinear regret, one has to impose some regularity constraints on the sequence of cost functions.  In this work, we define the outer and inner \textit{path-length} (of order $p$) quantities to capture the regularity of the sequences:
\begin{subequations}\label{eqn:pathes}
\begin{equation}\label{eq:pathlength}
    \begin{split}
& P_{p,T} :=  \sum_{t=2}^{T} \|\m{x}^*_{t-1} - \m{x}^*_{t}\|^p,~~\textnormal{and}~~\\
&   Y_{p,T} :=  \sum_{t=2}^{T}\norm{\m{y}_{t-1}^*(\m{x}^*_{t-1}) - \m{y}_{t}^*(\m{x}^*_t)}^p. 
\end{split}
\end{equation}
Here, $P_{p,T}$ is the path-length of the outer minimizers and is widely used for analyzing the dynamic regret of single-level non-stationary optimization; see Table~\ref{table:blo2:results}. $Y_{p,T}$ is a new regularity metric for OBO that measures how fast the minimizers of inner cost functions change.   For simplicity of notation, we set 
\begin{equation}\label{sp}
\begin{split}
S_{p,T}:=P_{p,T} + Y_{p,T}.  
\end{split}
\end{equation}
\end{subequations}
Besides \eqref{eq:pathlength}, other notions of regularity have also been considered in online learning such as function variation \citep{besbes2015non} $V_{T} := \sum_{t=2}^{T} \sup_{\m{x}\in \mc{X}} |f_{t-1}(\m{x}) -  f_t(\m{x})|$ 
and gradient variation \citep{chiang2012online} $G_T := \sum_{t=2}^{T} \sup_{\m{x}\in \mc{X}}  \norm{\nabla f_{t-1}(\m{x}) - \nabla f_t(\m{x})}^2$. For the leader's regret analysis in the static and local settings, we respectively define 
\begin{equation}\label{eqn:static:inner:paths}
\begin{split}
\bar{Y}_{p,T}&:=  \sum_{t=2}^{T}\norm{\m{y}_{t-1}^*(\m{x}^*) - \m{y}_{t}^*(\m{x}^*)}^p, \quad \textnormal{and}\\
H_{p,T} &:=  \sum_{t=2}^{T} \sup_{\m{x}\in \mb{R}^{d_1}} \| \m{y}^*_{t-1}(\m{x}) - \m{y}^*_{t}(\m{x})\|^p        
\end{split}
\end{equation}
for capturing the dynamics of the inner problem.  

The metrics in \eqref{eqn:pathes} and \eqref{eqn:static:inner:paths} are generally not comparable. Appendix~\ref{sec:example} demonstrates their significant differences across various online problems.

\noindent\textbf{Our Results.} 
Our main contributions lie in developing several new results for OBO, including the first-known regret bound. More specifically, we

\noindent $\bullet$ Define new notions of {bilevel regret}, as given in \eqref{reg:bilevel:dynamic} and \eqref{reg:bilevel:static}, which are applicable to a wide class of convex OBO problems. To minimize the proposed regret, we introduce an online alternating gradient descent (OAGD) method capable of leveraging smoothness and provide its regret bounds in terms of the path-length of the inner and/or outer minimizer sequences. 

\noindent$\bullet$  Present a {problem-dependent regret bound} on the proposed dynamic regret that depends solely on the outer and inner path-length in Theorem \ref{thm:dynamic:strong:B-OGD}. We then establish a lower bound for OBO in Theorem~\ref{thm:lower:bound} that matches the upper bound we obtain for smooth strongly convex functions.   Notably, our bound in the single-level setting (\( Y_{1,T} = Y_{2,T} = 0 \)) aligns with the state-of-the-art result in \citep{zhang2017improved}, \textit{without} the need for multiple  gradient queries (multiple updates to \( \mathbf{x}_t \)) as used in their analysis.

\noindent$\bullet$ Introduce a novel notion of {bilevel local regret} defined in \eqref{reg:non:bilevel:dynamic}, which permits efficient OBO in the non-convex setting. 
We give an alternating time-averaged gradient method, and prove in Theorem \ref{thm:dynamic:nonc} that it achieves sublinear bilevel local regret. 

\noindent\textbf{Notation.} Any notation is defined upon its use and is summarized in Table~\ref{tab:notation} for reference.

\begin{table}[t]
\centering
\scalebox{.82}{
\begin{tabular}{ccc|}
\multicolumn{3}{c}{\large{Single-Level Regret Minimization}} 
\vspace{.2cm}
\\
\cline{3-3}
& &\multicolumn{1}{|c|}{{\textbf{Regret~~~~~~~~~}}} 
\\
\hline
\multicolumn{1}{|l|}{SC-D} & \multicolumn{1}{|l|}{{\cite{mokhtari2016online}}} &   $ \mc{O} (1+P_{1,T})$ \\ 
\multicolumn{1}{|c|}{} & \multicolumn{1}{|l|}{ \cite{zhang2017improved}}&  \multirow{1}{*}{ $ \mc{O} \left(1+\min\{P_{1,T}, P_{2,T}\}\right)$} 
\\ 
\cline{1-3}
\multicolumn{1}{|l|}{SC-S} &\multicolumn{1}{|l|}{\cite{hazan2007logarithmic}}&    $ \mc{O} \left( \log{T} \right)$ 
\\
\cline{1-3}
\multicolumn{1}{|c|}{ } & \multicolumn{1}{|l|}{{\cite{besbes2015non}}} &   $\mc{O}(1+G_{T}^{ 1/2}P_{1,T}^{ 1/2})$\\
\multicolumn{1}{|c|}{C-D} &\multicolumn{1}{|l|}{\cite{jadbabaie2015online}} &  $\mc{O}(1+T^{2/3}{V_{T}}^{ 1/3})$
\\
\multicolumn{1}{|l|}{} &\multicolumn{1}{|l|}{{\cite{yang2016tracking}}} & $\mc{O} \left(1+P_{1,T}\right)$
\\
\cline{1-3}
\multicolumn{1}{|l|}{C-S} &\multicolumn{1}{|l|}{{\cite{zinkevich2003online}}} &    $ \mc{O}(\sqrt{T})$ 
\\
\cline{1-3}
\multicolumn{1}{|c|}{NC-L} &\multicolumn{1}{|l|}{{\cite{hazan2017efficient}}} &$\mc{O}(T/W^2)$
\\
\hline
\end{tabular}
}
\hfill
\hspace{.2cm}
\begin{tabular}{ccc|}
\multicolumn{1}{l}{}& \multicolumn{1}{l}{}&\multicolumn{1}{c}{}\\
\multicolumn{3}{c}{Bilevel Regret Minimization}
\vspace{.1cm}
\\
\cline{3-3}
& &\multicolumn{1}{|c|}{{\textbf{Leader's Regret~~~~~}}} 
\\
\hline
\multicolumn{1}{|c|}{SC-D} &\multicolumn{1}{|l|}{{Theorem~\ref{thm:dynamic:strong:B-OGD}}} & $\mc{O} \left( 1+\min\{S_{1,T},S_{2,T}\}\right)$
\\
\multicolumn{1}{|c|}{SC-S}  &\multicolumn{1}{|l|}{{Theorem~\ref{thm:static:strong:B-OGD}}} & $\mc{O} \left(\log T +\bar{Y}_{2,T}\right)$
\\
\cline{1-3}
\multicolumn{1}{|c|}{C-D} &\multicolumn{1}{|l|}{ {Theorem~\ref{thm:dynamic:convex:B-OGD}}} &  $\mc{O} \left(1+S_{1,T}+Y_{2,T}\right)$  
\\
\multicolumn{1}{|c|}{C-S}&\multicolumn{1}{|l|}{ {Theorem~\ref{itm:thm:s:convex} }} & $\mc{O} (\sqrt{T}+ \bar{Y}_{1,T}+\bar{Y}_{2,T})$  
\\
\hline
\multicolumn{1}{|c|}{NC-L}&\multicolumn{1}{|l|}{{Theorem~\ref{thm:dynamic:nonc}}}& $\mc{O}(T/W+ H_{1,T}+   H_{2,T})$
\\
\hline
\end{tabular}
    \caption{  Comparison with prior works on  \textbf{regret minimization}. Here, $W = \sum_{i=0}^{w-1} u_i$ for $\{u_i\}_{i=0}^{w-1}$ with $ 1=u_{0}\geq u_1 \ldots u_{w-1} >0$ (Def.~\ref{def:avg:grad});  SC-D (-S), C-D(-S), and NC-L denote strongly convex-dynamic (-static), convex-dynamic (-static), and non-convex-local settings, respectively.} 
 \label{table:blo2:results}
  \vspace{-3mm}
 \end{table} 
\section{Related Work}\label{sec:related-work} 
\noindent\textbf{Static Regret Minimization:}~\textit{Single-level} static regret (Eq.~\eqref{eqn:static:regret}) is well-studied in the literature of online learning~\citep{shalev2011online}. \citet{zinkevich2003online}  shows that online gradient descent (OGD) provides an $\mc{O}(\sqrt{T})$ regret bound for convex  functions $\{f_t\}_{t=1}^T$. \citet{hazan2007logarithmic} improve this bound to  $\mc{O}(\log T)$ for strongly-convex functions $\{f_t\}_{t=1}^T$.

\noindent\textbf{Dynamic Regret Minimization:} \textit{Single-level} dynamic regret forces the player to compete with time-varying comparators and is thus particularly favored in non-stationary environments~\citep{besbes2015non}.
There are two kinds of dynamic regret in previous studies: \textit{universal dynamic regret} aims to compare with any feasible comparator sequence \citep{zinkevich2003online}, while \textit{worst-case dynamic regret} (defined in \eqref{eqn:worst:dynamic:regret}) specifies the comparator sequence to be the sequence of minimizers of online functions \citep{besbes2015non}. 
We compare regret bounds from related works for the latter case in Table~\ref{table:blo2:results}, as it is the setting studied in this paper. 

\noindent\textbf{Local Regret Minimization:} Approaches to online single-level non-convex optimization include adversarial multi-armed bandit with a continuum of arms \citep{bubeck2008online,heliou2020online} and the Follow-the-Perturbed-Leader (FPL) algorithm with an offline non-convex optimization oracle \citep{agarwal2019learning,suggala2020online}. Complementing this, \citep{hazan2017efficient} considered a local regret that averages a sliding window of gradients at the current model $\m{x}_t$, quantifying the objective of predicting points with small gradients on average.


\noindent\textbf{Bilevel Optimization:}
Since its introduction in~\citep{stackelberg1952theory} and the initial mathematical model by \citep{bracken1973mathematical}, there has been a steady growth in investigations and applications of offline BO~\citep{liu2021investigating}. 
Recently, gradient-based approaches have become popular for their simplicity and efficacy \citep{franceschi2017forward,ghadimi2018approximation,ji2021bilevel,chen2021closing}, yet they assume a offline cost function, a limitation we overcome by exploring new bilevel optimization algorithms in the online setting. Since our initial submission, several OBO studies have emerged \citep{lin2024non,huang2023online}, with \cite{lin2024non} introducing an OBO method that updates $\m{x}_t$ based on an average of recent hypergradient estimates, enabling scalable OBO through an approximate Hessian-inverse vector product by solving a linear system.
\section{Algorithm and Regret Bounds}\label{sec:algplusreg}

In this section, we provide bilevel regret bounds based on the regularities defined in \eqref{eqn:pathes} and \eqref{eqn:static:inner:paths}. We first list assumptions for OBO.
%
\begin{assumption}\label{assu:f} 
Let $\m{z} = [\m{x}; \m{y}]$ and $\m{z}' = [\m{x}'; \m{y}']$, where $\m{x}, \m{x}' \in \mc{X}$ and $\m{y}, \m{y}' \in \mathbb{R}^{d_2}$. For all $t \in [T]$:
\begin{enumerate}[label={\textnormal{\textbf{A\arabic*}.}}, wide, labelindent=0pt, itemsep=0pt]
\vspace{-.3cm}
\item\label{assu:f:a1} $f_t$ is $\ell_{f,0}$-Lipschitz continuous; for all $\m{z}$ and $\m{z}'$, there exists a constant $\ell_{f,0}$ such that
\[ \|f_t(\m{x}, \m{y}) - f_t(\m{x}', \m{y}')\| \leq \ell_{f,0} \|\m{z} - \m{z}'\|. \]
\item\label{assu:f:a3} $g_t(\m{x}, \m{y})$ is $\mu_{g}$-strongly convex in $\m{y}$; for all $\m{x} \in \mc{X}$ and $\m{y}, \m{y}' \in \mathbb{R}^d$, there exists a constant $\mu_g>0$ such that
\[ g_t(\m{x}, \m{y}') \geq g_t(\m{x}, \m{y}) + \langle \nabla_{\m{y}} g_t(\m{x}, \m{y}), \m{y}' - \m{y} \rangle + \frac{\mu_g}{2} \|\m{y} - \m{y}'\|^2. \]
\item\label{assu:f:a2} $\nabla f_t$, $\nabla g_t$, and $\nabla^2 g_t$ are respectively $\ell_{f,1}$, $\ell_{g,1}$, and $\ell_{g,2}$-Lipschitz continuous; for all $\m{z}$ and $\m{z}'$, there exist  constants $\ell_{f,1}, \ell_{g,1}, \ell_{g,2} $ such that
\begin{align*}
    \|\nabla f_t(\m{x}, \m{y}) - \nabla f_t(\m{x}', \m{y}')\| &\leq \ell_{f,1} \|\m{z} - \m{z}'\|, \\
    \|\nabla g_t(\m{x}, \m{y}) - \nabla g_t(\m{x}', \m{y}')\| &\leq \ell_{g,1} \|\m{z} - \m{z}'\|, \\
    \|\nabla^2 g_t(\m{x}, \m{y}) - \nabla^2 g_t(\m{x}', \m{y}')\| &\leq \ell_{g,2} \|\m{z} - \m{z}'\|.
\end{align*}
\end{enumerate}
\end{assumption}

Assumption~\ref{assu:f} necessitates well-behaved $\{(f_t,g_t)\}_{t=1}^T$, typical in offline BO~\citep[Assumptions~1 and 2]{chen2021closing}. Throughout, we use $\kappa_g:={\ell_{g,1}}/{\mu_g}$ to denote the condition number of online inner functions $\{g_t\}_{t=1}^T$.
\begin{assumption}\label{assu:dom}
The non-empty closed and convex decision set $\mc{X} \subseteq \mb{R}^{d_1}$ is bounded, i.e., $\|\m{x} -\m{x}'\| \leq D$ for any $\m{x}, \m{x}' \in \mc{X}$ and some $D>0$. Further, $\|\m{y}_1 -\m{y}_1^*(\m{x}_1)\| \leq D'$ for some $D'>0$.
\end{assumption}
Assumption~\ref{assu:dom} is similar to the existing assumptions on the decision set in online single-level learning \citep{Hazan16a,zinkevich2003online}.

\subsection{OBO with (Hyper-)Gradient Information}\label{sec:path-bound}

Perhaps the simplest algorithm that applies to the most general setting of online (single-level) optimization is OGD~\citep{zinkevich2003online}: For each $t \in [T]$, play $\m{x}_{t} \in \mc{X}$, observe the function $f_t$, and set
\begin{equation}\label{eqn:ogd}
\m{x}_{t+1} = \Pi_{\mc{X}}\big(\m{x}_t - \alpha_t \nabla f_t(\m{x}_t)\big), ~~~ \alpha_t>0, \tag{OGD}
\end{equation}
where $\Pi_{\mc{X}}$ is the projection onto $\mc{X}$.

We consider a natural extension of \ref{eqn:ogd} to the bilevel setting (containing inner and outer OGD) and demonstrate that it exhibits regret bounds based on the path-length of the inner and/or outer minimizer sequences. To do so, we need to compute the gradient of the outer objective (called hypergradient)  $\nabla f_t(\m{x}, \m{y}^*_t(\m{x}))$ where  $\m{y}^*_t(\m{x})$ is defined in \eqref{eqn:obl}.  The computation of $\nabla  f_t(\m{x}, \m{y}^*_t(\m{x}))$ involves Jacobian $\nabla_\m{xy}^2 g_t(\m{x},\m{y}^*_t(\m{x}))$ and Hessian $\nabla_\m{y}^2 g_t(\m{x},\m{y}^*_t(\m{x}))$. More concretely, since $\nabla_\m{y} g_t(\m{x}, \m{y}^*_t(\m{x})) =0$, it follows from  Assumption~\ref{assu:f} and the implicit function theorem 
\begin{equation*}
 \nabla \m{y}^*_t(\m{x}) \nabla^2_{\m{y}}g_t\left(\m{x},\m{y}^*_t(\m{x}) \right) +  \nabla^2_{\m{x}\m{y}} g_t \left(\m{x},\m{y}^*_t (\m{x}) \right)=0,
\end{equation*}
which together with the chain rule gives
\begin{align*}
\nonumber
 \nabla  f_t(\m{x}, \m{y}_t^*(\m{x}))&=  \nabla_\m{x} f_t \left(\m{x}, \m{y}^*_t (\m{x}) \right) 
 \\&+\nabla \m{y}^*_t(\m{x}) \nabla_\m{y} f_t\left(\m{x},\m{y}^*_t(\m{x}) \right).
\end{align*}
%
The exact gradient $\nabla f_t(\m{x}, \m{y}^*_t(\m{x}))$ is generally not available, preventing the use of gradient-type methods for bilevel regret minimization. In this work, inspired by offline bilevel optimization~\citep{domke2012generic,ghadimi2018approximation} and online single-level optimization ~\citep{hazan2017efficient,aydore2019dynamic}, we define a new \textit{time-averaged hypergradient} as a surrogate of $\nabla f_t(\m{x}, \m{y}^*_t(\m{x}))$ by replacing $\m{y}^*_t (\m{x}_t)$ with its estimation $\m{y}_t\in \mb{R}^{d_2}$ and using the history of the hypergradients.
\begin{defn}[\textnormal{\textbf{Time-Averaged Hypergradient}}]\label{def:avg:grad}
Given a window size $w \in[T]$, let $\{u_{i}\}_{i=0}^{w-1}$ be a positive decreasing sequence with $u_{0}=1$. Let $F_{t,\m{u}}(\m{x}, \m{y}) := (1/W) \sum_{i=0}^{w-1} u_{i}f_{t-i}(\m{x},\m{y})$ with $W=\sum_{i=0}^{w-1} u_{i}$ and the convention $f_{t} \equiv 0$ for $t\leq 0$. Let $\m{M}_t(\m{x}, \m{y})$ be the solution of the following linear equation:
\begin{equation*}
\m{M}_t (\m{x}, \m{y}) \nabla^2_{\m{y}}g_{t}\left(\m{x}, \m{y}\right)+\nabla^2_{\m{x}\m{y}} g_{t} \left(\m{x}, \m{y}\right) =0.
\end{equation*}
Then, the time-averaged hypergradient is defined as
\begin{align}\label{eqn:tave:grad}
 \tilde{\nabla} F_{t,\m{u}}(\m{x}, \m{y})   := \frac{1}{W} \sum_{i=0}^{w-1} u_{i} \tilde{\nabla}  f_{t-i}(\m{x}, \m{y}), 
\end{align}
where \begin{align}\label{eq:bar:gradient}
\hspace{-.2cm}\tilde{\nabla}f_t(\m{x},\m{y}) := \nabla_\m{x} f_t(\m{x},\m{y}) +\m{M}_t (\m{x}, \m{y}) \nabla_\m{y} f_t(\m{x},\m{y}).
 \end{align}
\end{defn}
\begin{rem}\label{rem:seq:ui}
$\tilde{\nabla} F_{t,\mathbf{u}}(\mathbf{x}, \mathbf{y})$ is defined using the hypergradients of the losses from the $w$ recent rounds. By setting $u_i = 1$, it averages a sliding window of online hypergradients at each update. With $u_i = \delta^i$ for $\delta \in (0,1)$, it emphasizes recent values, giving an exponential average of hypergradients. Although $\tilde{\nabla} F_{t,\mathbf{u}}(\mathbf{x}, \mathbf{y})$ seems computationally intensive for large $w$, the $w$ terms can be processed in parallel, mitigating the cost.
\end{rem}
%
\begin{algorithm}[t]
\caption{:~\pmb{OAGD}~for Bilevel Regret Minimization}\label{alg:obgd}
\begin{algorithmic}[1]
\REQUIRE{Initial values $(\m{x}_{1},\m{y}_1)\in \mc{X} \times \mb{R}^{d_2}$; parameters $w, T, K_1, K_2, \ldots, K_T \in \mb{N}$; stepsizes $\{(\alpha_t, \beta_t) \in \mb{R}_{++}^2\}_{t=1}^T$; and weights $\{u_{i}\}_{i=0}^{w-1}$ with $1 = u_{0} \geq u_1 \geq \ldots \geq u_{w-1} > 0$.}
\FOR{$t=1$ {\bfseries to} $T$}
    \STATE Acquire information about functions $f_t$ and $g_t$
    \STATE Set $\m{z}_t^1 \leftarrow \m{y}_{t}$
    \FOR{$k=1$ {\bfseries to} $K_t$}
        \STATE Update $\m{z}_{t}^{k+1} \leftarrow \m{z}_{t}^{k} - \beta_t \nabla_\m{z} g_t(\m{x}_t, \m{z}_{t}^{k})$
    \ENDFOR
    \STATE Update $\m{y}_{t+1} \leftarrow \m{z}_{t}^{K_t+1}$
    \STATE Update $\m{x}_{t+1} \leftarrow \Pi_{\mc{X}} \big[\m{x}_t - \alpha_t \tilde{\nabla} F_{t,\m{u}}(\m{x}_{t}, \m{y}_{t+1})\big]$
\ENDFOR
\end{algorithmic}
\end{algorithm}

The pseudo-code for the online alternating gradient descent (OAGD) method is presented in Algorithm~\ref{alg:obgd}. This algorithm is very simple to implement. At each timestep $t \in [T]$,  OAGD alternates between the gradient update on $\m{y}_t$ and the time-averaged projected hypergradient on $\m{x}_t$. One can notice that the alternating update in Algorithm~\ref{alg:obgd} serves as a template for running \ref{eqn:ogd} on OBO problems. In OAGD, $w$ and $u_i$ are tunable parameters. Remark~\ref{rem:seq:ui} and Theorem~\ref{thm:dynamic:nonc} provide suggested values for them. Intuitively, the value of $w$ captures the level of averaging (smoothness) of the hypergradient at round $t$. 

We note that OAGD is similar to single-level time-smoothing OGD-type methods for the outer variable update \citep{hazan2017efficient}. Also, without the inner variable and by setting the window size $w=1$, OAGD reduces to \ref{eqn:ogd}. It should be mentioned that $w>1$ is not required for our bilevel dynamic and static regret minimization. However, evaluations in Section~\ref{sec:exp} reveal that Equation \eqref{eqn:tave:grad} with $w>1$ provides a performance boost over the case $w=1$. Finally, we note that for $w=1$, Algorithm~\ref{alg:obgd} is similar to the gradient methods for offline BO \citep{ghadimi2018approximation,chen2021closing,ji2021bilevel}.
\begin{lem}\label{lem:lips}
Under Assumption~\ref{assu:f}, for all $t \in [T]$, $\m{x}, \m{x}' \in \mc{X}$,  and $\m{y} \in \mb{R}^{d_2}$, we have
\begin{subequations}\label{eqn:newlips}
\begin{align*}
&\left\|\m{y}^*_t(\m{x})-\m{y}^*_t(\m{x}')\right\|\leq  L_{\m{y}} \left\|\m{x}-\m{x}'\right\|,\\
&\|\tilde{\nabla} f_t(\m{x},\m{y}) - \nabla f_t(\m{x},\m{y}^*_t(\m{x}))\| \leq  M_f \left\|\m{y}-\m{y}^*_t(\m{x})\right\|,\\
\nonumber
&\left\|\nabla f_t(\m{x},\m{y}^*_t(\m{x}))- \nabla f_t(\m{x}',\m{y}^*_t(\m{x}'))\right\|\leq  L_f \left\|\m{x}-\m{x}'\right\|.
\end{align*}
\end{subequations}
Here, $L_{\m{y}}=\mc{O}(\kappa_g)$, $M_f=\mc{O}(\kappa_g^2)$, and $L_f= \mc{O}(\kappa_g^3)$.
\end{lem}
The proof of Lemma~\ref{lem:lips} is provided in Appendix \ref{C21}.
\subsection{Main Results}
This section presents the convergence results of the OAGD algorithm.
In Theorems~\ref{thm:dynamic:strong:B-OGD}--\ref{itm:thm:s:convex}, we simplify the analysis by setting $w=1$. For a summary and comparison of these results with the single-level setting, we refer to Table~\ref{table:blo2:results}. Proofs can be found in Appendix \ref{app:addend:sec:algplusreg}.
\begin{thm}[\textbf{Strongly-Convex Dynamic}]\label{thm:dynamic:strong:B-OGD}Suppose  Assumptions~\ref{assu:f}--\ref{assu:dom} hold and $\{ f_t(\m{x}, \m{y}^*_t(\m{x}))\}_{t=1}^T$ are $\mu_{f}$-strongly convex. Then, Algorithm~\ref{alg:obgd} with
\begin{align*}
&\beta_t=\beta= \frac{2}{\ell_{g,1}+\mu_{g}}, ~~\alpha_t=\alpha \leq \min\big\{\frac{1}{\ell_{f,1}}, \frac{\mu_f}{128M_f^2 L_{\m{y}}^2}\big\},\\
&K_t = K >  \left\lceil 0.25(\kappa_g+1) \log \big(4(\frac{1}{\alpha \mu_f}+2)^2 \big)\right\rceil, 
\end{align*}
for all $t \in [T]$, satisfies the follwoing 
\begin{equation}\label{eq:bdreg:bound:strong}
\begin{aligned}
\textnormal{BD-Reg}_T 
&\leq
\mc{O} \Big(1+\min\big\{S_{1,T},
\\&\sum^T_{t=1} \norm{\nabla f_t(\m{x}^*_t, \m{y}^*_t(\m{x}^*_t) )}^2+S_{2,T}\big\}\Big).
\end{aligned}
\end{equation}
\end{thm}
Theorem~\ref{thm:dynamic:strong:B-OGD} shows that Algorithm~\ref{alg:obgd}, using fixed step sizes and \( K_t = \tilde{\mc{O}} (\kappa_g) \), achieves a problem-dependent regret bound. While it might appear advantageous to increase \( K_t \), our analysis suggests that even as \( K_t \) approaches infinity, the regret bound only improves by a constant factor.  

In single-level online setting, \citep{zhang2017improved} shows that if  $\sum^T_{t=1} \norm{\nabla f_t(\m{x}^*_t )}^2=\mc{O}(P_{2,T})$, the dynamic regret bound of \ref{eqn:ogd} can be further improved to \( \mathcal{O} \left( 1 + \min\{ P_{1,T}, P_{2,T} \} \right) \) by allowing multiple gradient queries (resulting in multiple updates to \( \mathbf{x}_t \)). When \( Y_{1,T} = Y_{2,T} = 0 \), we achieve a similar dynamic regret bound \textit{without} multiple gradient queries.

If $f_t=f$ and $g_t=g$, then $S_{1,T}=S_{2,T}=0$ implies a regret bound of $\mathcal{O}(1)$, leading to convergence rates for offline bilevel gradient methods~\citep{ghadimi2018approximation}. If the difference between consecutive inner and outer arguments decreases as $1/t$, then $P_{1,T} = Y_{1,T} = \mathcal{O}(\log T)$, resulting in a logarithmic regret bound.

The following theorem provides the lower bound $\Omega(1+S_{2,T} )$ for OBO. 
\begin{thm}[\textbf{Lower Bound}]\label{thm:lower:bound}
For any OBO algorithm, there always exists a sequence of smooth and strongly convex functions $\{(f_t,g_t)\}_{t=1}^T$ such that
 \begin{equation*}
\textnormal{BD-Reg}_T=\Omega(1+S_{2,T}).
 \end{equation*}
\end{thm}
Theorem \ref{thm:lower:bound} indicates that the upper bound in Theorem \ref{thm:dynamic:strong:B-OGD} cannot be improved in general.  

\begin{thm}[\textbf{Strongly-Convex Static}]\label{thm:static:strong:B-OGD}
Suppose Assumptions~\ref{assu:f}--\ref{assu:dom} hold, and $\{ f_t(\m{x}, \m{y}^*_t(\m{x}))\}_{t=1}^T$ are $\mu_{f}$-strongly convex. Then, Algorithm~\ref{alg:obgd} with
\begin{align*}
&\beta_t=\beta= \frac{2}{\ell_{g,1}+\mu_{g}},~~\alpha_{t} = \frac{2}{\mu_f t}, \\
&K_t = K >   \Big\lceil 0.25(\kappa_g+1) \log \big((\frac{24L_\m{y}M_f}{\mu_f})^2 +2 \big) \Big\rceil ,
\end{align*}
for all $t \in [T]$, satisfies the follwoing 
\begin{equation}\label{eq:bsreg:bound:strong}
\begin{split}
\textnormal{BS-Reg}_T \leq \mc{O}\left(\log T+ \bar{Y}_{2,T} \right).
\end{split}
\end{equation}
\end{thm}

Theorem~\ref{thm:static:strong:B-OGD} shows that Algorithm~\ref{alg:obgd}, with decreasing $\alpha_t$ and \( K_t = \tilde{\mc{O}}(\kappa_g) \), achieves a problem-dependent regret bound, where the \(\log T\) term mirrors single-level static  findings \citep{hazan2007logarithmic}, and \( \bar{Y}_{2,T}\) accounts for the variability in \(\{\m{y}^*_t(\m{x}^*)\}_{t=1}^T\) over \(T\).

The following theorem provides the regret bounds for online convex functions $\{f_t\}_{t=1}^T$ in the dynamic setting.

\begin{thm}[\textbf{Convex Dynamic}]\label{thm:dynamic:convex:B-OGD}  Suppose Assumptions~\ref{assu:f}--\ref{assu:dom} hold, $\{ f_t(\m{x}, \m{y}^*_t(\m{x}))\}_{t=1}^T$ are convex and $ \exists~(\m{x}^*_{t}, \m{y}^*_t(\m{x}^*_{t})) \in \mc{X} \times \mb{R}^{d_2}$ such that $\nabla f_t(\m{x}^*_t, \m{y}^*_t(\m{x}^*_t) )=0 $ for all $t \in [T]$. Then, Algorithm~\ref{alg:obgd} with 
\begin{align*}
&\beta_t=\beta= \frac{2}{\ell_{g,1}+\mu_{g}},~~ \alpha_t=\alpha \leq \frac{1}{4L_f},
\\&K_t>   \left\lceil 0.25 ( \kappa_g+1) \log 4t^2\right\rceil,
\end{align*}
 for all $t \in [T]$, satisfies the following
\begin{align}\label{eqn:convex:dbound}
\textnormal{BD-Reg}_T& \leq \mc{O} \left( 1+S_{1,T}+Y_{2,T}\right).
\end{align}
\end{thm}

From Theorem~\ref{thm:dynamic:convex:B-OGD} we see that Algorithm~\ref{alg:obgd} achieves an $\mc{O} \left(1+S_{1,T}+Y_{2,T}\right)$ dynamic regret for a sequence of loss functions that satisfy Assumption~\ref{assu:f} with only  gradient feedback. 
Note that the condition $\nabla f_t(\m{x}^*_t, \m{y}^*_t(\m{x}^*_t) )=0$ is referred to as the vanishing gradient condition, which is widely used in the analysis of OGD methods in the single-level convex setting~\cite[Assumption 2]{yang2016tracking}. 

\begin{thm}[\textbf{Convex Static}]\label{itm:thm:s:convex} Suppose Assumptions~\ref{assu:f}-\ref{assu:dom} hold and functions $\{ f_t(\m{x}, \m{y}^*_t(\m{x}))\}_{t=1}^T$ are convex.  Then, Algorithm~\ref{alg:obgd} with
\begin{align*}
 &\beta_t=\beta= \frac{2}{\ell_{g,1}+\mu_{g}}, ~~\alpha_t =\frac{D}{\ell_{f,0} \sqrt{t}},
 \\&K_t>  \left\lceil 0.25 ( \kappa_g+1) \log 4t^2\right\rceil,
\end{align*}
for all $t \in [T]$, satisfies the following 
\begin{align}\label{eqn:convex:sbound}
\textnormal{BS-Reg}_T& \leq   \mc{O} \left( \sqrt{T} + \bar{Y}_{1,T}+\bar{Y}_{2,T}\right). \end{align}
\end{thm}
Theorem~\ref{itm:thm:s:convex} provides static bounds when the cost functions are convex. We observe that the term $\sqrt{T}$ is identical to the bound in the single-level static setting \citep{zinkevich2003online}, and $\bar{Y}_{1,T}$ and $\bar{Y}_{2,T}$ account for the variability in $\{\m{y}^*_t(\m{x}^*)\}_{t=1}^T$ over $T$. Additionally, we note that $\bar{Y}_{1,T}$ and $\bar{Y}_{2,T}$ are not generally comparable; see Example~\ref{exm:not:comparable} in Appendix \ref{sec:example} for further discussion.

\subsubsection{Local Regret Minimization}\label{sec:nonconv}

In this section, we consider online bilevel learning with non-convex outer losses. While minimizing the regret \eqref{reg:bilevel:dynamic} makes sense for online convex functions $\{f_t\}_{t=1}^T$, it is not appropriate for general non-convex online costs, as the global minimization of a non-convex objective is generally intractable. We address this issue with a combined approach, leveraging optimality criteria and measures from offline non-convex bilevel analysis, together with smoothing of the online part of the outer objective function similar to \citep{hazan2017efficient}. Throughout this section, we set $\mc{X}\equiv\mb{R}^{d_1}$.

For the sequence $\{u_i\}_{i=0}^{w-1}$ given in Definition~\ref{def:avg:grad} and for all $w \in[T]$,  we define the following \textit{bilevel local regret}:
\begin{align}\label{reg:non:bilevel:dynamic}
\textnormal{BL-Reg}_{T,\m{u}} :=
\sum_{t=1}^T  \big\| \nabla F_{t,\m{u}} (\m{x}_{t}, \m{y}_{t}^*(\m{x}_{t})) \big\|^2.
\end{align}
Here, $\m{y}^*_t(\m{x}) \in \argmin_{\m{y} \in  \mb{R}^{d_2}} g_t(\m{x}, \m{y})$, and
\begin{align*}
 F_{t,\m{u}} (\m{x}_{t}, \m{y}_{t}^*(\m{x}_{t}))=  \frac{1}{W} \sum_{i=0}^{w-1} u_{i}  f_{t-i}(\m{x}_{t},\m{y}^*_{t}(\m{x}_{t})) 
\end{align*}
 with the convention $f_{t} \equiv 0$ for $t\leq 0$.

Note that in the single-level setting, \eqref{reg:non:bilevel:dynamic} with $ u_{i}=1$ simplifies to the local regret in \citep{hazan2017efficient}. 

For local regret analysis, we use $H_{p,T}$, as defined in \eqref{eqn:static:inner:paths}, to measure the variation of $\m{y}^*_{t}(\m{x})$. We introduce $H_{p,T}$ to account for cases where its value is inherently small. For instance, in the online problems discussed in Section~\ref{sec:exp}, $H_{p,T}$ represents model variability over $T$, which can be small for a good range of hyperparameters $\m{x} \in\mb{R}^{d_1}$.

\begin{assumption}\label{assu:f:b}
For all $t \in [T]$, $|f_t(\m{x}, \m{y} )| \leq M$ for some finite constant $M > 0$.
\end{assumption}
Assumption~\ref{assu:f:b} is widely used in the online learning literature \citep{hazan2017efficient}. The following theorem demonstrates the sublinear local regret of OAGD.
\begin{thm}[\textbf{Non-convex Local}]\label{thm:dynamic:nonc}
Suppose Assumptions~\ref{assu:f} and \ref{assu:f:b} hold. Then, Algorithm~\ref{alg:obgd} with
\begin{align*}
&\beta_t=\beta = \frac{2}{\ell_{g,1}+\mu_g},~~  K_t=1,
\\&\alpha_t=\alpha    \leq  \min\Big\{\frac{1}{8L_f},\frac{1}{2\sqrt{2}L_{\m{y}} M_f(\kappa_g^2-1  )^{1/2}}\Big\},
\end{align*}
 for all $t \in [T]$,  satisfies the following
\begin{align} \label{eqn:localreg:bound}
\textnormal{BL-Reg}_{T,\m{u}} &\leq \mc{O}\Big( \frac{ T}{ W}+  H_{1,T}+   H_{2,T}\Big).
\end{align}
\end{thm}
The above regret can become sublinear in \( T \) provided \( H_{1,T} = o(T) \), \( H_{2,T} = o(T) \),  and the weight \( w \) is appropriately chosen so that $W=o(T)$. Theorem~\ref{thm:dynamic:nonc} aligns closely with the existing bounds in various non-convex optimization contexts. In the OBO setting, it parallels \cite[Theorem 5.7]{lin2024non}, yet it does not require monitoring fluctuations between online objective functions. When \( H_{p,T} = 0 \), as in a single-level setting, Theorem~\ref{thm:dynamic:nonc} matches the findings of \citep{hazan2017efficient} but applies to a broader range of weight sequences $\{u_{i}\}_{i=0}^{w-1}$. It is important to note that when \( u_i = 1 \) for all \( i \in \{0, \ldots,w-1\}\), it yields a local regret bound of \( \mathcal{O}( T/w + H_{1,T} + H_{2,T} ) \). For the offline case where \( f_t = f \), the results provide a convergence guarantee for non-convex BO \citep{ghadimi2018approximation}.

\begin{figure*}[htbp]
		\includegraphics[width=5.6cm]{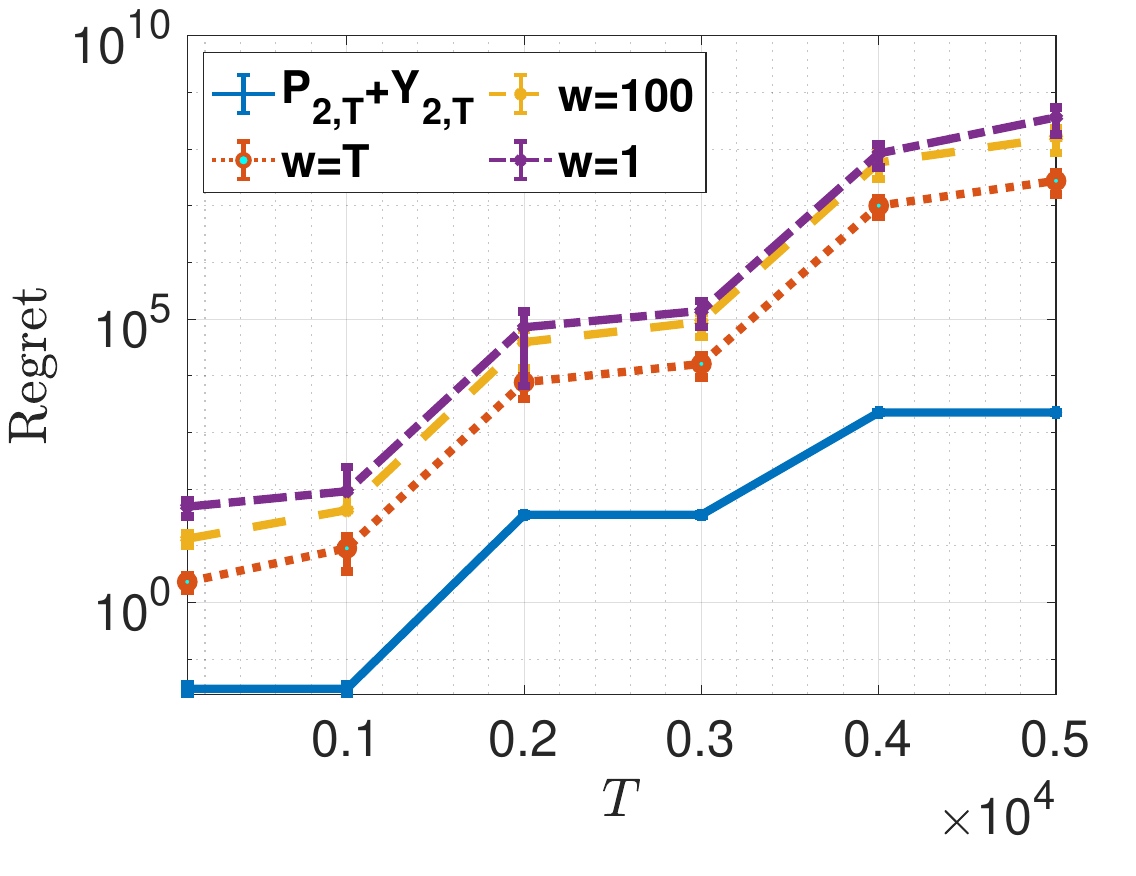}
		\hspace{.05em}
 		\includegraphics[width=5.6cm]{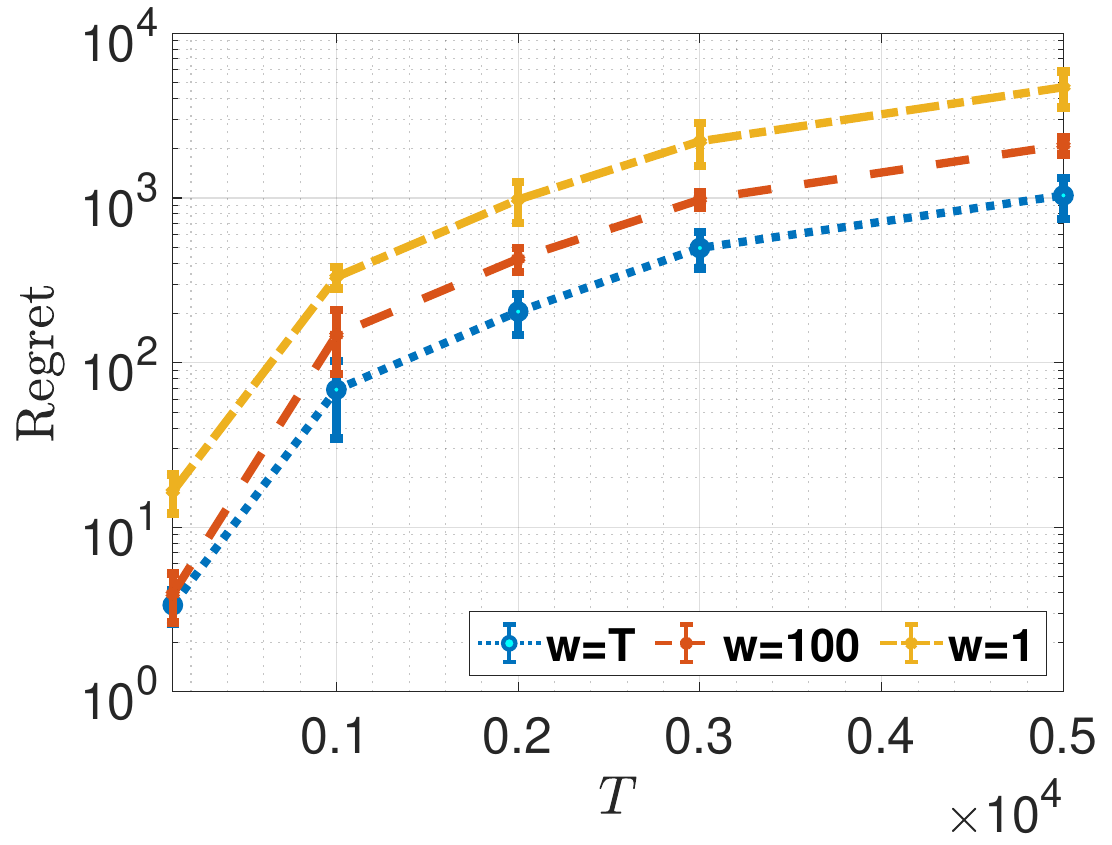} 
   	\hspace{.05em}
   \includegraphics[width=5.6cm]{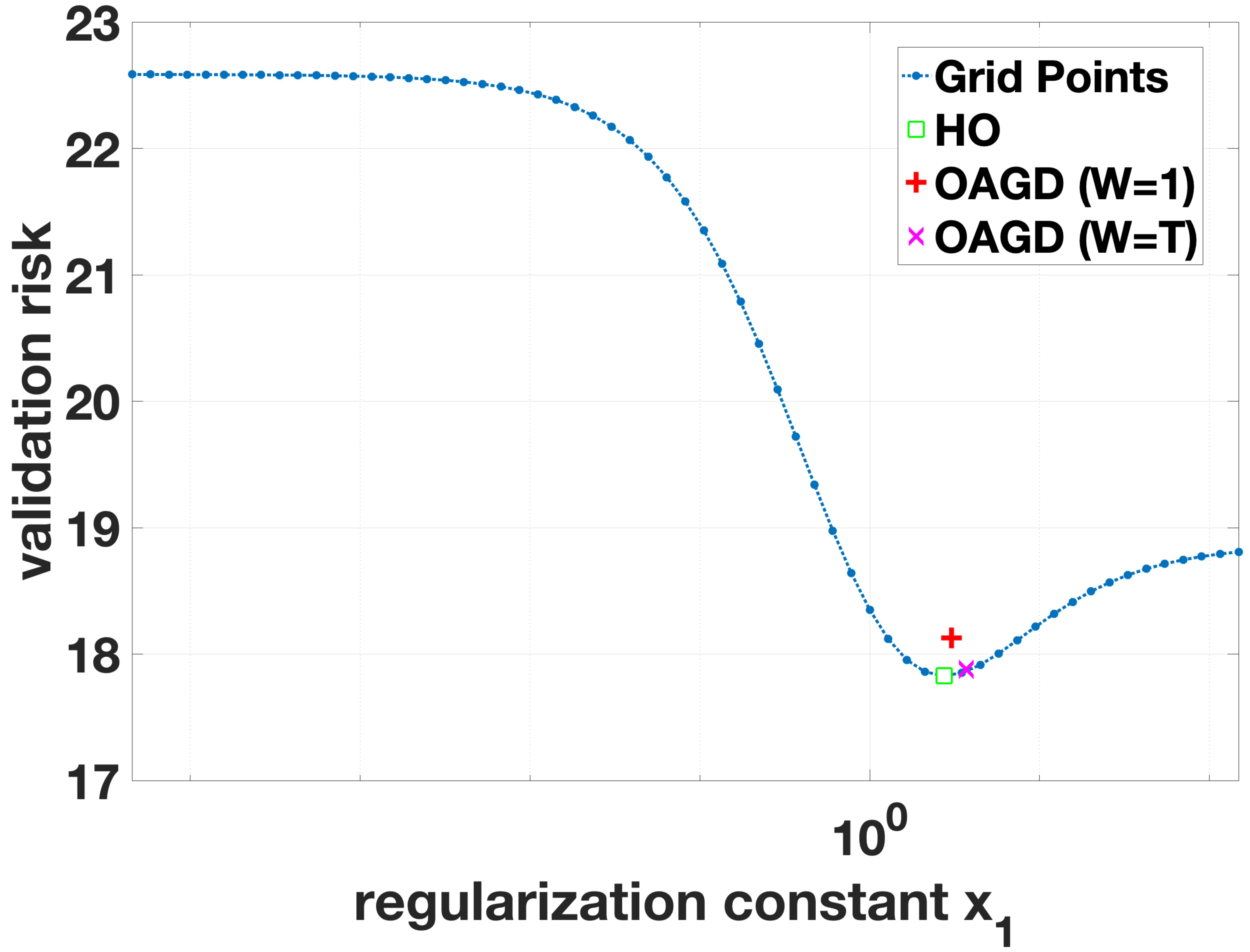}
\caption{Performance of OAGD in online hyperparameter learning over five runs. The left and middle figures show OBO's regret with three comparators and a fixed comparator, respectively. The right figure illustrates the outer problem's trajectories and the performance of OAGD and offline HO in learning the hyperparameter $x_1$.}
  \label{fig:app:onl}
\end{figure*}

\section{Experimental Results}\label{sec:exp}
In this section, we conduct preliminary experiments to evaluate OAGD performance, with additional experiments available in  Appendix \ref{app:sec:experiments}. Code is available at \url{https://github.com/BojianHou/OAGD}. 
\subsection{Online Hyperparameters Learning for
Dynamic Regression}\label{sec:exp:dynamic:regression}
Hyperparameter optimization (HO) is the process of finding the best set of hyperparameters that cannot be learned using the training data alone  \citep{franceschi2018bilevel}. An HO problem can be formulated as a BO problem. The outer objective, $f(\m{y}^*(\m{x});\mc{D}^{\text{val}})$, aims to minimize the validation loss concerning the hyperparameters $\m{x}$. Meanwhile, the inner objective, $g(\m{x},\m{y};\mc{D}^{\text{tr}})$, minimizes the training loss concerning the model parameters $\m{y}$. 

We consider online HO for dynamic regression as follows: At each round or timestep $t$, new samples $(\m{a}_t, b_t) \in \mc{D}_t:= \{\mc{D}^{\text{val}}_t, \mc{D}^{\text{tr}}_t\}$ for all $t \in [T]$ are received, where $\m{a}_t \in \mb{R}^{d_2}$ represents the feature vector and $b_t \in \mb{R}$ is the corresponding target. It's important to note that the potential correct decision can change abruptly. Specifically, we consider an $S$-stage scenario where $(\m{x}_s^*,\m{y}_s^*(\m{x}_s^*))$ represents potentially the best decisions for the $s$-th stage, encompassing all $s\in[S]$:
\begin{align}\label{eqn:onl:hyper}
\begin{array}{ll}
&\m{x}^*_s \in \underset{\m{x} \in\mc{X}}{\textnormal{argmin}} 
\begin{array}{c}
 \sum_{t=1}^{T_s} f\left(\m{y}^*_s(\m{x}); \mc{D}^{\text{val}}_t \right)
\end{array}\\
&\text{s.t.}  \begin{array}[t]{l} \m{y}^*_s(\m{\m{x}})
\in \underset{ \m{y}\in \mb{R}^{{d}_2}}{\textnormal{argmin}}~~\sum_{t=1}^{T_s} g\left(\m{x},\m{y};  \mc{D}^{\text{tr}}_t\right). 
\end{array}
\end{array}
\end{align}
At each round $t$ of online HO, given a sample $(\m{a}_t, b_t) \in \mc{D}^{\text{tr}}_t$, the follower is required to make the prediction by $\m{a}^\top_t \m{y}_t$ based on the learned inner and outer models $(\m{x}_{t-1},\m{y}_{t-1}) \in \mc{X} \times \mb{R}^{d_2}$; then, as a consequence
the follower suffers a loss $g(\m{x}_{t-1},\m{y}_t; \mc{D}^{\text{tr}}_t)=1/2 (\m{a}^\top_t \m{y}_t-b_t)^2 + \m{y}^\top_t \m{C}(\m{x}_{t-1}) \m{y}_t$, where $\m{C}(\m{x}):=\textnormal{diag}(\exp(x_i))_{i=1}^{d_1}$. The leader then receives the feedback of the inner model, i.e., $\m{y}_{t}$, predicts the new hyperparameter $\m{x}_t$ using a validation sample $(\m{a}_t, b_t) \in \mc{D}^{\text{val}}_t$, and suffers the loss $f(\m{y}_t(\m{x}_t); \mc{D}^{\text{val}}_t)= 1/2(\m{a}^\top_t \m{y}_t(\m{x}_t)-b_t)^2$. This process repeats across $T=\{T_1, \ldots, T_S\}$ rounds. 

Figure~\ref{fig:app:onl} (left and middle) shows the variation $P_{2,T}+Y_{2,T}$ and the regret bound of OAGD with three different window sizes $w\in\{1,100, T\}$ on synthetic data; see Appendix~\ref{app:sec:experiments:dynamic:regression} for further details. We observe that OAGD with $w=T$ performs the best, with a gradual decrease in performance as $w$ decreases to $100$ and $w=1$. Additionally, Figure~\ref{fig:app:onl}  (right) demonstrates that the performance of OAGD is comparable to the performance of the offline HO \citep{franceschi2018bilevel}. 

\subsection{Online Parametric Loss Tuning for
Imbalanced Data}
\begin{figure*}
\begin{center}
\includegraphics[scale=0.34]{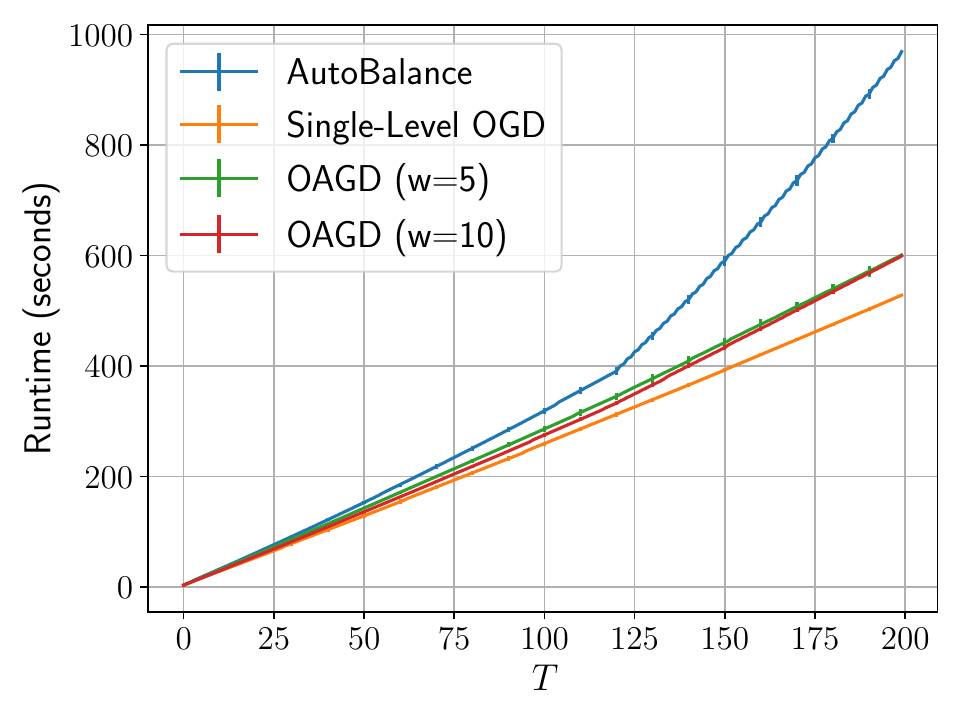}
\includegraphics[scale=0.34]{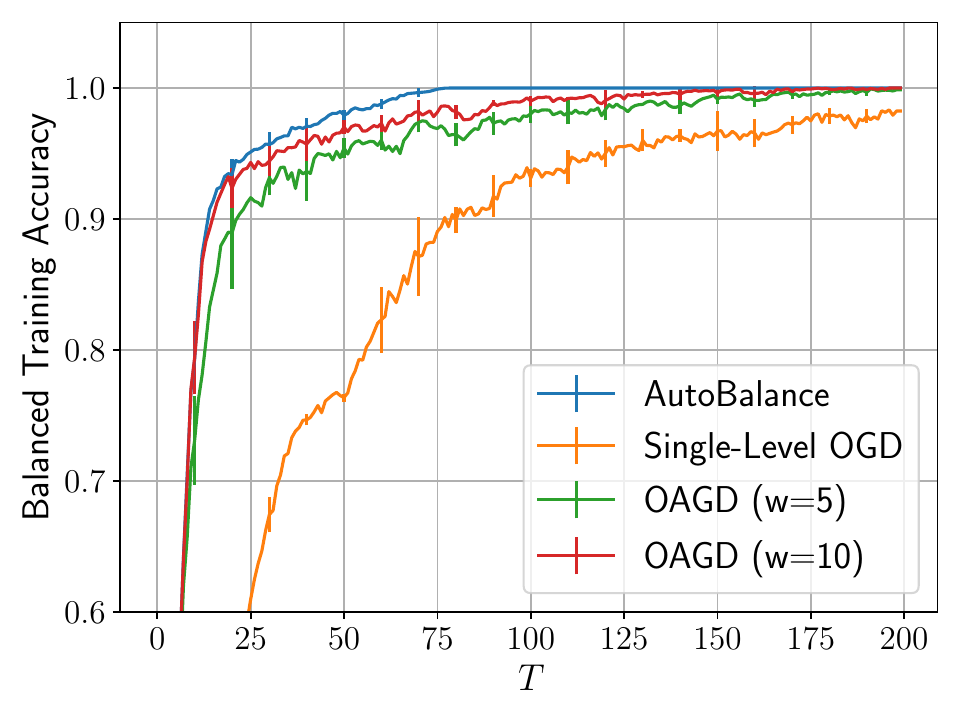} 
\includegraphics[scale=0.34]{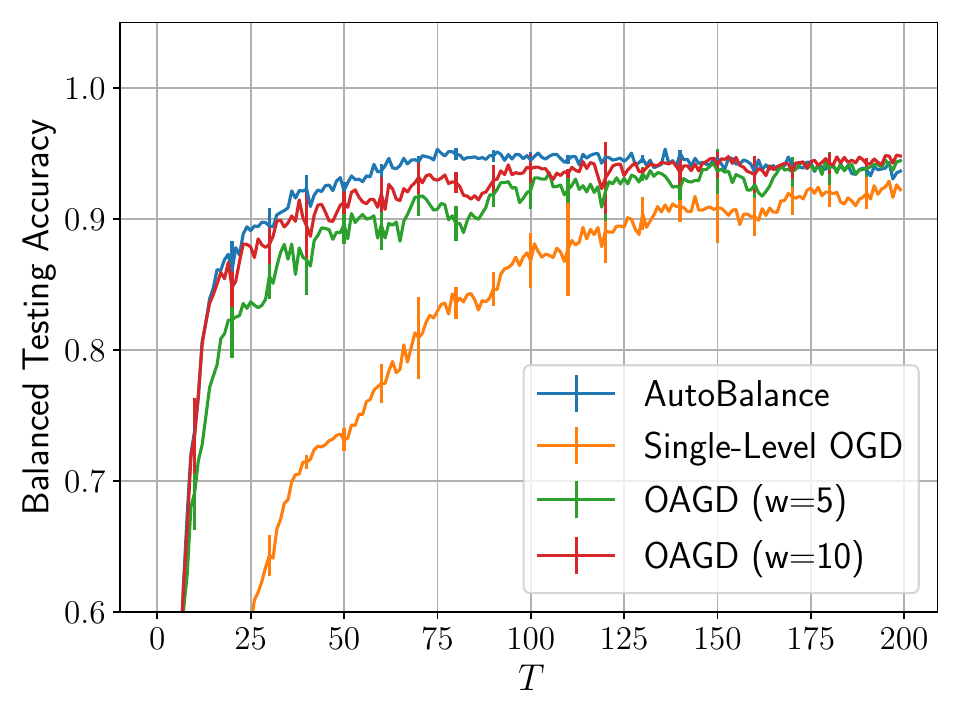}
\vspace{-.25cm}
\caption{\label{fig:imbalance}Performance comparison (mean$\pm$std) on loss tuning for imbalanced MNIST data across five runs. 
}  
\end{center}
\vspace{-4pt}
\end{figure*}
\begin{figure*}
\begin{center}
\includegraphics[scale=0.34]{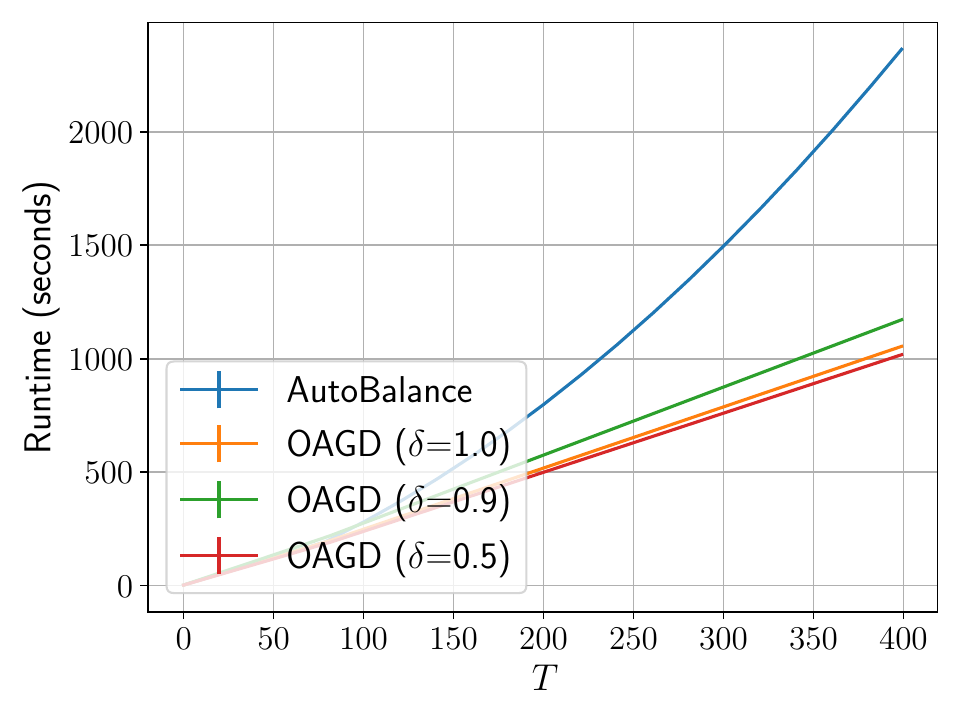}
\includegraphics[scale=0.34]{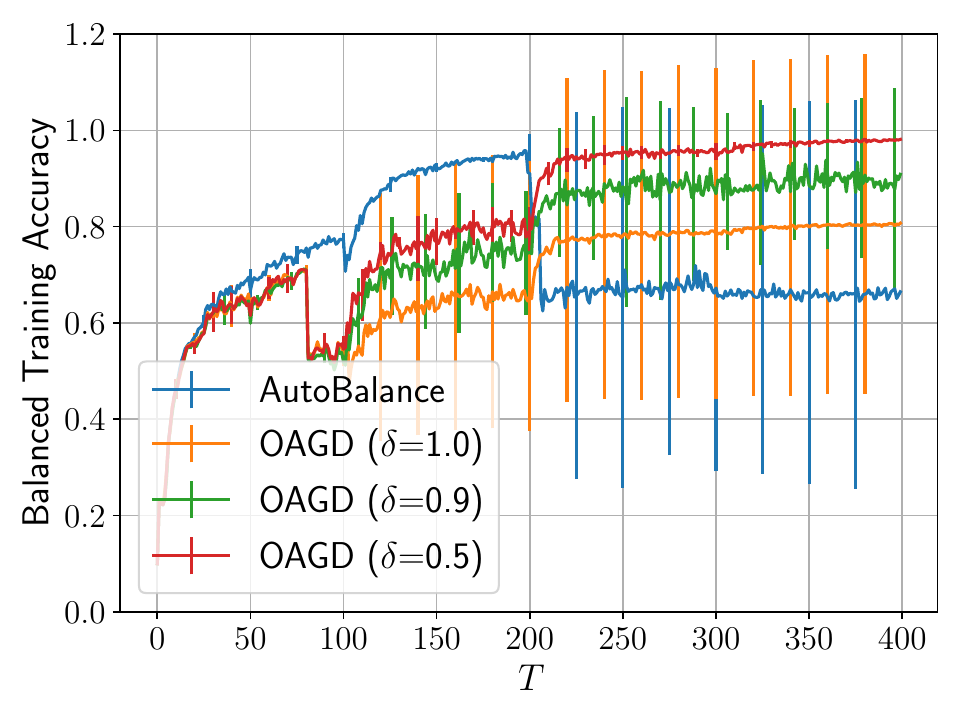} 
\includegraphics[scale=0.34]{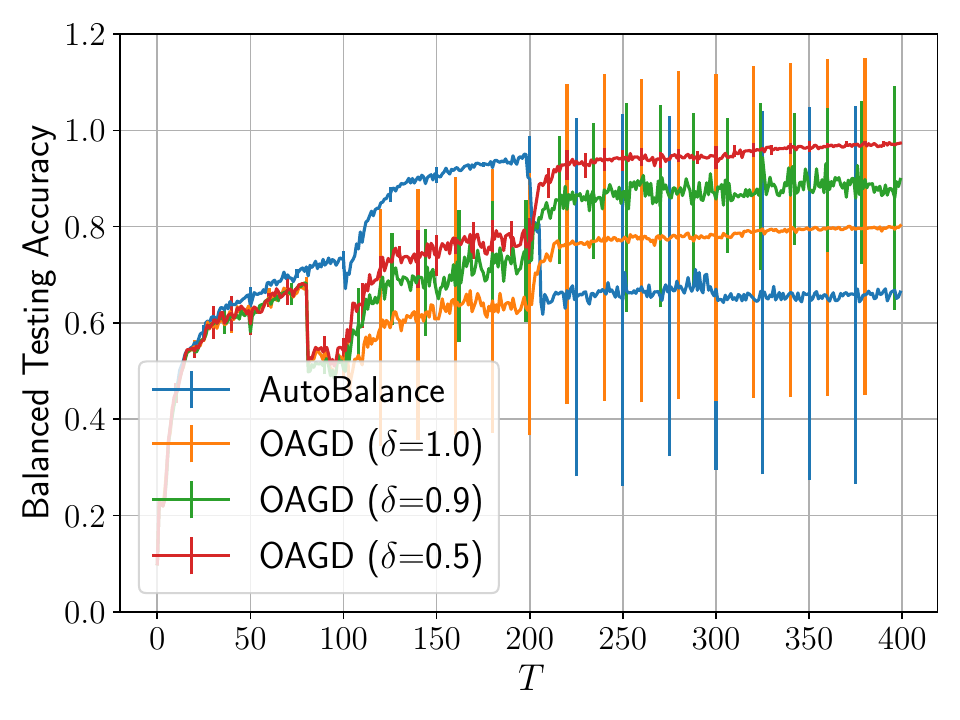}
\vspace{-.25cm}
\caption{Performance comparison (mean$\pm$std) on  loss tuning for imbalanced MNIST data across five runs, considering induced \textbf{distribution shift}. 
}
\label{fig:imbalance_dist_shift} 
\vspace{-.25cm}
\end{center}
\vspace{-.25cm}
\end{figure*}
\vspace{-.25cm}
Imbalanced datasets are common in modern machine learning, posing challenges in generalization and fairness due to underrepresented classes and sensitive attributes. This issue is exacerbated by deep neural networks' tendency to overfit, appearing accurate and fair during training but performing poorly during testing. AutoBalance~\citep{li2021autobalance} addresses this by automatically designing a parametric training loss to balance accuracy and fairness while preventing overfitting. We give an online variant of AutoBalance, demonstrating the enhanced performance of OAGD in this setting.

The bilevel objective function for loss tuning is the same as (\ref{eqn:onl:hyper}) but the leader's and the follower's loss functions are defined differently. At each round or timestep $t$, new samples $(\m{a}_t, b_t) \in \mc{D}_t:= \{\mc{D}^{\text{val}}_t, \mc{D}^{\text{tr}}_t\}$ for all $t \in [T]$ are received, where $\m{a}_t \in \mb{R}^{d_2}$ represents the feature vector and $b_t \in \{1,\ldots, J\}$ represents the corresponding label. For  a new sample $(\m{a}_t, b_t)$, the follower suffers from a \textit{parametric} cross-entropy loss:
\begin{subequations}\label{auto:main:obj}
\begin{equation}
g(\m{x}_{t-1},\m{y}_t; \mc{D}^{\text{tr}}_t)=-\log  \frac{e^{\gamma_{b_t} [\m{y}_t(\m{a}_t)]_{b_t} + \Delta_{b_t}}}{\sum_{j = 1}^{J}e^{\gamma_j [\m{y}_t(\m{a}_t)]_j + \Delta_j}},
\end{equation}
where $\m{x}_{t-1}:=(\Delta_j,\gamma_j)_{j=1}^J$ represents the \textit{logits}. 

In the outer-level, the leader suffers from a balanced cross entropy loss
\begin{equation}
f(\m{y}_t(\m{x}_t); \mc{D}^{\text{val}}_t)=-u_{b_t}\log  \frac{e^{[\m{y}_t(\m{a}_t)]_{b_t}}}{\sum_{j = 1}^{J}   e^{[\m{y}_t(\m{a}_t)]_{j}}}, 
\end{equation}
\end{subequations}
where $u_{j}$ represents the reciprocal of the proportion of samples from the $j$-th class to the total number of samples \citep{li2021autobalance}. 

There might be one notation abuse in \eqref{auto:main:obj} that we need to clarify: $\m{y}_t(\m{x}_t)$ still indicates that the follower $\m{y}_t$ is conditioned on the leader $\m{x}_t$, whereas $[\m{y}_t(\m{a}_t)]_{b_t}$ denotes the predicted logit for class $b_t$ that the follower $\m{y}_t$ makes on sample $\m{a}_t$. Note that the backbone model for $\m{y}_t$ is a 4-layer CNN, resulting in a nonconvex bilevel objective.  For more details, refer to Appendix \ref{app:sec:experiments}.

We compare Algorithm~\ref{alg:obgd} with the following baselines:

\vspace{-2mm}
\begin{enumerate}[wide, labelindent=0pt, itemsep=0pt]
    \item[-]  \textbf{Single-Level OGD}~\citep{zinkevich2003online}: Updates the model $\mathbf{y}_t$ with fixed hyperparameters $\mathbf{x}$ at each timestep on the newly observed data using gradient descent.
    \item[-]  \textbf{AutoBalance}~\citep{li2021autobalance}: An offline bilevel gradient descent framework that updates hyperparameters $\mathbf{x}_t$ and the model $\mathbf{y}_t$ to address imbalance issues.
\end{enumerate}
\vspace{-2mm}
\begin{figure*}[t]
\begin{center}
\includegraphics[scale=0.34]{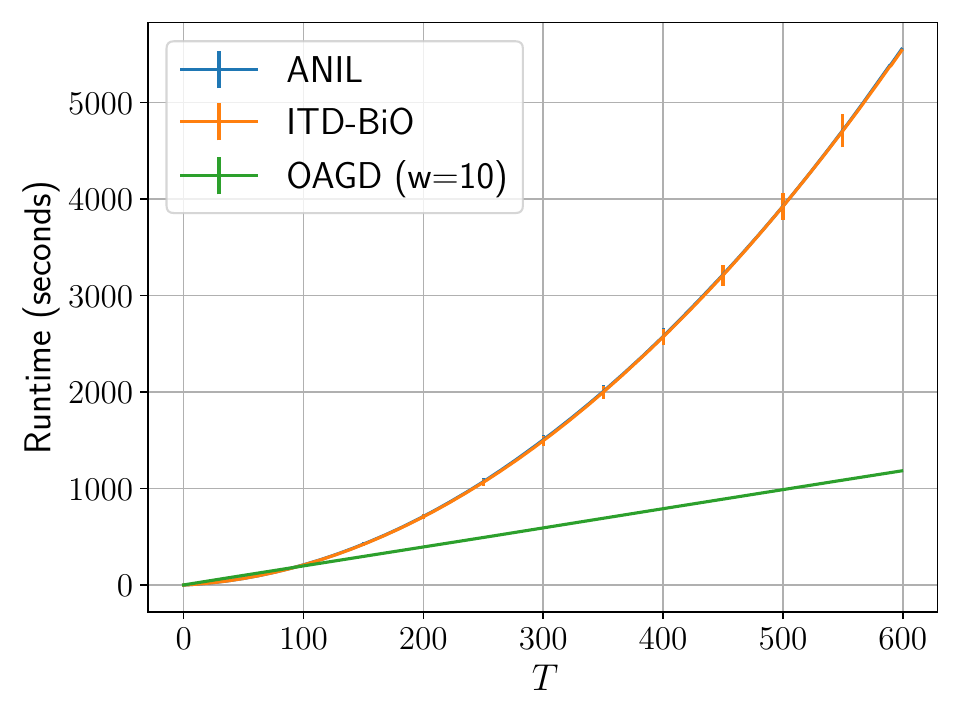}
\includegraphics[scale=0.34]{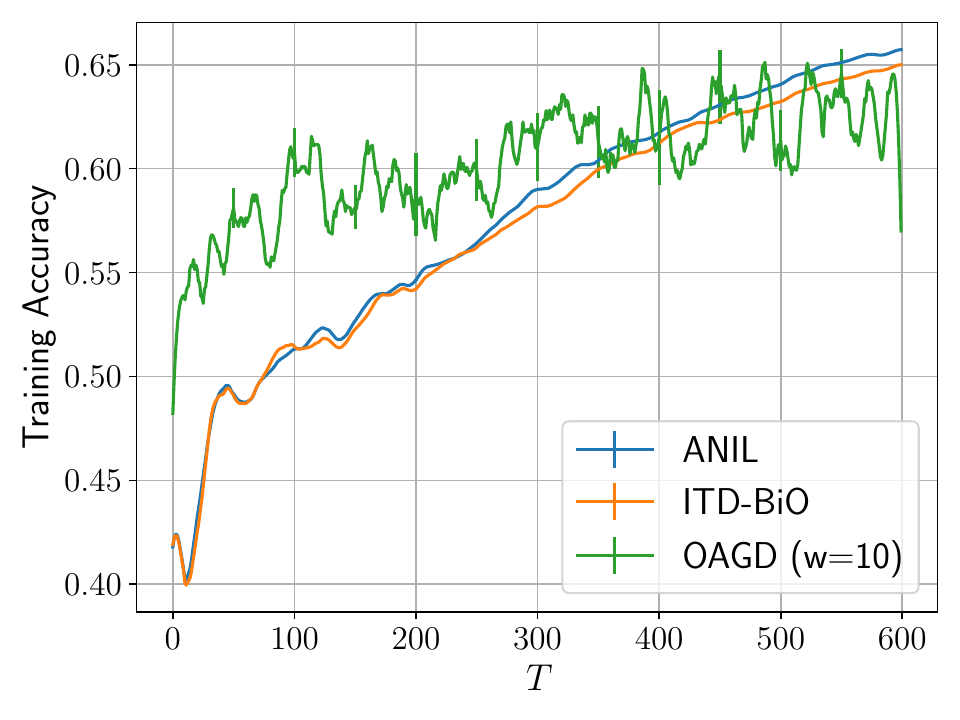} 
\includegraphics[scale=0.34]{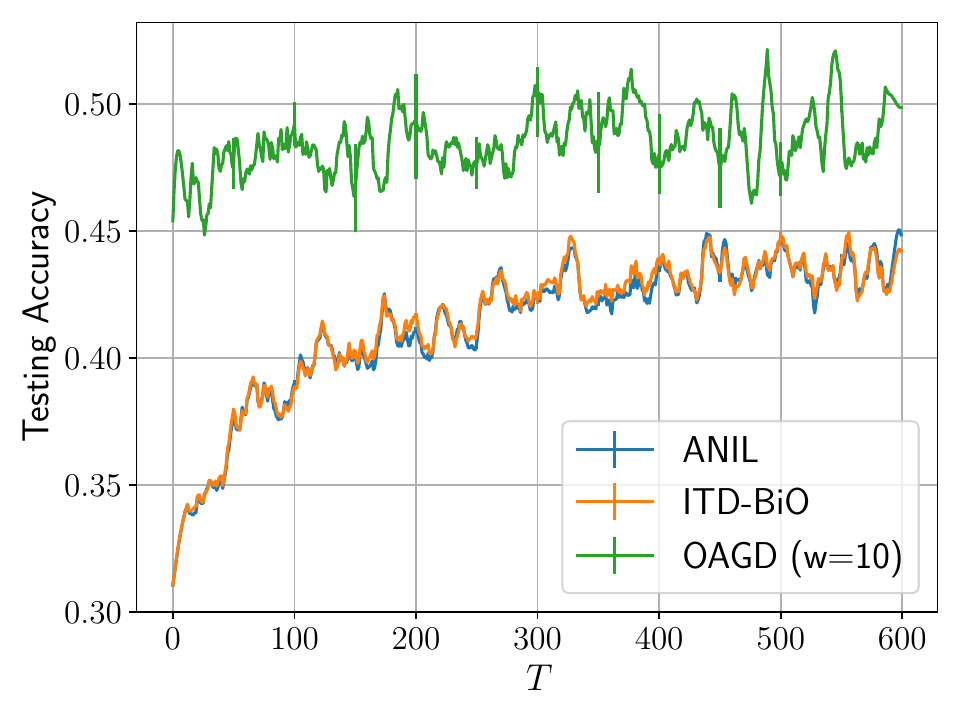}
\vspace{-.25cm}
\caption{\label{fig:meta learning}Performance comparison (mean$\pm$std) on meta-learning for FC100 data across five runs. 
}
\end{center}
\vspace{-.3cm}
\end{figure*}
\vspace{-.2cm}

We conducted experiments using the MNIST dataset~\citep{lecun2010mnist}. To create an imbalanced scenario, we selected samples in proportions of $0.6^i$ from each class ($i=0,1, \ldots, 9$). For online learning, we used a batch size of 128 at each timestep to train our OAGD. If the window size $w$ exceeded 1, we combined the current batch with the previous $w-1$ batches for OAGD training.  
We evaluated cumulative runtime, along with balanced training and testing accuracy, where balanced accuracy is the class-specific average accuracy: $\frac{1}{J}\sum_{j=1}^{J}{\mathbb{P}_{\m{a}_t\sim\mathcal{D}_j}[\argmax_i([\m{y}_t(\m{a}_t)]_i)=j]}$ where $\mathcal{D}_j$ refers to the distribution over samples whose groundtruth class label is $j$ and $\mathbb{P}[A]$ means the probability of event $A$ \citep{li2021autobalance}.
To ensure consistency, we maintained a fixed inner-level learning rate of $\beta=0.1$ for all bilevel algorithms and single-level OGD, with the outer-level learning rate set at $\alpha=0.001$.

Figure~\ref{fig:imbalance} (left) provides runtime comparisons. The single-level OGD algorithm is the fastest since it lacks an outer-level training step and trains on a single batch of data at each timestep. Our OAGD exhibits similar runtime characteristics, with the runtime increasing as the window size $w$ grows due to more extensive training. In contrast, AutoBalance is the slowest method as it trains on all observed data up to each timestep.

Figure~\ref{fig:imbalance} (middle and right) displays balanced training and testing accuracy. AutoBalance quickly achieves high accuracy  after 10 timesteps, whereas single-level OGD exhibits slower improvement. OAGD ($w=5$) and OAGD ($w=10$) exhibit rapid growth in both testing and training accuracy, eventually outperforming AutoBalance. They benefit from time-smoothing hypergradients. Larger window sizes further enhance OAGD's balanced training and testing accuracy.

We conducted experiments on the MNIST dataset to evaluate performance under time-varying distribution shifts across 400 timesteps, divided into four phases of varying distributions, each lasting 100 timesteps. Initially, the distribution was highly imbalanced, with class proportions set by $0.4^i$ for classes $i=0$ to $9$. This imbalance gradually lessened over the next two phases, changing from $0.6^i$ to $0.8^i$. The final 100 timesteps featured a balanced distribution across the 10 MNIST classes, ensuring normalized class proportions throughout all phases.

The parameter $\delta$ influences the weighting of each window in the ``time-averaged hypergradient,'' per Definition \ref{def:avg:grad}. Setting $u_i = 1$ computes the average of online hypergradients in a sliding window. For $u_i = \delta^i$ with $\delta \in (0,1)$, it weights recent values more through an exponential average. Experiments with $\delta$ values of 1, 0.9, and 0.5 show that lower $\delta$ gives more emphasis to recent windows.

The results, illustrated in Figure~\ref{fig:imbalance_dist_shift}, show a significant performance decrease for AutoBalance at timestep $T=200$, while OAGD remains comparatively stable against distribution shifts. At $T=80$, OAGD experiences a noticeable drop due to shifting from a single-level method with fixed hyperparameters (warm-up phase) to a bilevel framework that starts adjusting hyperparameters at $T=80$, causing a short disruption. This warm-up phase is standard in bilevel optimization \citep{li2021autobalance,lorraine2020optimizing}. Although larger $w$ are beneficial for gradient accuracy, they can produce outdated gradients if data distribution changes, negatively affecting model updates. In contrast, a smaller $\delta$ can improve adaptability and performance by focusing on recent data, even with large window sizes $w$.

\subsection{Online Meta-Learning}

Meta-learning aims to bootstrap from a set of given tasks to learn faster on future tasks~\citep{finn2017model,balcan2019provable}. A popular formulation is online meta-learning (OML) where agents sequentially face tasks and apply methods such as classical FPL ~\citep{finn2019online} or mirror descent  \citep{denevi2019online} for enhanced meta-learning. We consider an implicit gradient-based OML setting: for each task $\mc{T}_t$ and some  $\beta>0$, the follower adapts the leader's model $\m{w}_t \in \mc{W} \subset \mb{R}^{d}$ using training data $\mc{D}^{\text{tr}}_t$ and inner OGD:
\begin{align*}
\m{u}^*_t(\m{w}_t)\in \argmin_{\m{u}\in \mb{R}^{d}}\left\langle \m{u}, \nabla f (\m{w}_t; \mc{D}^{tr}_t)\right\rangle + \frac{1}{2\beta} \left\| \m{u}-\m{w}_t\right\|^2. 
\end{align*}
 Then, the test data $\mc{D}^{\text{ts}}_t$ will be revealed to the leader for evaluating the performance of the follower's model $\m{u}^*_t(\m{w}_t)$. The loss observed at this timestep, denoted as $f(\m{u}^*_t(\m{w}_t);\mc{D}^{\textnormal{ts}}_t)$, can then be fed into the leader's algorithm (outer OGD) to update $\m{w}_t$. 
Despite the convex nature of the loss function, which is a cross-entropy loss, where $\m{w}$ represents a 4-layer CNN, ultimately rendering the outer problem non-convex. Further, the inner problem for loss tuning involves training a CNN and is non-convex, showing our implementation’s wide scope.
We compare our {OAGD} with the following meta-learning methods:
\vspace{-.1cm}
\begin{enumerate}[wide, labelindent=0pt, itemsep=0pt]
    \item[-] \textbf{ANIL}~\citep{raghu2019rapid}: A widely used meta-learning algorithm, which simplifies MAML by removing the inner loop for all parts of the MAML-trained network except for the task-specific head.
    \item[-] \textbf{ITD-BiO}~\citep{ji2021bilevel}: A gradient-based stochastic bilevel optimization framework based on iterative differentiation (ITD).
    \vspace{-.1cm}
\end{enumerate}
\vspace{-.2cm}
We evaluate our model on the FC100 (Fewshot-CIFAR100) dataset~\citep{oreshkin2018tadam} through a 5-way 5-shot task.
 In the online setting, each timestep presents one task with 25 training and 25 testing samples. If the window size exceeds 1, data from up to $w-1$ previous tasks are included. In contrast, the offline setting allows baselines to use all observed data up to the current timestep. The inner and outer learning rates for ANIL and ITD-BiO are $0.01$ and $5e-5$, respectively, while our OAGD employs learning rates of 0.1 and $1e-4$. These experiments were conducted on a P100 GPU equipped with 12 GB of memory.

In Figure~\ref{fig:meta learning}, we provide the performance comparison in terms of runtime, training accuracy, and testing accuracy. We compare only our OAGD ($w=10$) to other baselines to enhance the precision of the figure. For the sensitivity analysis concerning the window size, please refer to the Appendix. From the left figure in Figure~\ref{fig:meta learning}, it's noticeable that the two baselines consume a similar longer time with an exponential trend, while our OAGD requires the least time, following a linear trend. However, as indicated by the middle and right figures, our OAGD demonstrates competitive training accuracy and even better testing accuracy across all the timesteps, highlighting its superiority.

\vspace{-.2cm}
\section{Conclusion}\label{sec:conc}
This paper studies online bilevel optimization and provides regret guarantees under different convexity assumptions on the time-varying objective functions. In particular, we propose a new class of online bilevel algorithms capable of leveraging smoothness and providing regret bound in terms of problem-dependent quantities, such as the path-length of the comparator sequence. 

\subsubsection*{Acknowledgements}
This work was supported in part by NSF CAREER award CCF1845076, AFOSR YIP award FA9550-19-1-0026, ARO YIP award W911NF1910027, and NIH grants  U01 AG066833, U01 AG068057, and RF1 AG063481.
\bibliographystyle{abbrvnat}
\bibliography{ref_obo.bib}
%
\onecolumn
\appendix
 \aistatstitle{Supplementary Materials for \hspace{10cm} 
 Online Bilevel Optimization: Regret Analysis of Online Alternating Gradient Methods}
\textbf{Roadmap.} The appendix is organized as follows:
\begin{itemize}
    \item Appendix \ref{sec:appendix-tech-lemma} provides some preliminaries on online optimization and a summary of notations used in the appendix.
    \item Appendix \ref{sec:app:rleated work} discusses additional related work on online single-level optimization and offline bilevel optimization.
    \item Appendix \ref{app:addend:sec:algplusreg} gives an addendum to Section \ref{sec:algplusreg}:
\begin{itemize}
    \item Appendix \ref{sec:app:reg:lower} provides the lower bound of OBO (proof of Theorem \ref{thm:lower:bound}).
    \item Appendix \ref{sec:app:reg:partial:strong} gives the proof for strongly convex OBO with partial information in both dynamic (proof of Theorem~\ref{thm:dynamic:strong:B-OGD}) and static (proof of Theorem~\ref{thm:static:strong:B-OGD}) settings.
    \item Appendix \ref{sec:app:reg:partial} gives the proof for convex OBO with partial information in both dynamic (proof of Theorem \ref{thm:dynamic:convex:B-OGD}) and static (proof of Theorem \ref{itm:thm:s:convex}) settings.
    \item Appendix \ref{sec:app:nonconv} provides the proof for non-convex OBO with partial information (proof of Theorem~\ref{thm:dynamic:nonc}).
    \end{itemize}
    \item Appendix \ref{app:sec:experiments} details the implementation and includes additional experiments:
\begin{itemize}
  \item Appendix \ref{app:sec:experiments:dynamic:regression} gives details on hyperparameters learning for dynamic regression.
    \item Appendix \ref{app:sec:detials:autobalance} gives details on online parametric loss tuning experiments as well as additional experiments.
    \item Appendix \ref{app:sec:detials:oml} provides details of online meta-learning experiments as well as additional experiments.
    \item Appendix \ref{app:sec:sensetivity} presents the numerical sensitivity of algorithms to window size and learning rate.
\end{itemize}
\vspace{8cm}
\end{itemize}
\begin{table}[htbp]
\begin{center}
    {\caption{Summary of the Notations}~\label{tab:notation}}
\scalebox{1}{
\begin{tabular}{|l l |}
		\hline
		\bfseries Notation & \bfseries Description 
		\\
		\hline
$t$ &   Time (round) index \\
$K_t$ &  The number of inner iterations at each round $t$\\
$T$ & The total number of rounds \\
$\alpha$&  Outer stepsize\\
$\beta$ &Inner stepsize\\
	 $\m{x}_t$ & Leader's  decision at round $t$ \\
		 $f_t$ & Leader's  objective at round $t$ \\
		$\m{y}_t$&  Follower's  decision at round $t$\\
		$g_t$&  Follower's  objective at round $t$\\
		 $\m{x}_t^*$ & Leader's  optimal decision in the dynamic setting at round $t$: $\m{x}_t^*\in \argmin_{\m{x}\in\mc{X}}f_t(\m{x})$ \\
	 $\m{x}^*$ & Leader's  optimal decision in static setting: $\m{x}^* \in\argmin_{\m{x} \in \mc{X}} \sum_{t=1}^T f_t(\m{x})$ \\
		$\m{y}_t^* (\m{x})$&  Follower's optimal decision at round $t$ for a given $\m{x}$\\
$\nabla h_t$, $\nabla_{\m{xy}}^2 h_t$, $\nabla_{\m{y}}^2 h_t$ & Gradient, Jacobian, and Hessian of $h_t$\\
$\tilde{\nabla} f_t(\m{x},\m{y})$ &  An approximation of the hypergradient $\nabla f_t (\m{x},\m{y})$\\
$W$, $w$&  $W = \sum_{i=0}^{w-1} u_i$ for $\{u_i\}_{i=0}^{w-1}$ with $ 1=u_{0}\geq u_1 \ldots u_{w-1} >0$ and window size $w \in [T]$\\
$\tilde{\nabla} F_{t,\m{u}}(\m{x}, \m{y})$ &   Approximate time-averaged  hypergradient:  $\tilde{\nabla} F_{t,\m{u}}(\m{x}, \m{y})=\frac{1}{W} \sum_{i=0}^{w-1} u_{i} \tilde{\nabla}  f_{t-i}(\m{x}, \m{y})$ \\
$ \nabla F_{t,\m{u}} (\m{x}_{t}, \m{y}_{t}^*(\m{x}_{t}))$  &  Exact time-averaged hypergradient: $\nabla F_{t,\m{u}} (\m{x}_{t}, \m{y}_{t}^*(\m{x}_{t}))=  \frac{1}{W} \sum_{i=0}^{w-1} u_{i}  \nabla f_{t-i}(\m{x}_{t},\m{y}^*_{t}(\m{x}_{t}))$ \\
$\norm{\cdot}$ &   The Euclidean norm 
\\
$\mb{E}\left[x \right]$ & Expectation of the random variable $x$
\\
$D$& The  (2-norm) diameter of $\mc{X}$: $D=\max_{\m{x}, \m{x}' \in \mc{X}}\|\m{x} -\m{x}'\|$ \\
$D'$& Upper bound on follower's initialization: $\|\m{y}_1 -\m{y}_1^*(\m{x}_1)\| \leq D'$
\\
 $M$ & Upper bound on the outer function:    $|f_t| \leq M$\\
		$M_f$ & Difference between $\tilde{\nabla} f_t(\m{x},\m{y}_t)$ and  $\nabla f_t (\m{x},\m{y}^*_t(\m{x}))$ w.r.t.~$\| \m{y}^*_t(\m{x})-\m{y}_t\|$
		\\
		$L_\m{y}$ & Lipschitz constant of $\m{y}^*_t(\m{x})$
		\\
		$L_f$ & Lipschitz constant of $\nabla f_t(\m{x})$\\
		 $F_{T}$ & Outer function value at the optimum:  $\sum_{t=1}^{T} f_t(\m{x}^*_{t}, \m{y}^*_t(\m{x}^*_{t}))$\\
		$P_{p,T}$ & Path-length of the outer minimizers:   $\sum_{t=2}^{T} \|\m{x}^*_{t-1} - \m{x}^*_{t}\|^p$
		\\
		$Y_{p,T}$ & Path-length of the inner minimizers: $\sum_{t=2}^{T}\norm{\m{y}_{t-1}^*(\m{x}^*_{t-1}) - \m{y}_{t}^*(\m{x}^*_t)}^p$ 
  \\
 $S_{p,T}$& The summation of the inner and outer path-lengths as $S_{p,T}=P_{p,T} + Y_{p,T}$
  \\
  $\bar{Y}_{p,T}$& The static variant of $Y_{p,T}$: $  \sum_{t=2}^{T}\norm{\m{y}_{t-1}^*(\m{x}^*) - \m{y}_{t}^*(\m{x}^*)}^p$\\
		$H_{p,T}$  & Inner minimizer function variation:  $ \sum_{t=2}^{T} \sup_{\m{x}\in \mb{R}^{d_1}} \| \m{y}^*_{t-1}(\m{x}) - \m{y}^*_{t}(\m{x})\|^p$
		\\
	    $V_T$ & Online functions variation:  $\sum_{t=2}^{T} \sup_{\m{x}\in \mc{X}} |f_{t-1}(\m{x}) -  f_t(\m{x})|$ \\
	    $G_T$ & Online gradients variation: $\sum_{t=2}^{T} \sup_{\m{x}\in \mc{X}} \norm{\nabla f_{t-1}(\m{x}) - \nabla f_t(\m{x})}^2$ 
		\\
		$\textnormal{D-Reg}_T$ & (single-level) dynamic regret: $\sum_{t=1}^T f_t(\m{x}_t)-\sum_{t=1}^T f_t(\m{x}^*_{t})$ \\
		 $\textnormal{S-Reg}_T$& (single-level) static regret: $ \sum_{t=1}^T f_t(\m{x}_t)- \min_{\m{x} \in \mc{X}}\sum_{t=1}^T f_t(\m{x})$\\
		$\textnormal{L-Reg}_T$ & (single-level) local regret: $\sum_{t=1}^T  \left\| \nabla  F_{t,\m{u}} (\m{x}_{t}) \right\|^2$\\		
		$\textnormal{BD-Reg}_T$ & Bilevel dynamic regret: $\sum_{t=1}^T f_t(\m{x}_t, \m{y}_t^*(\m{x}_t))-\sum_{t=1}^T f_t(\m{x}^*_{t},  \m{y}^*_t(\m{x}^*_{t}))$ \\
		 $\textnormal{BS-Reg}_T$&  Bilevel (outer) static regret: $ \sum_{t=1}^T f_t(\m{x}_t, \m{y}_t^*(\m{x}_t))- \min_{\m{x} \in \mc{X}}\sum_{t=1}^T f_t(\m{x},  \m{y}^*_t(\m{x}))$\\
		$\textnormal{BL-Reg}_T$ & Bilevel local regret: $\sum_{t=1}^T  \left\| \nabla  F_{t,\m{u}} (\m{x}_{t}, \m{y}_{t}^*(\m{x}_{t})) \right\|^2$\\
		\hline
	\end{tabular}
 }
\end{center}
 \end{table}

\section{Addendum to Section~\ref{sec:intro}: Preliminaries and Notations}\label{sec:appendix-tech-lemma}

We provide several technical lemmas used in the proofs. We start by assembling some well-known facts about convex and smooth functions.
\begin{enumerate}[label={\textbf{(F\arabic*})}, wide, labelindent=0pt,itemsep=0pt]
    \item  (\textbf{Smoothness}): Suppose $f(\m{x})$ is $L$-smooth for some constant $L$. Then, by definition, the following inequalities hold for any two points $\m{x},\m{y}\in\mathbb{R}^d$:
    \begin{subequations}
\begin{equation*}
    \Vert \nabla f(\m{x})- \nabla f(\m{y}) \Vert \leq L \Vert \m{x}-\m{y} \Vert, \,
\end{equation*}
\begin{equation*}
    f(\m{y})-f(\m{x}) \leq \langle \m{y}-\m{x}, \nabla f(\m{x}) \rangle +\frac{L}{2}{\Vert \m{y}-\m{x} \Vert}^2.
\label{eqn:smooth2}
\end{equation*}
Further,  if  $\m{x}^*\in\argmin_{\m{x}\in\mb{R}^{d}} f(\m{x})$, then
\begin{equation*}
{\Vert \nabla f(\m{y}) \Vert}^2 \leq 2L (f(\m{y})-f(\m{x}^*)).
\end{equation*}
\end{subequations}
\item (\textbf{Smoothness and Convexity}): Suppose $f(\m{x})$ is convex and $L$-smooth  for some constant $L$. Then, the following holds for any two points $\m{x,y} \in \mathbb{R}^{d}$:
\begin{equation*}
    \langle \nabla f(\m{y}) - \nabla f(\m{x}), \m{y}-\m{x} \rangle \geq \frac{1}{L} {\Vert \nabla f(\m{y}) - \nabla f(\m{x})\Vert}^2.
\end{equation*}
\item (\textbf{Strong-Convexity}): Suppose $f(\m{x})$ is $\mu$-strongly convex for some positive constant $\mu$. Then, by definition, the following inequality holds for any two points $\m{x,y}\in\mathbb{R}^d$:
\begin{subequations}
\begin{equation*}
    f(\m{y})-f(\m{x}) \geq \langle \m{y}-\m{x}, \nabla f(\m{x}) \rangle +\frac{\mu}{2}{\Vert \m{y}-\m{x} \Vert}^2.
\label{eqn:strconv1}
\end{equation*}
Using the above inequality, one can conclude that
\begin{equation*}
    \langle \nabla f(\m{y}) - \nabla f(\m{x}), \m{y}-\m{x} \rangle \geq \mu {\Vert \m{y}-\m{x} \Vert}^2.
\label{eqn:strconv2}
\end{equation*}
\end{subequations}
\end{enumerate}
%
The following lemma provides the self-bounding property of smooth functions.
\begin{lem}\label{lem:nonneg:smooth}
\cite[Lemma 3.1]{srebro2010smoothness}
For a non-negative and $L$--smooth function $f: \mc{X} \rightarrow \mb{R}$, we have
\[
\norm{\nabla f(\m{x})} \leq \sqrt{4 L f(\m{x})}, \ \forall \m{x} \in \mc{X}.
\]
\end{lem}
\begin{lem}\cite[Theorem 2.1.11]{nesterov2003introductory}\label{lm:nes}
Let $g :
\mathbb{R}^d \rightarrow \mathbb{R}$ be a function that is smooth, $\mu_g$-strongly convex,
and $L_g$-gradient Lipschitz continuous on an open convex set
$\mathcal{Y} \subseteq \mathbb{R}^d$. Let $\m{y}^*$ denote the global minimizer  of  $g$ over $\mathcal{Y}$.
Then, the sequence $\{\m{y}_t\}_{t=1}^T$ generated by the gradient descent
method 
\begin{align*}
 \m{y}_{t+1}=\m{y}_{t}-\beta\nabla g(\m{y}_{t})   
\end{align*}
with stepsize $\beta \in (0,2/(\mu_g+L_g)]$ satisfies
\begin{equation*}
    \|\m{y}_{t+1}-\m{y}^* \|^2\leq \left(1-\frac{2\beta \mu_g L_g}{\mu_g+L_g}\right)\|\m{y}_{t}-\m{y}^* \|^2.
\end{equation*}
If $\beta=2/(\mu_g+L_g)$, then
\begin{equation*}
    \|\m{y}_{t+1}-\m{y}^* \|^2\leq \left(\frac{\kappa_g-1}{\kappa_g+1}\right)^2\|\m{y}_{t}-\m{y}^* \|^2,
\end{equation*}
where $\kappa_g=L_g/\mu_g$.
\end{lem}
\begin{lem}\label{lem:Jens}
For any set of vectors $\{\m{x}_i\}_{i=1}^m$ with $\m{x}_i\in\mathbb{R}^d$, we have
 \begin{equation*}
    \left\| \sum\limits_{i=1}^{m} \m{x}_i\right\|^2 \leq m \sum\limits_{i=1}^{m} {\Vert \m{x}_i \Vert}^2.
 \end{equation*}
 \end{lem}
\begin{lem}\label{lem:rimann}
For all $T \in\mb{N}$,
\begin{enumerate}[label=\textnormal{\Roman*.}]
\item\label{itm:lem:a:rimann}  $  \log(T)+\frac{1}{T} \leq \sum_{t=1}^T\frac{1}{t} \leq \log(T)+1 ;$
\item\label{itm:lem:b:rimann} $ \sum_{t=1}^T\frac{1}{\sqrt{t}} \leq 2\sqrt{T} ;$

\item\label{itm:lem:c:rimann} If $1<s<\infty$, then
$
\xi(s)= \sum_{t=1}^T 1/t^s
$
is called the Riemann $\xi$-function and we have
\begin{align*}
   \xi(2n)=(-1)^{n+1}\frac{(2\pi)^{2n}B_{2n}}{2(2n)!},~~~~n=1,2, 3, \ldots,
\end{align*}
where the coefficients $B_{2n}$ are the Bernoulli numbers.
\end{enumerate}
\end{lem}
\begin{lem}\label{lem:trig}
For any $\m{x},\m{y}, \m{z} \in\mathbb{R}^d$, the following holds for any $c >0:$
\begin{align*}
    {\Vert \m{x}+\m{y} \Vert}^2 &\leq (1+c){\Vert \m{x} \Vert}^2 + \left(1+\frac{1}{c}\right){\Vert \m{y} \Vert}^2,~~\textnormal{and} \\
\norm{\m{x} - \m{y}}^2 &\geq  (1-c)\norm{\m{x}-\m{z}}^2 +\left(1-\frac{1}{c}\right)\norm{\m{z} - \m{y}}^2.  \label{eqn:triang:jens}
\end{align*}
\end{lem}
\begin{lem}\cite[Lemma 2.8]{shalev2011online}\label{lm:cc}
Let $\mathcal{X}\subseteq \mathbb{R}^d$ be a nonempty convex set. Let $f(\m{x}):\mathcal{X}\rightarrow \mathbb{R}$ be a $\mu_f$-strongly convex function over $\mathcal{X}$. Let
$\m{x}^* \in \argmin_{\m{x} \in  \mathcal{X}}\{f(\m{x}) \}$. Then, for any $\m{z}\in \mathcal{X}$, we have
\begin{align*}
f(\m{x}^*)-f(\m{z})\leq -\frac{\mu_f}{2}\|\m{z}- \m{x}^*\|^2.
\end{align*}
\end{lem}

\subsection{On the Comparability of Dynamic Metrics }\label{sec:example}
The following example shows that $P_{p,T}$, $Y_{p,T}$, and $\bar{Y}_{p,T}$ are not comparable in general and all three measures play a key role in OBO.
\begin{exm}\label{exm:not:comparable}
Let $x\in \mathcal{X} = \left[-1,1\right] \subset\mathbb{R}$, $y\in \mathbb{R}$, and consider a sequence of quadratic cost functions 
\begin{align*}
 f_{t}(x,y) &=  \frac{1}{2} \left(x+2{a}_t^{(1)}\right)^2  + \frac{1}{2} \left({y}-a_t^{(2)}\right)^2+a_t^{(3)}, \\  
 g_{t}({x},{y})&=\frac{1}{2} {y}^2- \left({x}-a_t^{(2)}\right) {y}+a_t^{(4)},
\end{align*}
for all $t \in [T]$, where $\{a_t^{(i)}\}_{i=1}^4$ are some time-varying constants.

It follows from \eqref{eqn:obl} that
\begin{align*}
{y}^*_t({x}_{t})={x}_t-a_t^{(2)},  \qquad 
{x}_{t}^*=-a_t^{(1)}+a_t^{(2)}, \qquad 
 {y}^*_t({x}_{t}^*)=a_t^{(1)}.  
\end{align*}
Let $a_t^{(2)}=(-1)^t/\sqrt{t}$ for all $t \in [T]$.
\begin{itemize}
    \item If $a_t^{(1)}=a_t^{(2)}$, then  $P_{1,T} =P_{2,T}=0$, $Y_{1,T} =\bar{Y}_{1,T} =\mc{O}(\sqrt{T})$, and $Y_{2,T} =\bar{Y}_{2,T} =\mc{O}(\log{T})$.  
    \item If $a_t^{(1)}=0$, then $P_{1,T} = \mc{O}(\sqrt{T})$,  $P_{2,T} = \mc{O}(\log{T})$, and $Y_{1,T}=Y_{2,T} =0$.
\end{itemize}
This shows that $S_{1,T}= P_{1,T}+Y_{2,T}$, $S_{2,T}= P_{2,T}+Y_{2,T}$ are not comparable in general. Similarly, static metrics $\bar{Y}_{1,T}$ and $\bar{Y}_{2,T}$  are not comparable.
\end{exm}
\section{Addendum to Section~\ref{sec:related-work}: Additional Related Work}\label{sec:app:rleated work}
Online learning and stochastic optimization are closely related. The key difference between them is that  at each round $t$ of the online optimization, the loss function can be arbitrarily chosen by the adversary. 
Given the vastness of the online and stochastic optimization literature, we do not strive to provide an exhaustive review. Instead, we mainly focus on a few representative works on online static and \textit{worst-case} dynamic regret minimization, as well as bilevel optimization. Refer to \citep{Hazan16a,hoi2021online} and \citep{liu2021investigating,sinha2017review} for surveys on online and bilevel optimization, respectively.

\paragraph{Static Regret Minimization:} In single-level online optimization, the goal of the player (learner) is to choose a sequence $\{\m{x}_t\}_{t=1}^T$ such that their regret is minimized. There are different notions of regret in the literature, including static, dynamic (defined in \eqref{eqn:worst:dynamic:regret}), and adaptive~\citep{hazan2016introduction,shalev2007online,shalev2011online}. In the case of static regret, $\m{x}^\ast_t$ is replaced by $\m{x}^\ast \in \argmin_{\m{x} \in \mc{X}} \sum_{t=1}^T f_t(\m{x})$. This type of regret is well-studied in the literature of online learning~\citep{hazan2016introduction,shalev2007online,shalev2011online}.  \citet{zinkevich2003online} shows that online gradient descent (OGD) provides an $\mc{O}(\sqrt{T})$ regret bound for convex (possibly nonsmooth) functions. \citet{hazan2007logarithmic} improve this bound to $\mc{O}(\log T)$ for strongly-convex functions. These results were also shown to be minimax optimal~\citep{abernethy2008optimal}.
\citet{zhou2020regret} provide regret bounds for online learning algorithms under relative Lipschitz and/or relative strongly-convexity assumptions.

In addition to exploiting the convexity of online functions, recent studies have focused on improving static regret by incorporating smoothness~\citep{chiang2012online,srebro2010smoothness}. These problem-dependent bounds can safeguard the worst-case minimax rate, yet they can be much better in easy cases of online learning  problems (e.g., loss functions with a small deviation). For instance, \citep{srebro2010smoothness} shows that for convex smooth non-negative functions,  OGD can achieve an $\mc{O}(1+\sqrt{F_T})$ small-loss regret bound, where $F_T= \sum_{t=1}^{T} f_t(\m{x}^*)$ and $\m{x}^\ast\in \argmin_{\m{x} \in \mc{X}} \sum_{t=1}^T f_t(\m{x})$. For convex smooth functions, \citep{chiang2012online} establishes an $\mc{O}(1+\sqrt{G_T})$ bound, where $G_T = \sum_{t=2}^{T} \sup_{\m{x}\in \mc{X}} \norm{\nabla f_{t-1}(\m{x}) - \nabla f_t(\m{x})}^2$ is the gradient variation. These bounds are particularly favored in slowly changing environments in which the online functions evolve gradually~\citep{zhao2020dynamic}.

\paragraph{Dynamic Regret Minimization:} \textit{Single-level} dynamic regret forces the player to compete with time-varying comparators, and thus is particularly favored in non-stationary environments~\citep{sugiyama2012machine}.  The notion of dynamic regret is also referred to as tracking regret or shifting regret in the prediction with expert advice setting~\citep{bousquet2002tracking,herbster1998tracking,herbster2001tracking,wei2016tracking,zheng2019equipping}.  There are two kinds of dynamic regret in previous studies:  The \textit{universal dynamic regret} aims to compare with any feasible comparator sequence \citep{zhang2018adaptive,zhao2020dynamic,zinkevich2003online}, while the \textit{worst-case dynamic regret} (defined in \eqref{eqn:worst:dynamic:regret}) specifies the comparator sequence to be the sequence of minimizers of online functions \citep{aydore2019dynamic,besbes2015non,jadbabaie2015online,mokhtari2016online,yang2016tracking,zhang2017improved, nazari2021dynamic}. We present related works for the latter case as it is the setting studied in this paper.

It is known that in the worst case, sublinear dynamic regret is not attainable unless one imposes regularity of some form on the comparator sequence or the function sequence~\citep{besbes2015non,hall2015online,jadbabaie2015online}.  \citet{yang2016tracking} shows that OGD enjoys an $\mc{O}(1+\sqrt{T P_{1,T}})$ worst-case dynamic regret bound for convex functions when the path-length $P_{1,T}$ is known. For strongly convex and smooth functions, \citep{mokhtari2016online} shows that an $\mc{O}(1+P_{1,T})$ dynamic regret bound is achievable. \citet{chang2021online} proves that OGD can achieve an $\mc{O}(1+P_{2,T})$ regret bound without the bounded gradient assumption. \citet{zhang2017improved} further proposes the online multiple gradient descent algorithm and proves that the algorithm enjoys an $\mc{O}(1+\min\{P_{1,T},P_{2,T}\})$ regret bound; this bound has been recently enhanced to $\mc{O}(1+ \min\{P_{1,T},P_{2,T}, V_{T}\})$ by an improved analysis~\citep{zhao2021improved}, where $V_{T} =  \sum_{t=2}^{T} \sup_{\m{x}\in \mc{X}} |f_{t-1}(\m{x}) -  f_t(\m{x})|$. \citet{yang2016tracking} further shows that the $\mc{O}(1+P_{2,T})$ rate is attainable for convex and smooth functions, provided that all the minimizers $\m{x}^*_t$ lie in the interior of the domain $\mc{X}$. The above results use path-length (or squared path-length) as the regularity, which is in terms of the trajectory of the comparator sequence. \citet{nazari2019adaptive,nazari2019dadam} extend the above results to the distributed settings and provide dynamic regret bounds in terms of the $\ell_1$ path-length. \citet{besbes2015non} shows that OGD with a restarting strategy attains an $\mc{O}(1+T^{2/3}{V_{T}}^{ 1/3})$ regret for convex functions when $V_{T}$ is available, which has been recently improved to $\mc{O}(1+T^{1/3}{V_{T}}^{ 2/3})$ for the square loss~\citep{baby2019online}.

\paragraph{Adaptive Regret:} Adaptive regret \citep{daniely2015strongly,hazan2007adaptive, zhang2019adaptive,zhang2020minimizing,zhang2018dynamic} is also used to capture the dynamics in the environment. Specifically,  it characterizes a local version of static regret, where
 $$ \mbox{Regret}_T([r,s]) \triangleq \sum_{t=r}^s f_t(\m{x}_t)-\min_{\m{x}\in\mc{X}}\sum_{t=r}^s f_t(\m{x}),$$
for each interval $[r,s]\subseteq [T]$. \citet{zhang2018dynamic} provide a connection between strongly adaptive regret and dynamic regret and proposes an adaptive algorithm that can bound the dynamic regret without prior knowledge of the functional variation. \citet{zhang2020minimizing} develop a new algorithm that can minimize the dynamic regret and the adaptive regret simultaneously.

\paragraph{Local Regret Minimization:}
Non-convex online optimization is a more challenging setting than the convex case. Some notable works in the non-convex literature include adversarial multi-armed bandit with a continuum of arms \citep{bubeck2008online,heliou2020online,heliou2021zeroth,krichene2015hedge} and classical Follow-the-Perturbed-Leader algorithm with access to an offline non-convex optimization oracle \citep{agarwal2019learning,kleinberg2008multi,suggala2020online}.
\citet{hazan2017efficient} introduces a local regret measure based on gradients of the loss  to address intractable non-convex online models. Their regret is local in the sense that it averages a sliding window of gradients and quantifies the objective of predicting points with small gradients on average. They are motivated by a game-theoretic perspective, where an adversary reveals observations from an unknown static loss. The gradients of the loss functions from the $w$ most recent rounds of play are evaluated at the current model parameters $\m{x}_t$, and these gradients are then averaged. The motivation behind averaging is two-fold: (i) a randomly selected update has a small time-averaged gradient in expectation if an algorithm incurs local regret sublinear in $T$, and (ii) for any online algorithm, an adversarial sequence of loss functions can force the local regret incurred to scale with $T$ as $\mc{O}(T/w^2)$. \citet{hallak2021regret} extends the local regret minimization to online, non-smooth, non-convex problems. These arguments, presented in \citep{aydore2019dynamic,hallak2021regret,hazan2017efficient,nazari2019dadam}, inspire our use of local regret for OBO.


\paragraph{(Offline) Bilevel Optimization:} Since its first formulation by Stackelberg \citep{stackelberg1952theory} and the first mathematical model by Bracken and McGill \citep{bracken1973mathematical}, there has been significant growth in the applications and developments of bilevel programming. Existing works either reduce the problem to a single-level optimization problem \citep{aiyoshi1984solution,al1992global,edmunds1991algorithms,hansen1992new,lv2007penalty,moore2010bilevel,shi2005extended,sinha2017review}, or apply (alternating) optimization methods to solve the original problem.  The single-level formulations, which employ the Karush-Kuhn-Tucker (KKT) conditions or penalty approaches, are generally difficult to solve~\citep{sinha2017review}.

Gradient-based approaches are more attractive for bilevel programming due to their simplicity and effectiveness. This type of approach estimates the hypergradients for iterative updates, and can generally be divided into two categories: approximate implicit differentiation (AID) and iterative differentiation (ITD) classes. ITD-based approaches~\citep{finn2017model,franceschi2017forward,grazzi2020iteration,maclaurin2015gradient}  estimate the hypergradient either in reverse (automatic differentiation) or forward manner. AID-based approaches~\citep{domke2012generic,ghadimi2018approximation,grazzi2020iteration,ji2021bilevel,pedregosa2016hyperparameter,nazari2022penalty} estimate the hypergradient via implicit differentiation.  \citet{franceschi2018bilevel} characterized the asymptotic convergence of a backpropagation-based approach as one of ITD-based algorithms by assuming the inner-level problem is strongly convex.   \citet{shaban2019truncated} provided a similar analysis for a truncated backpropagation scheme. \cite{li2020improved,liu2020generic} analyzed the asymptotic performance of ITD-based approaches when  the inner-level problem is convex.

Finite-time complexity analysis for bilevel optimization has also been explored. \cite{ghadimi2018approximation} provided a finite-time convergence analysis for an AID-based algorithm under various loss geometries: the outer function being strongly convex, convex, or non-convex, while the inner function remains strongly convex. \citet{ji2021bilevel} provided an improved finite-time analysis for both AID- and ITD-based algorithms under the nonconvex-strongly-convex geometry. \citet{Ji2021LowerBA} provided the lower bounds on complexity as well as upper bounds under these two geometries.  When the objective functions can be expressed in an expected or finite-time form, \citep{ghadimi2018approximation,hong2020two,ji2021bilevel} developed stochastic bilevel algorithms and provided the finite-time analysis. There have been subsequent studies on accelerating SGD-type bilevel optimization via momentum and variance reduction techniques~\citep{chen2021single,guo2021stochastic,huang2021biadam,Ji2020ProvablyFA} as well. However, a fundamental assumption in all the aforementioned works is that the cost function does \textit{not} change throughout the horizon over which we seek to optimize it. 

\section{ Addendum to  Section~\ref{sec:algplusreg}: Proof of Main Theorems}\label{app:addend:sec:algplusreg}

\subsection{Proof of Theorem~\ref{thm:lower:bound}}\label{sec:app:reg:lower}
\begin{proof}
We randomly generate a sequence of functions $\{(f_t,g_t)\}_{t=1}^T$ and show that there exists a distribution of online functions such that for any bilevel algorithm $\mc{A}$, we have $\mb{E}\left[\textnormal{BD-Reg}_T\right] \geq \mb{E}[S_{2,T} ]$. Specifically, for any bilevel algorithm $\mc{A}$ that generates a sequence of $( {x}_t,{y}_t^*)\in \mb{R} \times \mb{R}$ for all $t \in[T]$, we consider
the expected regret as follows:
\[
\mb{E}\left[\textnormal{BD-Reg}_T\right] = \mb{E}\left[\sum_{t=1}^T f_t({x}_t, {y}_t^*({x}_t))-\sum_{t=1}^T f_t({x}^*_{t},  {y}^*_t({x}^*_{t}))\right].
\]
For each round $t$, we randomly sample ${a}_t^{(1)},{a}_t^{(2)}\in \mathbb{R}$ from the Gaussian distribution $\mc{N}(0, 1)$. For all $t\in[T] $, let
\begin{align*}
  f_{t}({x},{y}_{t}^*({x})) &= 6 \left(            {y}_{t}^*({x})-\left({a}_t^{(1)}+{a}_t^{(2)}\right)\right)^2+ 6 \left({x}-{a}_t^{(1)}\right)^2, \\
\textnormal{s.t.}~~~ {y}_{t}^*({x})&\in \argmin_{{y} \in \mb{R}}g_{t}({x},{y})=\frac{1}{2} {y}^2- \left({x}+ {a}_t^{(2)}\right) {y}.
\end{align*}
It follows from \eqref{eqn:obl} that
\begin{equation*}
{y}^*_t({x}_{t})={x}_t+{a}_t^{(2)},  \quad {x}_{t}^*={a}_t^{(1)}, \quad  \textnormal{and} \quad {y}^*_t({x}_{t}^*)={a}_t^{(1)}+{a}_t^{(2)}.
\end{equation*}
Notice that ${x}_t$ is independent from ${a}_t^{(1)}$.
Hence,
\begin{equation}\label{eqn:reg:lower}
\begin{split}
\mb{E}\left[\textnormal{BD-Reg}_T\right]  
&=  6 \sum_{t=1}^T  \mb{E}\left[\left( {x}_t -{a}_t^{(1)}\right)^2 \right]\\
& =  6 \sum_{t=1}^T \left(  \mb{E}\left[\left( {a}_t^{(1)} \right)^2 \right] +\mb{E}\left[ {x}_t^2\right] \right) \\
&\geq 6T.
\end{split}
\end{equation}

For $T\geq 2$, we obtain
\begin{equation}\label{eqn:py:upper}
\begin{split}
\mb{E}\left[S_{2,T}\right] =\mb{E}[P_{2,T}]+\mb{E}[Y_{2,T}]&=\sum_{t=2}^{T}\mb{E}\left[\left({x}^*_{t-1} - {x}^*_{t}\right)^2\right]+\sum_{t=2}^{T}\mb{E}\left[\left({y}_{t-1}^*({x}^*_{t-1}) - {y}_{t}^*({x}^*_t)\right)^2\right]
 \\&= \sum_{t=2}^T \mb{E} \left[ \left({a}_t^{(1)}- {a}_{t-1}^{(1)}\right)^2 \right] + \sum_{t=2}^T \mb{E} \left[ \left({a}_t^{(1)}+ {a}_t^{(2)}- ({a}_{t-1}^{(1)}+{a}_{t-1}^{(2)})\right)^2 \right]\\
&=\sum_{t=2}^T 2 \left(\mb{E} \left[ \left({a}_t^{(1)}\right)^2\right]+ \mb{E} \left[ \left({a}_{t-1}^{(1)}\right)^2\right]\right)+ \mb{E} \left[ \left({a}_{t}^{(2)}\right)^2 \right]+\mb{E} \left[ \left({a}_{t-1}^{(2)}\right)^2\right]\\
& \leq 6 (T-1).
\end{split}
\end{equation}
Here, the third equality follows from the independence of ${a}^{(1)}_t$ and ${a}^{(2)}_t$ for all $t \in [T]$.

Now, it follows from \eqref{eqn:reg:lower} and  \eqref{eqn:py:upper} that $\mb{E}\left[\textnormal{BD-Reg}_T\right] \geq \mb{E}[S_{2,T} ]$. This completes the proof of Theorem ~\ref{thm:lower:bound}.
\end{proof}

\subsection{Proof for Strongly Convex OBO with Partial Information}\label{sec:app:reg:partial:strong}

In this section, we provide the dynamic regret bound for strongly convex OBO with partial information. Specifically, we derive a problem-dependent regret bound for Algorithm~\ref{alg:obgd}.  

\subsubsection{Auxiliary Lemmas}\label{C21}
\begin{lem}[Restatement of Lemma~\ref{lem:lips}]
Under Assumption~\ref{assu:f}, for all $t \in [T]$, $\m{x}, \m{x}' \in \mc{X}$,  and $\m{y} \in \mb{R}^{d_2}$, we have
\begin{align}
\label{eqn:lip:cons1}&\left\|\m{y}^*_t(\m{x})-\m{y}^*_t(\m{x}')\right\|\leq  L_{\m{y}} \left\|\m{x}-\m{x}'\right\|,\\\label{eqn:lip:cons2}
&\|\tilde{\nabla} f_t(\m{x},\m{y}) - \nabla f_t(\m{x},\m{y}^*_t(\m{x}))\| \leq  M_f \left\|\m{y}-\m{y}^*_t(\m{x})\right\|,\\\label{eqn:lip:cons3}
&\left\|\nabla f_t(\m{x},\m{y}_t^*(\m{x}))- \nabla f_t(\m{x}',\m{y}_t^*(\m{x}'))\right\|\leq  L_f \left\|\m{x}-\m{x}'\right\|.
\end{align}
Here, $L_{\m{y}}$, $M_f$, and $L_f$ are defined in  \eqref{eqn:def:ly}, \eqref{eqn:def:Mf}, and \eqref{eqn:def:Lf}, respectively.

\end{lem}
\begin{proof}
The proof is an adaptation of the proof from \cite[Lemma 2.2]{ghadimi2018approximation} to the online setting. 

We first show \eqref{eqn:lip:cons1}. 
Since $\m{y}^*_t(\m{x}) \in \argmin_{\m{y} \in  \mb{R}^{d_2}} g_t(\m{x}, \m{y})$, we have 
\begin{align*}
  \nabla_{\m{y}} g_t \left(\m{x},\m{y}^*_t (\m{x}) \right) =0, \quad \textnormal{and}\quad \nabla^2_{\m{x}\m{y}} g_t \left(\m{x},\m{y}^*_t (\m{x}) \right)=0.
\end{align*}
This together with the chain rule implies that 
\begin{equation*}
 \nabla \m{y}^*_t(\m{x}) \nabla^2_{\m{y}}g_t\left(\m{x},\m{y}^*_t(\m{x}) \right) +  \nabla^2_{\m{x}\m{y}} g_t \left(\m{x},\m{y}^*_t (\m{x}) \right)=0.
\end{equation*}
It follows from Assumption~\ref{assu:f:a3} that $\nabla^2_{\m{y}} g_t \left(\m{x},\m{y}^*_t (\m{x}) \right)$ is positive definite. Hence,
\begin{align*}
 \nabla \m{y}^*_t(\m{x})=- \nabla^2_{\m{x}\m{y}} g_t \left(\m{x},\m{y}^*_t (\m{x}) \right)\left( \nabla^2_{\m{y}} g_t \left(\m{x},\m{y}^*_t (\m{x}) \right)\right) ^{-1}  .
\end{align*}
Now, from Assumption \ref{assu:f:a2}, we get
\begin{equation}\label{eqn:def:ly}
    \begin{split}
 \norm{ \nabla \m{y}^*_t(\m{x})}&=\norm{\nabla^2_{\m{x}\m{y}} g_t \left(\m{x},\m{y}^*_t (\m{x}) \right)\left( \nabla^2_{\m{y}} g_t \left(\m{x},\m{y}^*_t (\m{x}) \right)\right) ^{-1}  }
\\&\leq \norm{\nabla^2_{\m{x}\m{y}} g_t \left(\m{x},\m{y}^*_t (\m{x}) \right)}\norm{\left( \nabla^2_{\m{y}} g_t \left(\m{x},\m{y}^*_t (\m{x}) \right)\right) ^{-1} }
\\&\leq \frac{\ell_{g,1}}{\mu_g}=:L_{\m{y}}.      
    \end{split}
\end{equation}
Next, we show \eqref{eqn:lip:cons2}. Let $\m{M}_t(\m{x},\m{y}):=\nabla^2_{\m{x}\m{y}} g_t \left(\m{x},\m{y} \right)\left( \nabla^2_{\m{y}} g_t \left(\m{x},\m{y}\right)\right) ^{-1} $.  Define
\begin{align*}
  &\Delta_t:= \nabla  f_t(\m{x}, \m{y})- \nabla  f_t(\m{x}, \m{y}_t^*(\m{x})),
  \\&\Delta_t^1:=\nabla_\m{x} f_t \left(\m{x}, \m{y}\right)-\nabla_\m{x} f_t \left(\m{x}, \m{y}^*_t (\m{x}) \right),
\\&\Delta_t^2:=\m{M}_t(\m{x},\m{y})\nabla_\m{y} f_t\left(\m{x},\m{y} \right)-\m{M}_t(\m{x},\m{y}^*_t(\m{x}))\nabla_\m{y} f_t\left(\m{x},\m{y}^*_t(\m{x}) \right),
\\&\Delta_t^3:=\m{M}_t(\m{x},\m{y})\{\nabla_\m{y}  f_t\left(\m{x},\m{y} \right)-\nabla f_t\left(\m{x},\m{y}^*_t(\m{x}) \right)\},
  \\&\Delta_t^4:=\{\m{M}_t(\m{x},\m{y})-\m{M}_t\left(\m{x},\m{y}^*_t(\m{x}) \right)\}\nabla_\m{y} f_t\left(\m{x},\m{y}^*_t(\m{x}) \right),
  \\&\Delta_t^5:=\left\{\nabla^2_{\m{x}\m{y}} g_t \left(\m{x},\m{y} \right)-\nabla^2_{\m{x}\m{y}} g_t \left(\m{x},\m{y}^*_t (\m{x}) \right) \right\}\left( \nabla^2_{\m{y}} g_t \left(\m{x},\m{y}\right)\right) ^{-1},
  \\&\Delta_t^6:=\nabla^2_{\m{x}\m{y}} g_t \left(\m{x},\m{y}^*_t (\m{x}) \right) \left\{\left( \nabla^2_{\m{y}} g_t \left(\m{x},\m{y} \right)\right) ^{-1}  -\left( \nabla^2_{\m{y}} g_t \left(\m{x},\m{y}^*_t (\m{x}) \right)\right) ^{-1}  \right \},
\end{align*}
which implies that 
  \begin{align}\label{jh}
  \Delta_t&= \Delta_t^1-\Delta_t^2  =\Delta_t^1-\Delta_t^3-\Delta_t^4= \Delta_t^1-\Delta_t^3-(\Delta_t^5+\Delta_t^6)\nabla_{\m{y}}  f_t(\m{x}, \m{y}_t^*(\m{x})).
  \end{align}
From Assumption \ref{assu:f}, we have
\begin{align}\label{44}
\|\Delta_t^1 \|\leq \ell_{f,1}\left\|\m{y}-\m{y}^*_t(\m{x})\right\|,  \quad\|\Delta_t^3 \|\leq \frac{\ell_{f,1}\ell_{g,1}}{\mu_g}\left\|\m{y}-\m{y}^*_t(\m{x})\right\|, \quad
\|\Delta_t^5 \|\leq \frac{\ell_{g,2}}{\mu_g}\left\|\m{y}-\m{y}^*_t(\m{x})\right\|. 
\end{align}
Note that, for any invertible matrices $\m{A}_1$ and $\m{A}_2$, we have
  \begin{align*}
    \|\m{A}_2^{-1}-\m{A}_1^{-1} \| & =\|\m{A}_1^{-1}(\m{A}_1-\m{A}_2)\m{A}_2^{-1} \|
\leq \|\m{A}_1^{-1} \| \|\m{A}_2^{-1} \| \|\m{A}_1-\m{A}_2 \|,
  \end{align*}
which implies that
\begin{align}\label{33}
\nonumber\|\Delta_t^6 \|&\leq \norm{\nabla^2_{\m{x}\m{y}} g_t \left(\m{x},\m{y}^*_t (\m{x}) \right)}\norm{ \left( \nabla^2_{\m{y}} g_t \left(\m{x},\m{y} \right)\right) ^{-1}  -\left( \nabla^2_{\m{y}} g_t \left(\m{x},\m{y}^*_t (\m{x}) \right) \right)^{-1}  }
\\\nonumber&\leq \ell_{g,1} \norm{ \left( \nabla^2_{\m{y}} g_t \left(\m{x},\m{y} \right)\right) ^{-1} } \norm{\left( \nabla^2_{\m{y}} g_t \left(\m{x},\m{y}^*_t (\m{x}) \right) \right)^{-1} } \norm{\nabla^2_{\m{y}} g_t \left(\m{x},\m{y} \right)-\nabla^2_{\m{y}} g_t \left(\m{x},\m{y}^*_t (\m{x}) \right)}
\\&\leq \frac{\ell_{g,1}\ell_{g,2}}{\mu_g^2}\left\|\m{y}-\m{y}_t^*(\m{x})\right\|.
\end{align}
Therefore, by substitution \eqref{44} and \eqref{33} into \eqref{jh}, we have
\begin{align}\label{ddc}
\nonumber  \| \nabla  f_t(\m{x}, \m{y})- \nabla  f_t(\m{x}, \m{y}_t^*(\m{x}))\|&\leq  \| \Delta_t^1\|+\|\Delta_t^3\|+\|\Delta_t^5+\Delta_t^6\| \|\nabla_{\m{y}}  f_t(\m{x}, \m{y}_t^*(\m{x}))\| 
  \\\nonumber&\leq \| \Delta_t^1\|+\|\Delta_t^3\|+\|\Delta_t^5+\Delta_t^6\| \ell_{f,0}
  \\&\leq M_f \left\|\m{y}-\m{y}^*_t(\m{x})\right\|,
\end{align}
where 
\begin{align}\label{eqn:def:Mf}
 M_f :=\ell_{f,1} + \frac{\ell_{g,1}\ell_{f,1}}{\mu_g} + \frac{\ell_{f,0}}{\mu_g}\Big(\ell_{g,2}+\frac{{\ell_{g,1} \ell_{g,2}}}{\mu_g}\Big).   
\end{align}

Next, we show \eqref{eqn:lip:cons3}. Note that
\begin{equation}
\begin{aligned}
 \label{8}   \left\|\nabla f_t(\m{x},\m{y}^*_t(\m{x}))- \nabla f_t(\m{x}',\m{y}^*_t(\m{x}'))\right\|&\leq \left\|\nabla f_t(\m{x},\m{y}^*_t(\m{x}))- \tilde{\nabla} f_t(\m{x},\m{y}^*_t(\m{x}'))\right\|
    \\&+\left\|\tilde{\nabla} f_t(\m{x},\m{y}_t^*(\m{x}'))- \nabla f_t(\m{x}',\m{y}_t^*(\m{x}'))\right\|.
\end{aligned}
\end{equation}
We then study each terms separately. From \eqref{ddc} and \eqref{eqn:lip:cons1}, we get
\begin{align*}
\left\|\nabla f_t(\m{x},\m{y}_t^*(\m{x}))- \tilde{\nabla} f_t(\m{x},\m{y}_t^*(\m{x}'))\right\|   \leq M_f\|\m{y}_t^*(\m{x})- \m{y}_t^*(\m{x}')\| \leq M_f L_{\m{y}} \|\m{x}-\m{x}' \|.
\end{align*}
Moreover, by similar argument to \eqref{ddc}, we obtain 
\begin{align*}
 \left\|\nabla f_t(\m{x},\m{y}^*_t(\m{x}'))- \nabla f_t(\m{x}',\m{y}^*_t(\m{x}'))\right\|\leq  \left(\ell_{f,1} + \frac{\ell_{g,1}\ell_{f,1}}{\mu_g} + \frac{\ell_{f,0}}{\mu_g}\left(\ell_{g,2}+\frac{{\ell_{g,1} \ell_{g,2}}}{\mu_g}\right) \right)\|\m{x}-\m{x}' \|.  
\end{align*}
By substituting the above two inequalities into \eqref{8}, we have
\begin{align*}
  \left\|\nabla f_t(\m{x},\m{y}^*_t(\m{x}))- \nabla f_t(\m{x}',\m{y}^*_t(\m{x}'))\right\| \leq  L_f\|\m{x}-\m{x}' \|,
\end{align*}
where 
\begin{align}\label{eqn:def:Lf}
 L_f:=\ell_{f,1} + \frac{\ell_{g,1}(\ell_{f,1}+M_f)}{\mu_g} + \frac{\ell_{f,0}}{\mu_g}\left(\ell_{g,2}+\frac{{\ell_{g,1} \ell_{g,2}}}{\mu_g}\right).    
\end{align}
\end{proof}

The following lemma characterizes the inner estimation errors
$\sum_{t=1}^T\|\m{y}_{t+1} -  \m{y}^*_t(\m{x}_t) \|^2$  and $\sum_{t=1}^T\|\m{y}_{t+1} -  \m{y}^*_t(\m{x}_t) \|$, where $ \m{y}_{t+1} $ is the inner variable updated via Algorithm~\ref{alg:obgd}. It shows that by applying inner OGD multiple times at each round $t$, we are able to extract more information from each inner function and, therefore, are more likely to obtain a tight bound for the inner error in terms of the path-length $Y_{p,T}$. 
\begin{lem}\label{eq:lower:dynamic}
Suppose Assumption~\ref{assu:f} holds. In Algorithm~\ref{alg:obgd}, choose
\begin{equation*}
\beta_t=\beta =\frac{2 }{\ell_{g,1}+\mu_g},~~\textnormal{and}~~K_t>\left\lceil\frac{(\kappa_g+1)\log \rho^{-2}_t}{4}\right\rceil 
\end{equation*}
for some positive deceasing sequence $\{\rho_t\}_{t=1}^T$.
Then, Algorithm~\ref{alg:obgd} guarantees the following.
\begin{enumerate}[label={\textnormal{{L}\arabic*.}}]
\item \label{item:leml1} If $\rho_1 < \sqrt{1/2}$, we have
\begin{align*}
\nonumber
\sum_{t=1}^T \|  \m{y}_{t+1} -  \m{y}^*_t(\m{x}_t) \|^2  &\leq   \frac{\rho^2_1}{1-2\rho^2_1} \|\m{y}_{1} - \m{y}_{1}^*(\m{x}_1)\|^2 \\
&+ \frac{6}{1-2\rho^2_1}  \left(2 L_{\m{y}}^2\sum_{t=1}^{T}\rho_{t}^2\|\m{x}_{t}-\m{x}^*_{t}\|^2 + \sum_{t=2}^T\rho_t^2 \|\m{y}^*_{t-1}(\m{x}^*_{t-1})-\m{y}^*_{t}(\m{x}^*_{t})\|^2 \right).
\end{align*}
\item \label{item:leml2} If $\rho_1 < 1$, we get
\begin{align*}
\sum_{t=1}^T \|  \m{y}_{t+1} -  \m{y}^*_t(\m{x}_t) \| &\leq   \frac{\rho_1}{1-\rho_1} \|\m{y}_{1} - \m{y}_{1}^*(\m{x}_1)\| \\
&+  \frac{1}{1-\rho_1}  \left(2 L_{\m{y}}\sum_{t=1}^{T}\rho_{t}\|\m{x}_{t}-\m{x}^*_{t}\|+  \sum_{t=2}^T\rho_t \|\m{y}^*_{t-1}(\m{x}^*_{t-1})-\m{y}^*_{t}(\m{x}^*_{t})\| \right).
\end{align*}
\end{enumerate}
\end{lem}
\begin{proof}
We show \ref{item:leml1}. The proof of \ref{item:leml2} follows similarly. 
Since $\beta =2/(\ell_{g,1}+\mu_g)$, from Lemma \ref{lm:nes},  we have
\begin{align*}
    \|\m{z}^{K_t+1}_t  - \m{y}^*_{t}(\m{x}_t)\|^2 \leq  \left( 1 -  \frac{2}{\kappa_g+1}  \right)^{2} \|\m{z}^{K_t}_t  - \m{y}^*_{t}(\m{x}_t)\|^2,
\end{align*}
which implies that
\begin{align}\label{eqn:lem:str:30}
\|\m{z}^{K_t+1}_t  - \m{y}^*_{t}(\m{x}_t)\|^2 \leq    \left( 1 -  \frac{2}{\kappa_g+1}  \right)^{2K_t}\|\m{z}^1_t- \m{y}^*_{t}(\m{x}_t)\|^2.
\end{align}
By our assumption $K_t>\lceil  0.25 (\kappa_g+1)\log \rho^{-2}_t\rceil$ which implies that
\begin{align}\label{eqn:lem:str:31}
\left( 1 -  \frac{2}{\kappa_g+1}  \right)^{2K_t} \leq  \exp \left(-   \frac{4K_t}{\kappa_g+1} \right) \leq \rho^2_t.
\end{align}
Then, using \eqref{eqn:lem:str:30} and \eqref{eqn:lem:str:31}, we have
\begin{align*}
\|\m{z}_t^{K_t+1}  - \m{y}^*_{t}(\m{x}_t)\|^2 =\|\m{y}_{t+1} - \m{y}^*_{t}(\m{x}_t)\|^2  \leq  \rho^2_t \|\m{y}_{t} - \m{y}^*_{t}(\m{x}_t)\|^2.
\end{align*}
Hence,
\begin{equation} \label{eqn:lem:lin3}
\sum_{t=1}^T \|  \m{y}_{t+1} -  \m{y}^*_t(\m{x}_t) \|^2 \leq \rho^2_1 \|\m{y}_{1} - \m{y}_{1}^*(\m{x}_1)\|^2 +  \sum_{t=2}^T  \rho^2_t \|\m{y}_{t} - \m{y}^*_{t}(\m{x}_t)\|^2,
\end{equation}
which implies that
\begin{equation}\label{eqn:lem:lin4}
\begin{split}
\sum_{t=2}^T  \rho^2_t \|\m{y}_{t} - \m{y}^*_{t}(\m{x}_t)\|^2  &\leq  2 \sum_{t=2}^{T} \rho^2_t\left( \|\m{y}_{t} - \m{y}^*_{t-1}(\m{x}_{t-1})\|^2 + \|\m{y}^*_{t-1}(\m{x}_{t-1}) - \m{y}^*_{t}(\m{x}_{t})\|^2  \right)\\
& \leq 2 \sum_{t=1}^{T}  \rho^2_t \|\m{y}_{t+1} - \m{y}^*_{t}(\m{x}_{t})\|^2 +   2\sum_{t=2}^{T}\rho^2_t \|\m{y}^*_{t-1}(\m{x}_{t-1}) - \m{y}^*_{t}(\m{x}_{t})\|^2.
\end{split}
\end{equation}
It follows from Lemma~\ref{lem:Jens} that
\begin{eqnarray}\label{eqn:lem:lin5}
\nonumber
\rho^2_t \|\m{y}^*_{t-1}(\m{x}_{t-1}) - \m{y}^*_{t}(\m{x}_{t})\|^2
 &\leq &  3 \rho_t^2 \|\m{y}^*_{t}(\m{x}_{t})-\m{y}^*_{t}(\m{x}^*_{t})\|^2 \\
 \nonumber
 &+&  3 \rho_t^2  \|\m{y}^*_{t-1}(\m{x}_{t-1}) - \m{y}^*_{t-1}(\m{x}^*_{t-1})\|^2\\
\nonumber
 &+& 3 \rho_t^2 \|\m{y}^*_{t-1}(\m{x}^*_{t-1})-\m{y}^*_{t}(\m{x}^*_{t})\|^2\\
 \nonumber
&  \leq &  3 L_{\m{y}}^2 \rho_{t-1}^2\|\m{x}_{t-1} -\m{x}^*_{t-1}\|^2 \\
\nonumber
 &+ & 3 L_{\m{y}}^2   \rho_t^2 \|\m{x}_{t}-\m{x}^*_{t}\|^2 \\
&+&3 \rho_t^2  \|\m{y}^*_{t-1}(\m{x}^*_{t-1})-\m{y}^*_{t}(\m{x}^*_{t})\|^2,
\end{eqnarray}
where the second inequality uses the assumption that  $\rho_{t} \leq  \rho_{t-1}$ for all $t \in[T]$.

Now, combining \eqref{eqn:lem:lin3}, \eqref{eqn:lem:lin4}, and \eqref{eqn:lem:lin5}, we obtain
\begin{equation*}
\begin{split}
 \sum_{t=1}^T \left(1-2\rho^2_t\right)  \|  \m{y}_{t+1} -  \m{y}^*_t(\m{x}_t) \|^2 & \leq \rho^2_1   \|\m{y}_{1} - \m{y}_{1}^*(\m{x}_1)\|^2 + 12 L_{\m{y}}^2\sum_{t=1}^T \rho_t^2 \|\m{x}_{t}-\m{x}^*_{t}\|^2 \\
&+ 6 \sum_{t=2}^T \rho_t^2 \|\m{y}^*_{t-1}(\m{x}^*_{t-1})-\m{y}^*_{t}(\m{x}^*_{t})\|^2,
  \end{split}
\end{equation*}
which together with our assumption that  $\rho_{t} \leq  \rho_{t-1}$  completes the proof.
\end{proof}
The following lemma is an extension of \citep[Proposition 2]{mokhtari2016online} to online bilevel optimization, characterizing the dynamics of the tracking error $\|\mathbf{x}_{t}-\mathbf{x}_{t}^*\|^2$. Specifically, it shows that $\|\mathbf{x}_{t+1} - \mathbf{x}_{t}^*\|^2$ can be upper bounded in terms of $\|\mathbf{x}_{t}-\mathbf{x}_{t}^*\|^2$ and $\|\mathbf{y}_{t+1} - \mathbf{y}^*_{t}(\mathbf{x}_{t})\|^2$.
\begin{lem}\label{lm:xxx}
Suppose Assumption~\ref{assu:f} holds and $\alpha_t=\alpha \leq   {1}/{\ell_{f,1}}$ for all $t \in [T]$. Further, assume functions  $\{ f_t(\m{x}, \m{y}^*_t(\m{x}))\}_{t=1}^T$ are strongly convex with parameter $\mu_{f}$. Then, for the sequence $\{(\m{x}_t, \m{y}_t)\}_{t=1}^T$ generated by  Algorithm~\ref{alg:obgd}, we have 
\begin{align}\label{ree}
  \|\m{x}_{t+1} - \m{x}_{t}^*\|^2
&\leq (1-\gamma) \|\m{x}_{t}-\m{x}_{t}^*\|^2
  + \frac{2M_f^2\alpha}{\big(1+\frac{\mu_f}{2}\alpha \big)\mu_f} 
 \norm{\m{y}_{t+1} -\m{y}^*_{t}(\m{x}_{t})}^2,
\end{align}
where
\begin{equation}\label{eqn:gamma}
 \gamma:=\frac{3\mu_f}{\frac{2}{\alpha}+\mu_f} \in (0,1].
\end{equation}
\end{lem}
\begin{proof}
From $\mu_f$-strong convexity of $f_{t}$, we get
\begin{align}\label{dfaa}
\nonumber f_{t}(\m{x}, \m{y}_{t}^*(\m{x}))&\geq  f_{t}(\m{x}_{t},\m{y}_{t}^*(\m{x}_{t}))+
  \left\langle\nabla f_{t}(\m{x}_{t}, \m{y}_{t}^*(\m{x}_{t})), \m{x}-\m{x}_{t}\right \rangle+\frac{\mu_f}{2} \left\|\m{x}-\m{x}_{t}\right\|^2
 \\\nonumber &=f_{t}(\m{x}_{t},\m{y}_{t}^*(\m{x}_{t}))
 +
  \left\langle\nabla f_{t}(\m{x}_{t},\m{y}_{t}^*(\m{x}_{t})), \m{x}_{t+1}-\m{x}_{t}\right \rangle
 \\&+
  \left\langle\nabla f_{t}(\m{x}_{t}, \m{y}_{t}^*(\m{x}_{t})), \m{x}-\m{x}_{t+1}\right \rangle+\frac{\mu_f}{2} \left\|\m{x}-\m{x}_{t}\right\|^2.
\end{align}
According to the optimality condition of the update rule $\m{x}_{t+1} =\Pi_{\mc{X}} \left[ \m{x}_{t}- \alpha\tilde{\nabla} f_t(\m{x}_t,\m{y}_{t+1})\right]$, we have
\begin{align*}
\langle \tilde{\nabla} f_t(\m{x}_t,\m{y}_{t+1})+\frac{1}{\alpha}(\m{x}_{t+1}-\m{x}_{t}),\m{x}-\m{x}_{t+1} \rangle \geq 0,
\end{align*}
which is equivalent to
\begin{align*}
 \langle \tilde{\nabla} f_t(\m{x}_t,\m{y}_{t+1})-\nabla f_{t}(\m{x}_{t}, \m{y}_{t}^*(\m{x}_{t}))+\nabla f_{t}(\m{x}_{t}, \m{y}_{t}^*(\m{x}_{t})),\m{x}-\m{x}_{t+1} \rangle
\geq \frac{1}{\alpha}\langle \m{x}_{t}-\m{x}_{t+1},\m{x}-\m{x}_{t+1} \rangle.
\end{align*}
Hence, 
\begin{align*}
\langle \nabla f_{t}(\m{x}_{t},\m{y}_{t}^*(\m{x}_{t})),\m{x}-\m{x}_{t+1} \rangle
&\geq \frac{1}{\alpha}\langle \m{x}_{t}-\m{x}_{t+1},\m{x}-\m{x}_{t+1} \rangle
\\&+\langle  \nabla f_{t}(\m{x}_{t},\m{y}_{t}^*(\m{x}_{t}))-\tilde{\nabla} f_t(\m{x}_t,\m{y}_{t+1}),\m{x}-\m{x}_{t+1} \rangle.
\end{align*}
Substituting this inequality in  \eqref{dfaa}, we get
\begin{align}\label{fdbb}
\nonumber
 f_{t}(\m{x}, \m{y}_{t}^*(\m{x}))&\geq  f_{t}(\m{x}_{t},\m{y}_{t}^*(\m{x}_{t}))+
  \left\langle\nabla f_{t}(\m{x}_{t}, \m{y}_{t}^*(\m{x}_{t})), \m{x}_{t+1}-\m{x}_{t}\right \rangle
  +\frac{1}{\alpha}\langle \m{x}_{t}-\m{x}_{t+1},\m{x}-\m{x}_{t+1} \rangle
  \\&+\langle \nabla f_{t}(\m{x}_{t},\m{y}_{t}^*(\m{x}_{t}))-\tilde{\nabla} f_t(\m{x}_t,\m{y}_{t+1}),\m{x}-\m{x}_{t+1} \rangle+\frac{\mu_f}{2} \left\|\m{x}-\m{x}_{t}\right\|^2.
\end{align}
In addition, $\ell_{f,1}$-smoothness of $f_{t}$ (Assumption~\ref{assu:f:a2}) gives
\begin{align*}
f_{t}(\m{x}_{t+1}, \m{y}_{t}^*(\m{x}_{t+1}))&\leq  f_{t}(\m{x}_{t},\m{y}_{t}^*(\m{x}_{t}))+
  \left\langle\nabla f_{t}(\m{x}_{t}, \m{y}_{t}^*(\m{x}_{t})), \m{x}_{t+1}-\m{x}_{t}\right \rangle+\frac{\ell_{f,1}}{2} \left\|\m{x}_{t+1}-\m{x}_{t}\right\|^2
  \\&\leq  f_{t}(\m{x}_{t},\m{y}_{t}^*(\m{x}_{t}))+
  \left\langle\nabla f_{t}(\m{x}_{t}, \m{y}_{t}^*(\m{x}_{t})), \m{x}_{t+1}-\m{x}_{t}\right \rangle+\frac{1}{2\alpha} \left\|\m{x}_{t+1}-\m{x}_{t}\right\|^2,
\end{align*}
where the inequality is by $\alpha\leq 1/\ell_{f,1}$.

Thus,
\begin{align}\label{dfv}
\nonumber &\quad f_{t}(\m{x}_{t},\m{y}_{t}^*(\m{x}_{t}))+
  \left\langle\nabla f_{t}(\m{x}_{t}, \m{y}_{t}^*(\m{x}_{t})), \m{x}_{t+1}-\m{x}_{t}\right \rangle
  \\\nonumber &\geq f_{t}(\m{x}_{t+1}, \m{y}_{t}^*(\m{x}_{t+1}))-\frac{1}{2\alpha} \left\|\m{x}_{t+1}-\m{x}_{t}\right\|^2
  \\&\geq f_{t}(\m{x}_{t}^*,\m{y}_{t}^*(\m{x}_{t}^*))+\frac{\mu_f}{2}\left\|\m{x}_{t+1}-\m{x}_{t}^*\right\|^2-\frac{1}{2\alpha} \left\|\m{x}_{t+1}-\m{x}_{t}\right\|^2,
\end{align}
where the second inequality holds since
from Lemma \ref{lm:cc}, we have
\begin{align*}
f_{t}(\m{x}_{t}^*, \m{y}_{t}^*(\m{x}_{t}^*))\leq f_{t}(\m{x}_{t+1}, \m{y}_{t}^*(\m{x}_{t+1}))-\frac{\mu_f}{2}\left\|\m{x}_{t+1}-\m{x}_{t}^*\right\|^2.
\end{align*}

Combining \eqref{fdbb} and \eqref{dfv}, we get
\begin{align*}
 \nonumber f_{t}(\m{x}, \m{y}_{t}^*(\m{x}))&\geq f_{t}(\m{x}_{t}^*, \m{y}_{t}^*(\m{x}_{t}^*))+\frac{\mu_f}{2}\left\|\m{x}_{t+1}-\m{x}_{t}^*\right\|^2
 -\frac{1}{2\alpha} \left\|\m{x}_{t+1}-\m{x}_{t}\right\|^2
 +\frac{1}{\alpha}\langle \m{x}_{t}-\m{x}_{t+1},\m{x}-\m{x}_{t+1} \rangle
 \\&+\langle \nabla f_{t}(\m{x}_{t},\m{y}_{t}^*(\m{x}_{t}))-\tilde{\nabla} f_t(\m{x}_t,\m{y}_{t+1}),\m{x}-\m{x}_{t+1} \rangle+\frac{\mu_f}{2} \left\|\m{x}-\m{x}_{t}\right\|^2.
\end{align*}
By setting $\m{x}=\m{x}_{t}^*$, we have
\begin{align*}
 \nonumber  f_{t}(\m{x}_{t}^*,\m{y}_{t}^*(\m{x}_{t}^*))
  &\geq f_{t}(\m{x}_{t}^*,\m{y}_{t}^*(\m{x}_{t}^*))+\frac{\mu_f}{2}\left\|\m{x}_{t+1}-\m{x}_{t}^*\right\|^2
 \\\nonumber &-\frac{1}{2\alpha} \left\|\m{x}_{t+1}-\m{x}_{t}\right\|^2
 +\frac{1}{\alpha}\langle \m{x}_{t}-\m{x}_{t+1},\m{x}_{t}^*-\m{x}_{t+1} \rangle
\\ & +\langle \nabla f_{t}(\m{x}_{t},\m{y}_{t}^*(\m{x}_{t}))-\tilde{\nabla} f_t(\m{x}_t,\m{y}_{t+1}),\m{x}_{t}^*-\m{x}_{t+1} \rangle
+\frac{\mu_f}{2} \left\|\m{x}_{t}^*-\m{x}_{t}\right\|^2.
\end{align*}
Since $\langle u,v\rangle \geq -\frac{c}{2} \|u \|^2-\frac{1}{2c}\| v\|^2 $, $\forall u,v\in \mathbb{R}^n$, $ \forall c>0$, we obtain
\begin{align*}
 \nonumber 0\geq &-\frac{1}{2\alpha} \left\|\m{x}_{t+1}-\m{x}_{t}\right\|^2
 +\frac{1}{\alpha}\langle \m{x}_{t}-\m{x}_{t+1},\m{x}_{t}^*-\m{x}_{t} \rangle
\\\nonumber & +\frac{1}{\alpha}\langle \m{x}_{t}-\m{x}_{t+1},\m{x}_{t}-\m{x}_{t+1} \rangle
  -\frac{1}{2c }\|\nabla f_{t}(\m{x}_{t},\m{y}_{t}^*(\m{x}_{t}))-\tilde{\nabla} f_t(\m{x}_t,\m{y}_{t+1}) \|^2
 \\&
 +\left(\frac{\mu_f}{2}-\frac{c}{2}\right)\|\m{x}_{t}^*-\m{x}_{t+1} \|^2
 +\frac{\mu_f}{2} \left\|\m{x}_{t}^*-\m{x}_{t}\right\|^2.
\end{align*}
After rearranging, we obtain
\begin{align}\label{dgdv}
\nonumber  &\quad \langle \m{x}_{t}-\m{x}_{t+1},\m{x}_{t}^*-\m{x}_{t} \rangle
\\\nonumber &\leq \frac{1}{2} \left\|\m{x}_{t+1}-\m{x}_{t}\right\|^2
-\frac{\mu_f}{2}\alpha \left\|\m{x}_{t}^*-\m{x}_{t}\right\|^2
-\| \m{x}_{t}-\m{x}_{t+1}\|^2
+(\frac{c}{2}-\frac{\mu_f}{2})\alpha\|\m{x}_{t}^*-\m{x}_{t+1} \|^2
\\&+ \frac{\alpha}{2c}\|\nabla f_{t}(\m{x}_{t},\m{y}_{t}^*(\m{x}_{t}))-\tilde{\nabla} f_t(\m{x}_t,\m{y}_{t+1}) \|^2.
\end{align}
Note that
\begin{align}\label{s1}
\nonumber \|\m{x}_{t+1} - \m{x}_{t}^*\|^2&=\|\m{x}_{t+1}-\m{x}_{t}+\m{x}_{t} - \m{x}_{t}^*\|^2
 \\\nonumber &=\|\m{x}_{t+1}-\m{x}_{t}\|^2+\|\m{x}_{t} - \m{x}_{t}^*\|^2+2\langle \m{x}_{t+1}-\m{x}_{t},\m{x}_{t} - \m{x}_{t}^* \rangle
  \\\nonumber &\leq (1-\mu_f\alpha)\left\|\m{x}_{t}^*-\m{x}_{t}\right\|^2+(c-\mu_f)\alpha\|\m{x}_{t}^*-\m{x}_{t+1} \|^2
  \\&+\frac{\alpha}{c} \|\nabla f_{t}(\m{x}_{t},\m{y}_{t}^*(\m{x}_{t}))-\tilde{\nabla} f_t(\m{x}_t,\m{y}_{t+1}) \|^2,
\end{align}
where the inequality follows from \eqref{dgdv}.

From Lemma \ref{lem:lips}, we obtain
\begin{align}\label{dcjk}
 \|\nabla f_{t}(\m{x}_{t},\m{y}_{t}^*(\m{x}_{t}))-\tilde{\nabla} f_t(\m{x}_t,\m{y}_{t+1}) \|^2&\leq  M_f^2 \norm{\m{y}_{t+1} -\m{y}^*_{t}(\m{x}_{t})}^2.
\end{align}
Inserting \eqref{dcjk} into \eqref{s1} implies
\begin{align*}
  \nonumber  \left(1-(c-\mu_f)\alpha\right)\|\m{x}_{t+1} - \m{x}_{t}^*\|^2
 &\leq \left(1-\mu_f\alpha  \right)\left\|\m{x}_{t}^*-\m{x}_{t}\right\|^2
  +\frac{M_f^2}{c}\alpha\norm{\m{y}_{t+1} -\m{y}^*_{t}(\m{x}_{t})}^2.
\end{align*}
By setting $ c = {\mu_f}/{2}$, we get
\begin{align*}
  \nonumber \left(1+\frac{\mu_f}{2}\alpha \right)\|\m{x}_{t+1} - \m{x}_{t}^*\|^2
&\leq \left(1-\mu_f\alpha \right)\left\|\m{x}_{t}^*-\m{x}_{t}\right\|^2
  +\frac{2M_f^2}{\mu_f}\alpha \left(\norm{\m{y}_{t} -\m{y}^*_{t}(\m{x}_{t})}^2
   +\norm{\m{v}_{t} -\m{v}^*_{t}(\m{x}_{t})}^2\right)
   .
\end{align*}
Finally, dividing both sides of the above inequality by $\big(1+\frac{\mu_f}{2}\alpha \big)$, we obtain
\begin{align*}
\nonumber \|\m{x}_{t+1} - \m{x}_{t}^*\|^2&\leq \gamma\left\|\m{x}_{t}^*-\m{x}_{t}\right\|^2
  +\frac{2M_f^2\alpha}{\left(1+\frac{\mu_f}{2}\alpha \right)\mu_f} \left(\norm{\m{y}_{t} -\m{y}^*_{t}(\m{x}_{t})}^2
   +\norm{\m{v}_{t} -\m{v}^*_{t}(\m{x}_{t})}^2\right),
\end{align*}
where $\gamma$ is defined in \eqref{eqn:gamma}.

Since the strong convexity constant $\mu_f$ is smaller than the constant of gradient Lipschitz continuity $\ell_{f,1}$, and the constant $\alpha$ is chosen such that $\alpha \leq 1/\ell_{f,1}$, we have $\alpha \leq 1/\mu_f$, which implies that $\gamma \leq 1$.
\end{proof}

The following lemma plays a key role in the proof of OAGD in the strongly convex setting. It basically shows that under certain conditions on inner and outer step sizes, $\sum_{t=1}^T\|\m{x}_{t} - \m{x}^*_{t}\|^p$ can be bounded in terms of $P_{p,T}$ and $Y_{p,T}$.

\begin{lem}\label{lm:xxxb}
Suppose Assumption~\ref{assu:f} holds. Further, assume functions $\{ f_t(\m{x}, \m{y}^*_t(\m{x}))\}_{t=1}^T$ are strongly convex with parameter $\mu_{f}$. In Algorithm~\ref{alg:obgd}, for all $t \in [T]$, choose
\begin{align*}
\beta_t&=\beta =\frac{2 }{\ell_{g,1}+\mu_g},\\
\alpha_t&=\alpha \leq \min\left\{\frac{1}{\ell_{f,1}}, \frac{\mu_f}{128M_f^2 L_{\m{y}}^2}\right\},  \quad \textnormal{and}~~\\
K_t&>\left\lceil\frac{(\kappa_g+1)\log \rho^{-2}_t}{4}\right\rceil .
\end{align*}
Then, Algorithm~\ref{alg:obgd} guarantees the following.
\begin{enumerate}[label={\textnormal{{H}\arabic*.}}]
\item \label{vff} 
If $\rho_t=\rho \leq \frac{\gamma}{1+\gamma}$, then
\begin{align*}
 \nonumber  \sum_{t=1}^{T} \|\m{x}_{t}-\m{x}^*_{t}\|
&\leq \frac{4}{\gamma}\left(\|\m{x}_{1}-\m{x}^*_{1}\|+P_{1,T}\right)
\\&+\frac{1}{2L_{\m{y}}} \left(  \|\m{y}_{1} - \m{y}_{1}^*(\m{x}_1)\|+   Y_{1,T}\right).
\end{align*}
\item \label{vgg}
If $\rho_t=\rho \leq \frac{\sqrt{\gamma}}{\sqrt{2}\sqrt{\gamma+1}}$, then
\begin{align*}
\nonumber \sum_{t=1}^{T} \|\m{x}_{t}-\m{x}^*_{t}\|^2&\leq \frac{64}{23\gamma}\left(\|\m{x}_{1}-\m{x}^*_{1}\|^2
  +(1+\frac{2}{\gamma} )P_{2,T}\right)\\
  &+\frac{3}{92L_{\m{y}}^2}\left( \|\m{y}_{1} - \m{y}_{1}^*(\m{x}_1)\|^2+ 6Y_{2,T}\right).
\end{align*}
\end{enumerate}
Here, $\gamma$ is defined in \eqref{eqn:gamma}; $P_{p,T}$ and $Y_{p,T}$ are defined in \eqref{eq:pathlength}.
\end{lem}
\begin{proof}
We first show \ref{vff}.  It follows from the triangle inequality that
\begin{align}\label{cvxv}
\nonumber \sum_{t=1}^{T} \|\m{x}_{t}-\m{x}^*_{t}\|&=\|\m{x}_{1}-\m{x}^*_{1}\|+\sum_{t=2}^{T} \|\m{x}_{t}-\m{x}^*_{t}\|
 \\\nonumber &\leq \|\m{x}_{1}-\m{x}^*_{1}\|+\sum_{t=2}^{T} \Big(\|\m{x}_{t}-\m{x}^*_{t-1}\|+ \|\m{x}^*_{t-1}-\m{x}^*_{t}\|\Big)
 \\&\leq \|\m{x}_{1}-\m{x}^*_{1}\|+\sum_{t=1}^{T} \|\m{x}_{t+1}-\m{x}^*_{t}\|+P_{1,T}.
\end{align}
Next, we provide an upper bound for the second term on the right-hand side of \eqref{cvxv}. Note that our choice of the stepsize $\alpha_t$ in the statement of Lemma~\ref{lm:xxxb} satisfies the condition of Lemma \ref{lm:xxx}. Hence, from Lemma \ref{lm:xxx} and the inequality $\sqrt{a+b}\leq \sqrt{a}+\sqrt{b}$ for $a,b\geq 0$, we get
\begin{align*}
  \|\m{x}_{t+1}-\m{x}^*_{t}\| &\leq  \sqrt{1-\gamma} \|\m{x}_{t}-\m{x}_{t}^*\|
  +M_f\sqrt{\frac{2\alpha}{\mu_f}} \norm{\m{y}_{t+1} -\m{y}^*_{t}(\m{x}_{t})} .  
\end{align*}
Summing both sides of the above inequality from $t = 1$ to $T$, we get
\begin{align}\label{vv}
\nonumber  \sum_{t=1}^{T} \|\m{x}_{t+1}-\m{x}^*_{t}\| &\leq  \sqrt{1-\gamma} \sum_{t=1}^{T}\|\m{x}_{t}-\m{x}_{t}^*\|
  +M_f\sqrt{\frac{2\alpha}{\mu_f}} \sum_{t=1}^{T}\norm{\m{y}_{t+1} -\m{y}^*_{t}(\m{x}_{t})} 
  \\&\leq (1-\frac{\gamma}{2}) \sum_{t=1}^{T}\|\m{x}_{t}-\m{x}_{t}^*\|
  +M_f\sqrt{\frac{2\alpha}{\mu_f}} \sum_{t=1}^{T}\norm{\m{y}_{t+1} -\m{y}^*_{t}(\m{x}_{t})}.
\end{align}
Here, $\gamma=\frac{3\mu_f}{\frac{2}{\alpha}+\mu_f}$, and the second inequality follows since $\sqrt{1-a}\leq 1-\frac{a}{2}$ for any $a\leq 1$.

Note that our assumption on $\rho$ in the statement of Lemma~\ref{lm:xxxb}--\ref{vff} satisfies the requirement of Lemma \ref{eq:lower:dynamic}--\ref{item:leml2}. Hence, we have 
\begin{align}\label{ddn}
\nonumber\sum_{t=1}^T \|  \m{y}_{t+1} -  \m{y}^*_t(\m{x}_t) \| &\leq   \frac{\rho_1}{1-\rho_1} \|\m{y}_{1} - \m{y}_{1}^*(\m{x}_1)\| \\
&+  \frac{1}{1-\rho_1}  \left(2 L_{\m{y}}\sum_{t=1}^{T}\rho_{t}\|\m{x}_{t}-\m{x}^*_{t}\|+  \sum_{t=2}^T\rho_t \|\m{y}^*_{t-1}(\m{x}^*_{t-1})-\m{y}^*_{t}(\m{x}^*_{t})\| \right).
  \end{align}
Substituting \eqref{ddn} into \eqref{vv}, we get
\begin{align}\label{cc}
\nonumber  \sum_{t=1}^{T} \|\m{x}_{t+1}-\m{x}^*_{t}\| &\leq \sum_{t=1}^{T}\left(1-\frac{\gamma}{2}+\frac{M_fL_{\m{y}}2\sqrt{2}}{\sqrt{\mu_f}}\sqrt{\alpha}\frac{\rho_{t}}{1-\rho_1}\right) \|\m{x}_{t}-\m{x}_{t}^*\|
  \\&+M_f\sqrt{\frac{2\alpha}{\mu_f}} \left( \frac{\rho_1}{1-\rho_1} \|\m{y}_{1} - \m{y}_{1}^*(\m{x}_1)\|+ \frac{1}{1-\rho_1} \sum_{t=2}^T\rho_t \|\m{y}^*_{t-1}(\m{x}^*_{t-1})-\m{y}^*_{t}(\m{x}^*_{t})\|\right).
\end{align}
By setting $\rho_t=\rho \leq \frac{\gamma}{1+\gamma}$, we have
\begin{equation}\label{bb}
\frac{M_fL_{\m{y}}2\sqrt{2}}{\sqrt{\mu_f}}\sqrt{\alpha}\frac{\rho}{1-\rho} \leq \frac{M_fL_{\m{y}}2\sqrt{2}}{\sqrt{\mu_f}}\sqrt{\alpha}\gamma\leq \frac{\gamma}{4},
\end{equation}
where the second inequality holds due to our assumption on the outer stepsize, i.e., $\alpha \leq \frac{\mu_f}{128M_f^2 L_{\m{y}}^2}$.

Combining the above two inequalities \eqref{bb} and \eqref{cc}, we conclude that
\begin{align*}
\nonumber  \sum_{t=1}^{T} \|\m{x}_{t+1}-\m{x}^*_{t}\| &\leq \left(1-\frac{\gamma}{4}\right)\sum_{t=1}^{T} \|\m{x}_{t}-\m{x}_{t}^*\|
  +M_f\sqrt{\frac{2\alpha}{\mu_f}} \left( \frac{\rho}{1-\rho} \|\m{y}_{1} - \m{y}_{1}^*(\m{x}_1)\|+ \frac{\rho}{1-\rho}  Y_{1,T}\right)
  \\&\leq \left(1-\frac{\gamma}{4}\right)\sum_{t=1}^{T} \|\m{x}_{t}-\m{x}_{t}^*\|
  +\frac{\gamma}{8L_{\m{y}}} \left(  \|\m{y}_{1} - \m{y}_{1}^*(\m{x}_1)\|+  Y_{1,T}\right),
\end{align*}
where the second inequality is by  $\alpha \leq \frac{\mu_f}{128M_f^2 L_{\m{y}}^2}$ and $\rho \leq \frac{\gamma}{1+\gamma}$.
\\
Plugging this into \eqref{cvxv} yields 
\begin{align*}
\nonumber \sum_{t=1}^{T} \|\m{x}_{t}-\m{x}^*_{t}\|&\leq \|\m{x}_{1}-\m{x}^*_{1}\|+\left(1-\frac{\gamma}{4}\right) \sum_{t=1}^{T}\|\m{x}_{t}-\m{x}_{t}^*\|\\
  &+\frac{\gamma}{8L_{\m{y}}} \left(  \|\m{y}_{1} - \m{y}_{1}^*(\m{x}_1)\|+   Y_{1,T}\right)+P_{1,T}.
\end{align*}
Rearranging terms in the above inequality finishes the proof.

We now show part \ref{vgg} of the lemma.

First, note that
\begin{align}\label{nnk}
\nonumber  \sum_{t=1}^{T} \|\m{x}_{t}-\m{x}^*_{t}\|^2&=\|\m{x}_{1}-\m{x}^*_{1}\|^2+\sum_{t=2}^{T} \|\m{x}_{t}-\m{x}^*_{t}\|^2
 \\\nonumber &\leq \|\m{x}_{1}-\m{x}^*_{1}\|^2+(1+\frac{\gamma}{2} )\sum_{t=2}^{T}\|\m{x}_{t}- \m{x}_{t-1}^*\|^2+(1+\frac{2}{\gamma} )\sum_{t=2}^{T}\|\m{x}_{t-1}^*-\m{x}_{t}^*\|^2
 \\&\leq \|\m{x}_{1}-\m{x}^*_{1}\|^2+(1+\frac{\gamma}{2} )\sum_{t=1}^{T}\|\m{x}_{t+1}- \m{x}_{t}^*\|^2+(1+\frac{2}{\gamma} )P_{2,T},
\end{align}
where the first inequality follows from Lemma \ref{lem:trig}.

Similar to the previous case, we provide an upper bound for the second term on the right-hand side of \eqref{nnk}.

Since our assumption on $\alpha$ in the statement of Lemma~\ref{lm:xxxb} satisfies the requirement of  Lemma \ref{lm:xxx}, we have
\begin{align*}
 \nonumber   \| \m{x}_{t+1} - \m{x}_{t}^* \|^2&\leq (1-\gamma) \|\m{x}_{t}- \m{x}_{t}^*\|^2
  +\frac{2M_f^2\alpha}{\big(1+\frac{\mu_f}{2}\alpha \big)\mu_f} \norm{\m{y}_{t+1} -\m{y}^*_{t}(\m{x}_{t})}^2,
\end{align*}
which implies
\begin{align*}
\nonumber (1+\frac{\gamma}{2} )\|\m{x}_{t+1}- \m{x}_{t}^*\|^2 &\leq (1+\frac{\gamma}{2} )(1-\gamma) \|\m{x}_{t}- \m{x}_{t}^*\|^2
  +(1+\frac{\gamma}{2} )\frac{2M_f^2\alpha}{\big(1+\frac{\mu_f}{2}\alpha \big)\mu_f} \norm{\m{y}_{t+1} -\m{y}^*_{t}(\m{x}_{t})}^2
  \\&\leq (1-\frac{\gamma}{2} ) \|\m{x}_{t}- \m{x}_{t}^*\|^2
  +(1+\frac{\gamma}{2} )\frac{2M_f^2\alpha}{\mu_f} \norm{\m{y}_{t+1} -\m{y}^*_{t}(\m{x}_{t})}^2,
\end{align*}
where the second inequality is due to $(1+a/2)(1-a)\leq (1-a/2-a^2/2)\leq 1-a/2$.

Summing both sides of the above inequality from $t=1$ to $T$, we get
\begin{align}\label{ss}
 \nonumber  (1+\frac{\gamma}{2} )\sum_{t=1}^T\|\m{x}_{t+1}- \m{x}_{t}^*\|^2 &\leq   (1-\frac{\gamma}{2} )\sum_{t=1}^T\|\m{x}_{t}- \m{x}_{t}^*\|^2
  +(1+\frac{\gamma}{2} )\frac{2M_f^2\alpha}{\mu_f} \sum_{t=1}^T\norm{\m{y}_{t+1} -\m{y}^*_{t}(\m{x}_{t})}^2
  \\&\leq   (1-\frac{\gamma}{2} )\sum_{t=1}^T\|\m{x}_{t}- \m{x}_{t}^*\|^2
  +\frac{3M_f^2\alpha}{\mu_f} \sum_{t=1}^T\norm{\m{y}_{t+1} -\m{y}^*_{t}(\m{x}_{t})}^2,
\end{align}
where the second inequality follows since $\gamma =\frac{3\mu_f}{\frac{2}{\alpha}+\mu_f}\leq 1$; see \eqref{eqn:gamma}.

Since our assumption on $\rho$ in the statement of Lemma~\ref{lm:xxxb}--\ref{vgg} satisfies the requirement of  Lemma \ref{eq:lower:dynamic}--\ref{item:leml1} , we have 
\begin{align*}
\nonumber \sum_{t=1}^T \|  \m{y}_{t+1} -  \m{y}^*_t(\m{x}_t) \|^2  \leq &  \frac{\rho^2_1}{1-2\rho^2_1} \|\m{y}_{1} - \m{y}_{1}^*(\m{x}_1)\|^2 \\
+ \frac{6}{1-2\rho^2_1} & \left(2 L_{\m{y}}^2\sum_{t=1}^{T}\rho_{t}^2\|\m{x}_{t}-\m{x}^*_{t}\|^2 + \sum_{t=2}^T\rho_t^2 \|\m{y}^*_{t-1}(\m{x}^*_{t-1})-\m{y}^*_{t}(\m{x}^*_{t})\|^2 \right).
\end{align*}
This together with \eqref{ss} gives
\begin{align}\label{ssc}
 \nonumber   (1+\frac{\gamma}{2} )\sum_{t=1}^T \| \m{x}_{t+1} - \m{x}_{t}^* \|^2&\leq   \left(1-\frac{\gamma}{2} +\frac{36M_f^2L_{\m{y}}^2\alpha}{\mu_f} \frac{\rho^2_t}{(1-2\rho^2_1)}\right)\sum_{t=1}^T\|\m{x}_{t}- \m{x}_{t}^*\|^2
  \\&
  +\frac{3M_f^2\alpha}{\mu_f} \left(\frac{\rho^2_1}{1-2\rho^2_1} \|\m{y}_{1} - \m{y}_{1}^*(\m{x}_1)\|^2+ \frac{6}{1-2\rho^2_1}\sum_{t=2}^T\rho_t^2 \|\m{y}^*_{t-1}(\m{x}^*_{t-1})-\m{y}^*_{t}(\m{x}^*_{t})\|^2\right).
\end{align} 
Since by our setting $\rho_t=\rho \leq \frac{\sqrt{\gamma}}{\sqrt{2}\sqrt{\gamma+1}}$, we have
\begin{align*}
 \frac{36M_f^2L_{\m{y}}^2\alpha}{\mu_f}  \frac{\rho^2}{(1-2\rho^2)}\leq 
 \frac{36M_f^2L_{\m{y}}^2\alpha}{\mu_f}  \frac{\gamma}{2}\leq \frac{9\gamma}{64},
\end{align*}
where the second inequality uses our assumption on the stepsize, i.e., $\alpha \leq \frac{\mu_f}{128 M_f^2 L_{\m{y}}^2}$.

Combining the above inequality with \eqref{ssc}, we obtain 
\begin{align}\label{gg}
 \nonumber  (1+\frac{\gamma}{2} )\sum_{t=1}^T \| \m{x}_{t+1} - \m{x}_{t}^* \|^2&\leq   \left(1-\frac{23\gamma}{64} \right)\sum_{t=1}^T\|\m{x}_{t}- \m{x}_{t}^*\|^2
  \\\nonumber&
  +\frac{3M_f^2\alpha}{\mu_f} \left(\frac{\rho^2}{1-2\rho^2} \|\m{y}_{1} - \m{y}_{1}^*(\m{x}_1)\|^2+ \frac{6\rho^2}{1-2\rho^2}Y_{2,T}\right)
  \\&\leq   \left(1-\frac{23\gamma}{64} \right)\sum_{t=1}^T\|\m{x}_{t}- \m{x}_{t}^*\|^2
  +\frac{3\gamma}{256L_{\m{y}}^2} \left( \|\m{y}_{1} - \m{y}_{1}^*(\m{x}_1)\|^2+ 6Y_{2,T}\right),
\end{align}
where the second inequality follows from 
$\alpha \leq \frac{\mu_f}{128 M_f^2 L_{\m{y}}^2}$ and
$\rho \leq \frac{\sqrt{\gamma}}{\sqrt{2}\sqrt{\gamma+1}}$.
\\
Further, plugging \eqref{gg} into \eqref{nnk} yields:
\begin{align*}
\nonumber \sum_{t=1}^{T} \|\m{x}_{t}-\m{x}^*_{t}\|^2&\leq \|\m{x}_{1}-\m{x}^*_{1}\|^2+\left(1-\frac{23\gamma}{64} \right)\sum_{t=1}^T\|\m{x}_{t}- \m{x}_{t}^*\|^2
  +(1+\frac{2}{\gamma} )P_{2,T}
  \\&
  +\frac{3\gamma}{256L_{\m{y}}^2} \left( \|\m{y}_{1} - \m{y}_{1}^*(\m{x}_1)\|^2+ 6Y_{2,T}\right).
\end{align*}
Rearranging the inequality gives \ref{vgg}.
\end{proof}

\subsubsection{Proof of Theorem~\ref{thm:dynamic:strong:B-OGD}}\label{sec:app:dreg:up}
\begin{proof}
Assumption \ref{assu:f:a1} implies that $\| \nabla f_{t}(\m{x}, \m{y}(\m{x}))\|\leq \ell_{f,0}$ for any $t\in [T]$ and any $\m{x}\in \mathcal{X}$. Thus, we get
\begin{align*}
\sum_{t=1}^{T}\left( f_{t}(\m{x}_{t},\m{y}^*_t(\m{x}_{t})) - f_{t}(\m{x}_t^*,\m{y}^*_t(\m{x}_{t}^*))\right) 
&
\leq   \ell_{f,0}\sum_{t=1}^{T}  \|\m{x}_{t}-\m{x}^*_{t}\|.
\end{align*}
Note that our choices of the stepsize $\alpha_t$ and $K_t$ in the theorem statement  can be rewritten as 
\begin{align}\label{kkk}
\nonumber 
\alpha_t&=\alpha \leq \min\left\{\frac{1}{\ell_{f,1}}, \frac{\mu_f}{128M_f^2 L_{\m{y}}^2}\right\},\quad \textnormal{and}\quad
\\
K_t&>\left\lceil\frac{(\kappa_g+1)\log \rho^{-2}_t}{4}\right\rceil,\quad \textnormal{with}\quad \rho_t=\rho \leq \frac{\gamma}{3(1+\gamma)}.
\end{align}
These choices satisfy the condition of Lemma \ref{lm:xxxb}--\ref{vff}. Hence, from Lemma \ref{lm:xxxb}--\ref{vff}, we get
\begin{align*}
 \nonumber  \sum_{t=1}^{T} \|\m{x}_{t}-\m{x}^*_{t}\|
\leq \frac{4}{\gamma}\left(\|\m{x}_{1}-\m{x}^*_{1}\|+P_{1,T}\right)
  +\frac{1}{2L_{\m{y}}} \left(  \|\m{y}_{1} - \m{y}_{1}^*(\m{x}_1)\|+   Y_{1,T}\right),
\end{align*}
which, in conjunction with Assumption~\ref{assu:dom}, implies
\begin{align}\label{bnp}
\nonumber \sum_{t=1}^{T}\left( f_{t}(\m{x}_{t},\m{y}^*_t(\m{x}_{t})) - f_{t}(\m{x}_t^*,\m{y}^*_t(\m{x}_{t}^*))\right)
&\leq \frac{4\ell_{f,0}}{\gamma}\left(\|\m{x}_{1}-\m{x}^*_{1}\|+P_{1,T}\right)
 +\frac{\ell_{f,0}}{2L_{\m{y}}} \left( \|\m{y}_{1} - \m{y}_{1}^*(\m{x}_1)\|+   Y_{1,T}\right)
\\\nonumber &\leq \frac{4\ell_{f,0}}{\gamma}\left(D+P_{1,T}\right)
 +\frac{\ell_{f,0}}{2L_{\m{y}}} \left( D'+   Y_{1,T}\right)
\\&=\mc{O}\left(1+   S_{1,T}\right).
\end{align}
In the following, we show that the dynamic regret can also be upper bounded by $ S_{2,T}=P_{2,T}+ Y_{2,T}$.

It follows from Lemma \ref{lem:lips} that 
    \begin{align}\label{eqn:prthm:non1x}
\nonumber  f_{t}(\m{x}_{t}, \m{y}^*_t(\m{x}_{t})) - f_{t}(\m{x}_t^*,\m{y}^*_t(\m{x}_{t}^*))
&\leq \left\langle \nabla f_{t}(\m{x}_t^*, \m{y}^*_t(\m{x}_{t}^*)), \m{x}_{t} - \m{x}_t^* \right\rangle
+ \frac{L_f}{2} \| \m{x}_{t} - \m{x}_t^* \|^2
\\&\leq \frac{1}{2}\| \nabla f_{t}(\m{x}_t^*, \m{y}^*_t(\m{x}_{t}^*))\|^2+\frac{1}{2}( 1+L_f )\| \m{x}_{t} - \m{x}_t^* \|^2.
    \end{align}
Summing the inequality \eqref{eqn:prthm:non1x} over $t\in [T]$, we get
    \begin{align}\label{nkmP}
\nonumber  &\quad \sum_{t=1}^{T}\left( f_{t}(\m{x}_{t},\m{y}^*_t(\m{x}_{t})) - f_{t}(\m{x}_t^*,\m{y}^*_t(\m{x}_{t}^*))\right)
   \\ &\leq  \frac{1}{2}\sum_{t=1}^{T}\| \nabla f_{t}(\m{x}_t^*, \m{y}^*_t(\m{x}_{t}^*))\|^2
   +\frac{1}{2}( 1+L_f )\sum_{t=1}^{T}\| \m{x}_{t} - \m{x}_t^* \|^2.
    \end{align}
The choices of $\alpha_t$ and $K_t$ in \eqref{kkk} satisfy the condition of Lemma \ref{lm:xxxb}--\ref{vgg} as well. Hence, from Lemma \ref{lm:xxxb}--\ref{vgg} and Assumption~\ref{assu:dom}, we get
\begin{align}\label{gfz}
 \nonumber
 \sum_{t=1}^{T} \|\m{x}_{t}-\m{x}^*_{t}\|^2&\leq \frac{64}{23\gamma}\left(\|\m{x}_{1}-\m{x}^*_{1}\|^2
  +(1+\frac{2}{\gamma} )P_{2,T}\right)\\\nonumber
  &+\frac{3}{92L_{\m{y}}^2}\left( \|\m{y}_{1} - \m{y}_{1}^*(\m{x}_1)\|^2+ 6Y_{2,T}\right)
  \\&\leq \frac{64}{23\gamma}\left(D^2
  +(1+\frac{2}{\gamma} )P_{2,T}\right)+\frac{3}{92L_{\m{y}}^2}\left( {D'}^2+ 6Y_{2,T}\right).
\end{align}
Putting together \eqref{gfz} and \eqref{nkmP}, we get
\begin{align}\label{cdf}
\nonumber  &\quad \sum_{t=1}^{T}\left( f_{t}(\m{x}_{t},\m{y}^*_t(\m{x}_{t})) - f_{t}(\m{x}_t^*,\m{y}^*_t(\m{x}_{t}^*))\right)
  \\\nonumber &   \leq  \frac{1}{2}\sum_{t=1}^{T}\| \nabla f_{t}(\m{x}_t^*, \m{y}^*_t(\m{x}_{t}^*))\|^2
 +\frac{32}{23\gamma}( 1+L_f )  \left(D^2
  +(1+\frac{2}{\gamma} )P_{2,T}\right)
  \\\nonumber&+\frac{3}{184L_{\m{y}}^2}( 1+L_f )\left( {D'}^2+ 6Y_{2,T}\right) 
\\&=\mc{O}\left(1+  \sum^T_{t=1} \norm{\nabla f_t(\m{x}^*_t, \m{y}^*_t(\m{x}^*_t) )}^2 + P_{2,T}+ Y_{2,T}\right).
\end{align}
Now, from \eqref{bnp} and \eqref{cdf}, we have
\begin{align*}
  &\quad \sum_{t=1}^{T}\left( f_{t}(\m{x}_{t},\m{y}^*_t(\m{x}_{t})) - f_{t}(\m{x}_t^*,\m{y}^*_t(\m{x}_{t}^*))\right)
  \\& \leq   \mc{O} \left(1+\min\big\{S_{1,T},\sum^T_{t=1} \norm{\nabla f_t(\m{x}^*_t, \m{y}^*_t(\m{x}^*_t) )}^2+S_{2,T}\big\}\right).
\end{align*}
This completes the proof.
\end{proof}
\subsubsection{Proof of Theorem~\ref{thm:static:strong:B-OGD}}\label{sec:app:reg:up}

\begin{proof}
Recall the update rule of Algorithm \ref{alg:obgd} (with $w=1$): $\m{x}_{t+1} = \Pi_{\mc{X}}\left[\m{x}_t - \alpha_t\tilde{\nabla} f_t(\m{x}_t,\m{y}_{t+1})\right]$. From the Pythagorean theorem, we get
\begin{equation*}\label{eqn1:thm1:oagd}
\begin{aligned}
 \norm{\m{x}_{t+1} - \m{x}^*}^2 &\leq \norm{ \m{x}_t - \alpha_t \tilde{\nabla} f_t(\m{x}_t,\m{y}_{t+1}) - \m{x}^*}^2
\\
&=\norm{\m{x}_{t} - \m{x}^*}^2- 2\alpha_t \langle\tilde{\nabla} f_t(\m{x}_t, \m{y}_{t+1}),\m{x}_{t} - \m{x}^*\rangle + \alpha_t^2 \norm{\tilde{\nabla} f_t(\m{x}_t, \m{y}_{t+1})}^2
\\&\leq \norm{\m{x}_{t} - \m{x}^*}^2- 2\alpha_t \langle\tilde{\nabla} f_t(\m{x}_t, \m{y}_{t+1}),\m{x}_{t} - \m{x}^*\rangle 
\\&+2 \alpha_t^2 \norm{\nabla f_t(\m{x}_t,\m{y}^*_t(\m{x}_t))}^2+2\alpha_t^2 \norm{\tilde{\nabla} f_t(\m{x}_t,  \m{y}_{t+1})-\nabla f_t(\m{x}_t,\m{y}^*_t(\m{x}_t))}^2,
\end{aligned}
\end{equation*}
where the second inequality uses Lemma \ref{lem:trig} with $c=1$.
\\
Rearranging the above inequality yields
\begin{equation}\label{eq1:static:strong}
\begin{split}
\left\langle\nabla f_t(\m{x}_t, \m{y}^*_t(\m{x}_t)),\m{x}_t - \m{x}^*\right\rangle &\leq  \frac{1}{2\alpha_t} \norm{\m{x}_t - \m{x}^*}^2 - \frac{1}{2\alpha_t} \norm{\m{x}_{t+1} - \m{x}^*}^2+\alpha_t \norm{\nabla f_t(\m{x}_t, \m{y}^*_t(\m{x}_t))}^2 \\
  &+ \alpha_t \norm{\nabla f_t(\m{x}_t,\m{y}^*_t(\m{x}_t))-\tilde{\nabla} f_t(\m{x}_t,  \m{y}_{t+1})}^2
\\
&+ \left\langle  \nabla f_t(\m{x}_t,\m{y}^*_t(\m{x}_t))- \tilde{\nabla} f_t(\m{x}_t,  \m{y}_{t+1}), \m{x}_t-\m{x}^* \right \rangle.
\end{split}
\end{equation}
From Lemma~\ref{lem:lips}, for any $c >0$,  we have
\begin{equation}
\begin{aligned}
\label{eq2:static:strong}
\left\langle  \nabla f_t(\m{x}_t,\m{y}^*_t(\m{x}_t))- \tilde{\nabla} f_t(\m{x}_t,  \m{y}_{t+1}),\m{x}_t- \m{x}^* \right \rangle 
&\leq \frac{c}{2}\norm{\m{x}_t-\m{x}^*  }^2+\frac{1}{2c}\norm{\nabla f_t(\m{x}_t,\m{y}^*_t(\m{x}_t))- \tilde{\nabla} f_t(\m{x}_t,  \m{y}_{t+1}) }^2
\\&\leq \frac{c}{2}\norm{\m{x}_t - \m{x}^*}^2+ \frac{ M^2_{f}}{2c} \norm{\m{y}_{t+1} -\m{y}^*_t(\m{x}_t)}^2,
\\\norm{\nabla f_t(\m{x}_t,\m{y}^*_t(\m{x}_t))-\tilde{\nabla} f_t(\m{x}_t,  \m{y}_{t+1})}^2 &\leq M_f^2 \norm{\m{y}_{t+1} -\m{y}^*_t(\m{x}_t)}^2. 
\end{aligned}
\end{equation}
Combining \eqref{eq1:static:strong} and \eqref{eq2:static:strong}, we get
\begin{equation}\label{eq3:static:strong}
\begin{split}
\langle \nabla f_t(\m{x}_t, \m{y}^*_t(\m{x}_t)),\m{x}_t - \m{x}^*\rangle &\leq  \frac{1}{2}\left(\frac{1}{\alpha_t} +c\right)\norm{\m{x}_t - \m{x}^*}^2
- \frac{1}{2\alpha_t}\norm{\m{x}_{t+1} - \m{x}^*}^2 
\\
&+\alpha_t \norm{\nabla f_t(\m{x}_t, \m{y}^*_t(\m{x}_t))}^2 + M_f^2 \left(\alpha_t+\frac{1}{2c}\right) \norm{\m{y}_{t+1} -\m{y}^*_t(\m{x}_t)}^2.
\end{split}
\end{equation}
Applying the definition of $\mu_f$-strong convexity to the pair of points $\{\m{x}_t$,$\m{x}^*\}$, we have
\begin{equation}\label{eq4:static:strong}
\begin{split}
2\left(f_t(\m{x}_t,\m{y}^*_t(\m{x}_t)) - f_t(\m{x}^*,\m{y}^*_t(\m{x}^*))\right) &\leq 2 \left\langle \nabla f_t(\m{x}_t, \m{y}^*_t(\m{x}_t)),\m{x}_t-\m{x}^*\right\rangle-\mu_f\norm{\m{x}_t-\m{x}^*}^2.
\end{split}
\end{equation}
From \eqref{eq3:static:strong} and \eqref{eq4:static:strong}, we get
\begin{equation*}\label{eq:dyanamic:strong}
\begin{split}
2\left(f_t(\m{x}_t,\m{y}^*_t(\m{x}_t)) - f_t(\m{x}^*,\m{y}^*_t(\m{x}^*))\right) &\leq 2 \langle\nabla f_t(\m{x}_t, \m{y}^*_t(\m{x}_t)),\m{x}_t-\m{x}^*\rangle-\mu_f \norm{\m{x}_t-\m{x}^*}^2 \\
&\leq  \left(\frac{1}{\alpha_t}+ c- \mu_f\right) \norm{\m{x}_t - \m{x}^*}^2
- \frac{1}{\alpha_t}\norm{\m{x}_{t+1} - \m{x}^*}^2 
\\
&+2\alpha_t \norm{\nabla f_t(\m{x}_t, \m{y}^*_t(\m{x}_t))}^2 +2 M_f^2\left( \alpha_t+\frac{1}{c}\right) \norm{\m{y}_{t+1} -\m{y}^*_t(\m{x}_t)}^2.
\end{split}
\end{equation*}
Summing \label{eq5:static:strong} from $t= 1$ to $T$, we have
\begin{align}\label{eqn:reg:bound:in:stat}
\nonumber&\quad  2 \sum_{t=1}^T  \left(f_t(\m{x}_t,\m{y}^*_t(\m{x}_t)) - f_t(\m{x}^*,\m{y}^*_t(\m{x}^*)) \right)\\\nonumber
&\leq  \left(\frac{1}{\alpha_1}+ c- \mu_f\right) \norm{\m{x}_1 - \m{x}^*}^2+\sum_{t=2}^T \|\m{x}_t-\m{x}^*\|^2
\left(\frac{1}{\alpha_{t}}-\frac{1}{\alpha_{t-1}} +c-\mu_f\right) \\
&+2
\sum_{t=1}^{T}  \alpha_t \norm{\nabla f_t(\m{x}_t, \m{y}^*_t(\m{x}_t))}^2
+2 M_f^2 \sum_{t=1}^{T} \left( \alpha_t+\frac{1}{c}\right)\norm{\m{y}_{t+1} -\m{y}^*_t(\m{x}_t)}^2.
\end{align}
Next, we bound the last term in the right-hand side of \eqref{eqn:reg:bound:in:stat}.
To proceed, note that our choice of $K_t$ as   
\begin{equation*}
    K_t>\left\lceil\frac{(\kappa_g+1)\log \rho^{-2}_t}{4}\right\rceil, \quad \textnormal{with}\quad \rho_t=\rho\leq \frac{1}{\sqrt{2}}\sqrt{\frac{\theta}{1+\theta}},\quad\textnormal{and}\quad \theta:=\frac{c}{12 L_{\m{y}}^2M_f^2(\alpha_1+\frac{1}{c})}
\end{equation*}
satisfies the condition required in Lemma \ref{eq:lower:dynamic}--\ref{item:leml1}. Hence, 
from  Lemma \ref{eq:lower:dynamic}--\ref{item:leml1}, we obtain
\begin{align}\label{eqn:nm}
\nonumber&\quad 2M_f^2\sum_{t=1}^{T}  \left( \alpha_t+\frac{1}{c}\right) \norm{\m{y}_{t+1} -\m{y}^*_t(\m{x}_t)}^2  
\\&\leq 2M_f^2\left( \alpha_1+\frac{1}{c}\right) \frac{6\rho^2}{1-2\rho^2}  \left(\frac{1}{6}\norm{\m{y}_{1} - \m{y}_{1}^*(\m{x}_1)}^2  +  2  L_{\m{y}}^2 \sum_{t=1}^{T} \norm{\m{x}_{t}-\m{x}^*}^2  + \bar{Y}_{2,T} \right).
\end{align}
Since $\rho\leq \frac{1}{\sqrt{2}}\sqrt{\frac{\theta}{1+\theta}}$, we have
\begin{align*}
  4 L_{\m{y}}^2M_f^2\left(\alpha_1+\frac{1}{c}\right)  \frac{6\rho^2}{1-2\rho^2} \leq c,
\end{align*}
which, in conjunction with Eq. \eqref{eqn:nm}, yields
\begin{align}\label{eqn:bound:in:stat}
\nonumber
&\quad 2M_f^2\sum_{t=1}^{T}  \left( \alpha_t+\frac{1}{c}\right) \norm{\m{y}_{t+1} -\m{y}^*_t(\m{x}_t)}^2  \\&\leq  \frac{c}{12L_{\m{y}}^2} \norm{\m{y}_{1} - \m{y}_{1}^*(\m{x}_1)}^2+c \sum_{t=1}^{T} \norm{\m{x}_{t}-\m{x}^*}^2      +  \frac{c}{2L_{\m{y}}^2} \bar{Y}_{2,T}.
\end{align}
Thus, combining \eqref{eqn:bound:in:stat} and \eqref{eqn:reg:bound:in:stat} and using Assumption \ref{assu:f:a1}, we obtain
\begin{align}\label{eqn:reg:bound:in:stat2}
\nonumber  2 \sum_{t=1}^T  \left(f_t(\m{x}_t,\m{y}^*_t(\m{x}_t)) - f_t(\m{x}^*,\m{y}^*_t(\m{x}^*)) \right)
 &\leq   \left(\frac{1}{\alpha_1}+ c- \mu_f\right) \norm{\m{x}_1 - \m{x}^*}^2\\
 \nonumber 
 &+\sum_{t=2}^T \|\m{x}_t-\m{x}^*\|^2
 \left(\frac{1}{\alpha_{t}}-\frac{1}{\alpha_{t-1}} + 2c -\mu_f\right) \\
 &+2\ell_{f,0}^2\sum_{t=1}^{T}  \alpha_t+\frac{c}{12L_{\m{y}}^2} \|\m{y}_{1} - \m{y}_{1}^*(\m{x}_1)\|^2  +  \frac{c}{2L_{\m{y}}^2} \bar{Y}_{2,T} .
\end{align}
By setting $c=\mu_f/4$ and  $\alpha_{t} = 2/(\mu_f t)$ and utilizing Assumption~\ref{assu:dom}, we have
\begin{align}\label{eqn:pl}
   2 \sum_{t=1}^T  \left(f_t(\m{x}_t,\m{y}^*_t(\m{x}_t)) - f_t(\m{x}^*,\m{y}^*_t(\m{x}^*)) \right)
 &\leq \frac{4\ell_{f,0}^2}{\mu_f}\sum_{t=1}^{T}  \frac{1}{t}+\frac{\mu_f}{48L_{\m{y}}^2} {D'}^2  +  \frac{\mu_f}{8L_{\m{y}}^2} \bar{Y}_{2,T}.
\end{align}
Let
\begin{equation}\label{vvz}
    \begin{split}
e_1&:=\frac{\mu_f}{48L_{\m{y}}^2} {D'}^2,\quad
e_2:= \frac{\mu_f}{8L_{\m{y}}^2}, ~~~~~e_3:= \frac{4\ell_{f,0}^2}{\mu_f}.
    \end{split}
\end{equation}
Combining  Lemma~\ref{lem:rimann}--\ref{itm:lem:a:rimann} and \eqref{vvz}  with \eqref{eqn:pl}, we obtain
\begin{equation*}\label{eq:app:bsreg:bound:strong}
\begin{split}
\textnormal{BS-Reg}_T \leq e_3\log T+ e_2 \bar{Y}_{2,T}+e_1.
\end{split}
\end{equation*}
\end{proof}

\begin{cor}\label{cor:dynamic:strong:B-OGD}
Under the same setting as Theorem~\ref{thm:dynamic:strong:B-OGD},
\begin{enumerate} [label=\textnormal{\textbf{(\Roman*)}},wide, labelindent=0pt]
\item  If function $f_t$ is non-negative for each $t \in [T]$, then

\begin{align}\label{eqn:ref:str:non:bounds0}
 \textnormal{BD-Reg}_T
   \leq   \mc{O} \left(1+\min\{S_{1,T}, F_{T} +S_{2,T}\}\right),
\end{align}
where $F_{T} :=  \sum_{t=1}^{T} f_t(\m{x}^*_{t}, \m{y}^*_t(\m{x}^*_{t}))$.
\item \label{ref:str:bounds} If $\sum^T_{t=1} \norm{\nabla f_t(\m{x}^*_t, \m{y}^*_t(\m{x}^*_t) )}=\mc{O}(S_{2,T})$, then
\begin{align}\label{eqn:ref:str:bounds}
  \textnormal{BD-Reg}_T \leq  \mc{O} \left(1+\min\{ S_{1,T},S_{2,T}\}\right).
\end{align}
\end{enumerate}
\end{cor}
\begin{proof}
\begin{enumerate}[label=\textnormal{(\roman*)}]
\item If $f_t \geq 0$ for all $t \in[T]$, then it follows from Lemma~\ref{lem:nonneg:smooth} that $ \norm{\nabla f_t(\m{x}^*_t, \m{y}^*_t(\m{x}^*_t) )} \leq \sqrt{4 \ell_{f,1}  f_t(\m{x}^*_t, \m{y}^*_t(\m{x}^*_t) )}$. This together with \eqref{eq:bdreg:bound:strong} gives the desired result.
\item If $\m{x}_t^*\in \argmin_{\m{x} \in \mc{X}} f_t(\m{x})$ for all $t \in[T]$ and the minimizers $\{\m{x}_t^*\}_{t=1}^T$ lie in the interior of the domain $\mc{X}$, we have $\norm{\nabla f_t(\m{x}^*_t, \m{y}^*_t(\m{x}^*_t) )}=0$ which together with \eqref{eq:bdreg:bound:strong} gives the desired result.  If $\sum_{t=1}^{T} \norm{\nabla f_t(\m{x}^*_t, \m{y}^*_t(\m{x}^*_t) )}= \mc{O}( S_{2,T})$, then \eqref{eqn:ref:str:bounds} follows from \eqref{eq:bdreg:bound:strong}.
\end{enumerate}
\end{proof}

Corollary~\ref{cor:dynamic:strong:B-OGD} naturally interpolates between single-level and bilevel regret.  In the case when $Y_{1,T}=Y_{2,T}=0$, Eq.~\eqref{eqn:ref:str:non:bounds0} gives a single-level regret for strongly convex, smooth, and non-negative losses, similar to \citep{srebro2010smoothness,zhao2020dynamic}. We note that if the minimizers $\{\m{x}_t^*\}_{t=1}^T$ lie in the interior of the domain $\mc{X}$, we have $\norm{\nabla f_t(\m{x}^*_t, \m{y}^*_t(\m{x}^*_t) )} =0$ for all $t \in [T]$, which implies the  $\mc{O}\left(1+\min\{ S_{1,T},S_{2,T}\}\right)$ regret bound.

\subsection{Proof for Convex OBO with Partial Information}\label{sec:app:reg:partial}

\subsubsection{Proof of Theorem~\ref{thm:dynamic:convex:B-OGD}}\label{sec:app:convex}
\begin{proof}
From the update rule of Algorithm \ref{alg:obgd} (with $w=1$), we have $\m{x}_{t+1} = \Pi_{\mc{X}}\left[\m{x}_t - \alpha\tilde{\nabla} f_t(\m{x}_t,\m{y}_{t+1})\right]$. Now, from the Pythagorean theorem, we get
\begin{align}\label{eqn:conve:pyht}
\nonumber
\frac{1}{2}\norm{\m{x}_{t+1} - \m{x}^*_t}^2& \leq  \frac{1}{2}\norm{\m{x}_t - \alpha   \tilde{\nabla} f_t(\m{x}_t,\m{y}_{t+1})  - \m{x}^*_t}^2 \\\nonumber
&=\frac{1}{2}\norm{\m{x}_t - \m{x}^*_t}^2 - \alpha  \langle \tilde{\nabla} f_t(\m{x}_t,\m{y}_{t+1}),\m{x}_t - \m{x}^*_t\rangle + \frac{\alpha^2}{2}\norm{  \tilde{\nabla} f_t(\m{x}_t,\m{y}_{t+1})}^2
\\\nonumber&\leq \frac{1}{2}\norm{\m{x}_t - \m{x}^*_t}^2 - \alpha  \langle \tilde{\nabla} f_t(\m{x}_t,\m{y}_{t+1}),\m{x}_t - \m{x}^*_t\rangle + \alpha^2\norm{  \nabla f_t(\m{x}_t,\m{y}^*_t(\m{x}_t))}^2
\\&+ \alpha^2\norm{\tilde{\nabla} f_t(\m{x}_t,  \m{y}_{t+1})-\nabla f_t(\m{x}_t,\m{y}^*_t(\m{x}_t))}^2.
\end{align}
Here, the second inequality holds because of the Lemma \ref{lem:trig} by setting $c=1$.
\\
Rearranging the above inequality and summing over $t\in [T]$, we obtain
\begin{subequations}\label{eqn:convex:reg1}
\begin{align}
\sum_{t=1}^{T}\left\langle \nabla f_t(\m{x}_t,\m{y}^*_t(\m{x}_t)),\m{x}_t - \m{x}^*_t\right\rangle&\leq \sum_{t=1}^{T} \left(\frac{1}{2\alpha}\norm{\m{x}_t - \m{x}^*_t}^2 - \frac{1}{2\alpha}\norm{\m{x}_{t+1} - \m{x}^*_t}^2 \right)\label{eqn:convex:reg11}\\
&+ \alpha \sum_{t=1}^{T}\norm{\nabla f_t(\m{x}_t,\m{y}^*_t(\m{x}_t))-\tilde{\nabla} f_t(\m{x}_t,  \m{y}_{t+1})}^2\label{eqn:convex:reg12}\\
&+\sum_{t=1}^{T}\left\langle  \nabla f_t(\m{x}_t,\m{y}^*_t(\m{x}_t))- \tilde{\nabla} f_t(\m{x}_t,  \m{y}_{t+1}), \m{x}_t-\m{x}^*_t \right\rangle
\label{eqn:convex:reg13}\\
&+\alpha\sum_{t=1}^{T}\norm{ \nabla f_t(\m{x}_t,\m{y}^*_t(\m{x}_t)) }^2.
\label{eqn:convex:reg14s}
\end{align}
\end{subequations}
Next, we upper bound each term of \eqref{eqn:convex:reg1}.
\\
$\bullet$~\textbf{Bounding \eqref{eqn:convex:reg11}:}
Observe that
\begin{subequations}
\begin{align}\label{eqn:convex:reg2}
\nonumber  \eqref{eqn:convex:reg11}
 & \leq\frac{1}{2\alpha}\|\m{x}_1 - \m{x}^*_1\|^2-\frac{1}{2\alpha}\|\m{x}_{T+1} - \m{x}^*_T\|^2+\frac{1}{2\alpha}\sum_{t=2}^{T}\big(\|\m{x}_t - \m{x}^*_t\|^2 -\|\m{x}_{t} -\m{x}^*_{t-1}\|^2\big)
 \\\nonumber&\leq \frac{1}{2\alpha}\|\m{x}_1 - \m{x}^*_1\|^2+\frac{1}{2\alpha}\sum_{t=2}^{T}\|\m{x}^*_{t}-\m{x}_{t} +\m{x}^*_{t-1}- \m{x}_{t}\| \|\m{x}^*_{t}-\m{x}^*_{t-1} \|
 \\
& \leq \frac{D^2}{2\alpha} +\frac{D}{2\alpha}\sum_{t=2}^{T}\|\m{x}^*_{t}-\m{x}^*_{t-1} \|,
\end{align}
where the second inequality follows since
\begin{align*}
\|\m{x}^*_{t}-\m{x}_{t} \|^2 - \|\m{x}^*_{t-1}-\m{x}_{t} \|^2&=\left\langle \m{x}^*_{t}-\m{x}_{t}+\m{x}^*_{t-1}-\m{x}_{t},\m{x}^*_{t}- \m{x}_{t}-(\m{x}^*_{t-1} - \m{x}_{t}) \right\rangle
\\& \leq \|\m{x}^*_{t}-\m{x}_{t} +\m{x}^*_{t-1}- \m{x}_{t}\| \|\m{x}^*_{t}-\m{x}^*_{t-1} \|,
\end{align*}
and the last inequality follows from Assumption \ref{assu:dom}.
\\
$\bullet$~\textbf{Bounding \eqref{eqn:convex:reg12} and \eqref{eqn:convex:reg13}:}
It follows from Lemma~\ref{lem:lips} and Assumption~\ref{assu:dom} that
\begin{align*}
\sum_{t=1}^{T}\left\langle  \nabla f_t(\m{x}_t,\m{y}^*_t(\m{x}_t))- \tilde{\nabla} f_t(\m{x}_t,  \m{y}_{t+1}), \m{x}_t-\m{x}^*_t \right\rangle &\leq \sum_{t=1}^{T}\|\m{x}^*_t-\m{x}_t  \| \| \nabla f_t(\m{x}_t,\m{y}^*_t(\m{x}_t))- \tilde{\nabla} f_t(\m{x}_t,  \m{y}_{t+1})\|
\\&\leq    D M_f \sum_{t=1}^{T}\norm{\m{y}_{t+1} -\m{y}^*_t(\m{x}_t)}, \\
\sum_{t=1}^{T}\norm{\nabla f_t(\m{x}_t,\m{y}^*_t(\m{x}_t))-\tilde{\nabla} f_t(\m{x}_t,  \m{y}_{t+1})}^2 &\leq M_f^2 \sum_{t=1}^{T}\norm{\m{y}_{t+1} -\m{y}^*_t(\m{x}_t)}^2 .
\end{align*}
Hence,
\begin{align}\label{eqn:convex:reg3}
\eqref{eqn:convex:reg12} + \eqref{eqn:convex:reg13}  \leq    D M_f \sum_{t=1}^{T}\norm{\m{y}_{t+1} -\m{y}^*_t(\m{x}_t)} +\alpha M_f^2 \sum_{t=1}^{T}\norm{\m{y}_{t+1} -\m{y}^*_t(\m{x}_t)}^2 .
\end{align}
\\
$\bullet$~\textbf{Bounding \eqref{eqn:convex:reg14s}:}
By the smoothness of $\phi_t(\m{x)}=  f_t(\m{x},\m{y}^*_t(\m{x}))$, for any $\m{x}\in \mb{R}^{d_1}$, we have
\[
\phi_t(\m{x}) - \phi_t(\m{x}_t) \leq \langle \nabla \phi_t(\m{x}_t), \m{x} - \m{x}_t\rangle + \frac{L_f}{2}\|\m{x} - \m{x}_t\|^2.
\]
Let $\m{x}=\m{x}'_t = \m{x}_t - \frac{1}{L_f}\nabla \phi_t(\m{x}_t)$ in the above inequality, we have
$\phi_t(\m{x}'_t) - \phi_t(\m{x}_t)\leq  - \frac{\|\nabla \phi_t(\m{x}_t)\|^2}{2L_f}$.

It follows from the convexity of $f_t(\m{x}, \cdot )$ that
\[
\phi_t(\m{x}'_t)\geq \phi_t(\m{x}^*_t) + \langle\nabla \phi_t(\m{x}^*_t),\m{x}'_t - \m{x}^*_t\rangle=\phi_t(\m{x}^*_t),
\]
where the equality follows from the vanishing gradient condition ($\exists(\m{x}^*_{t}, \m{y}^*_t(\m{x}^*_{t})) \in \mc{X} \times \mb{R}^{d_2}$ such that $\nabla f_t(\m{x}^*_t, \m{y}^*_t(\m{x}^*_t) )=0 $ for  all $t \in [T]$).

Hence,
\[
\phi_t(\m{x}^*_t) - \phi_t(\m{x}_t)\leq \phi_t(\m{x}'_t) - \phi_t(\m{x}_t) \leq - \frac{\|\nabla \phi_t(\m{x}_t)\|^2}{2L_f},
\]
which implies that
\begin{equation}\label{eqn:convex:reg31}
\eqref{eqn:convex:reg14s} \leq    2\alpha L_f \sum_{t=1}^{T}\left(f_t(\m{x}_t,\m{y}^*_t(\m{x}_t)) - f_t(\m{x}^*_t,\m{y}^*_t(\m{x}^*_t))\right).
\end{equation}
\end{subequations}
\\
$\bullet$~\textbf{Bounding  $\sum_{t=1}^T \left\langle \nabla f_t(\m{x}_t, \m{y}^*_t(\m{x}_t)),\m{x}_t - \m{x}^*_t\right\rangle$}:  Substituting \eqref{eqn:convex:reg2}--\eqref{eqn:convex:reg31} into \eqref{eqn:convex:reg1},  we get
\begin{equation}
\begin{aligned}\label{eqn:fe}
\sum_{t=1}^T \left\langle \nabla f_t(\m{x}_t, \m{y}^*_t(\m{x}_t)),\m{x}_t - \m{x}^*_t\right\rangle 
&\leq \sum_{t=1}^T \Big(\alpha M_f^2  \norm{\m{y}_{t+1} -\m{y}^*_t(\m{x}_t)}^2 + D M_f   \norm{\m{y}_{t+1} -\m{y}^*_t(\m{x}_t)} 
\\
& +2\alpha L_f \left(f_t(\m{x}_t,\m{y}^*_t(\m{x}_t)) - f_t(\m{x}^*_t,\m{y}^*_t(\m{x}^*_t))\right)\Big) 
\\&+ \frac{D^2}{2\alpha}+\frac{D}{2\alpha}\sum_{t=2}^{T}\|\m{x}^*_{t}-\m{x}^*_{t-1} \| .
\end{aligned}
\end{equation}
$\bullet$~\textbf{Completing the proof of Theorem~\ref{thm:dynamic:convex:B-OGD}:}
By the convexity of $f_t$ and \eqref{eqn:fe}, we obtain
\begin{equation}\label{eqn:regbc:1}
\begin{aligned}
&\quad \sum_{t=1}^T \left( f_t(\m{x}_t,\m{y}^*_t(\m{x}_t)) - f_t(\m{x}^*_t,\m{y}^*_t(\m{x}^*_t))  \right)
\\&\leq \sum_{t=1}^T \left\langle \nabla f_t(\m{x}_t, \m{y}^*_t(\m{x}_t)),\m{x}_t - \m{x}^*_t\right\rangle
\\&\leq \sum_{t=1}^T \Big(\alpha M_f^2  \norm{\m{y}_{t+1} -\m{y}^*_t(\m{x}_t)}^2 + D M_f   \norm{\m{y}_{t+1} -\m{y}^*_t(\m{x}_t)} 
\\
& +2\alpha L_f \big(f_t(\m{x}_t,\m{y}^*_t(\m{x}_t)) - f_t(\m{x}^*_t,\m{y}^*_t(\m{x}^*_t))\big)\Big) 
+ \frac{D^2}{2\alpha}+\frac{D}{2\alpha}P_{1,T} .
\end{aligned}
\end{equation}
Note that our choice of $K_t$ in the theorem statement as 
\begin{eqnarray}\label{rr}
  K_t>\left\lceil\frac{(\kappa_g+1)\log \rho^{-2}_t}{4}\right\rceil, \quad \textnormal{with}\quad \rho_t=\frac{1}{2t^2}
\end{eqnarray}
satisfies the condition of Lemma \ref{eq:lower:dynamic}--\ref{item:leml1}. Moreover, using Lemma \ref{lem:rimann}--\ref{itm:lem:c:rimann}, we have
\begin{eqnarray*}
\sum_{t=1}^T\rho_t =\frac{\pi^2}{12}~~\textnormal{and}~~\sum_{t=1}^T\rho_t^2 =\frac{\pi^4}{360}.
\end{eqnarray*}

This, together with Lemma~\ref{eq:lower:dynamic}--\ref{item:leml1} and Assumption \ref{assu:dom}, gives
\begin{subequations}\label{eqn:regbc:2}
\begin{equation}
\begin{aligned}\label{nop}
\sum_{t=1}^T &\norm{\m{y}_{t+1} -\m{y}^*_t(\m{x}_t)}^2  \leq   \frac{\rho^2_1}{1-2\rho^2_1}  \|\m{y}_{1} - \m{y}_{1}^*(\m{x}_1)\|^2\\
& +  \frac{6}{1-2\rho^2_1}  \left( 2L_{\m{y}}^2\sum_{t=1}^{T}{\rho}_{t}^2\|\m{x}_{t}-\m{x}^*_{t}\|^2 +  \sum_{t=2}^T{\rho}_t^2 \|\m{y}^*_{t-1}(\m{x}^*_{t-1})-\m{y}^*_{t}(\m{x}^*_{t})\|^2 \right)\\
& \leq  \frac{1}{2} {D'}^2 + \frac{\pi^4}{30}L_{\m{y}}^2  D^2+  \frac{\pi^4}{60}L_{\m{y}}^2 Y_{2,T}.
\end{aligned}
\end{equation}
Similarly, we obtain
\begin{equation}\label{eqn:regbc:3}
\begin{aligned}
\sum_{t=1}^T& \norm{\m{y}_{t+1} -\m{y}^*_t(\m{x}_t)} \leq   \frac{\rho_1}{1-\rho_1} \|\m{y}_{1} - \m{y}_{1}^*(\m{x}_1)\|\\
& +  \frac{1}{1-\rho_1}  \left(2 L_{\m{y}}\sum_{t=1}^{T}\rho_{t}\|\m{x}_{t}-\m{x}^*_{t}\| +  \sum_{t=2}^T\rho_t \|\m{y}^*_{t-1}(\m{x}^*_{t-1})-\m{y}^*_{t}(\m{x}^*_{t})\| \right) \\
&\leq D'+  \pi^2L_{\m{y}} D+  \frac{\pi^2}{2}L_{\m{y}} Y_{1,T}.
\end{aligned}
\end{equation}
\end{subequations}
Now, let
\begin{align*}
\dot{E}_1 (\alpha) &:= \alpha M_f^2\left( \frac{1}{2}{D'}^2+ \frac{\pi^4}{30} L_{\m{y}}^2  D^2\right)  + D M_f \left( D'+  \pi^2L_{\m{y}} D\right).
\end{align*}
Substituting \eqref{nop} and \eqref{eqn:regbc:3} into \eqref{eqn:regbc:1}, we have
\begin{align}\label{final:cons:dconv}
\nonumber
& (1-2\alpha L_f)  \sum_{t=1}^T \left(f_t(\m{x}_t,\m{y}^*_t(\m{x}_t)) - f_t(\m{x}^*_t,\m{y}^*_t(\m{x}^*_t)) \right)\\
& \leq \dot{E}_1 (\alpha) +\frac{D^2}{2\alpha} + \frac{D}{2\alpha}P_{1,T} +\alpha M_f^2  \frac{\pi^4}{60}L_{\m{y}}^2 Y_{2,T} + D M_f \frac{\pi^2}{2}L_{\m{y}} Y_{1,T}.
\end{align}
Let $\alpha \leq 1/(4L_f)$ and
\begin{equation*}
\begin{aligned}
\dot{c}_1(\alpha) &:= \frac{D}{2\alpha(1-2\alpha L_f)},\\
\dot{c}_2(\alpha) &:=\frac{1}{1-2\alpha L_f}  D M_f \frac{\pi^2}{2}L_{\m{y}},\\
\dot{c}_{3}(\alpha)&:=\frac{1}{1-2\alpha L_f} \alpha M_f^2  \frac{\pi^4}{60}L_{\m{y}}^2 ,\\
\dot{c}_4(\alpha) &:=\frac{\dot{E}_1 (\alpha)}{1-2\alpha L_f} + \frac{D^2}{2\alpha(1-2\alpha L_f)}.
\end{aligned}
\end{equation*}
The above definitions together with \eqref{final:cons:dconv} implies
\begin{align*}
\textnormal{BD-Reg}_T \leq  \dot{c}_1(\alpha) P_{1,T} + \dot{c}_2(\alpha) Y_{1,T} +\dot{c}_3(\alpha) Y_{2,T} + \dot{c}_4(\alpha).
\end{align*}
This gives the desired result in \eqref{eqn:convex:dbound}.
\end{proof}
\subsubsection{Proof of Theorem~
\ref{itm:thm:s:convex}}
\begin{proof}
The proof is similar to Theorem \ref{thm:dynamic:convex:B-OGD}.  From the update rule of Algorithm \ref{alg:obgd} (with $w=1$), we have $\m{x}_{t+1} = \Pi_{\mc{X}}\left[\m{x}_t - \alpha_t\tilde{\nabla} f_t(\m{x}_t,\m{y}_{t+1})\right]$. Now, applying the Pythagorean theorem and employing an argument identical to \eqref{eqn:conve:pyht} and \eqref{eqn:convex:reg1}, we obtain 
\begin{subequations}\label{eqn:static:convex:reg1}
\begin{align}
\sum_{t=1}^{T}\left\langle \nabla f_t(\m{x}_t,\m{y}^*_t(\m{x}_t)),\m{x}_t - \m{x}^*\right\rangle&\leq \sum_{t=1}^{T}\left(\frac{1}{2\alpha_t}\norm{\m{x}_t - \m{x}^*}^2 - \frac{1}{2\alpha_t}\norm{\m{x}_{t+1} - \m{x}^*}^2 \right)\label{eqn:static:convex:reg11}\\
&+ \sum_{t=1}^{T}   \alpha_t\norm{\nabla f_t(\m{x}_t,\m{y}^*_t(\m{x}_t))-\tilde{\nabla} f_t(\m{x}_t,  \m{y}_{t+1})}^2\label{eqn:static:convex:reg12}\\
&+ \sum_{t=1}^{T} \left\langle  \nabla f_t(\m{x}_t,\m{y}^*_t(\m{x}_t))- \tilde{\nabla} f_t(\m{x}_t,  \m{y}_{t+1}), \m{x}_t-\m{x}^* \right\rangle
\label{eqn:static:convex:reg13}\\
&+ \sum_{t=1}^{T} \alpha_t\norm{ \nabla f_t(\m{x}_t,\m{y}^*_t(\m{x}_t)) }^2.
\label{eqn:static:convex:reg14s}
\end{align}
\end{subequations}
Next, we upper bound each term of  \eqref{eqn:static:convex:reg1}. 

From Assumption \ref{assu:dom}, we have
\begin{subequations}
\begin{align}\label{eqn:static:convex:reg2}
\nonumber  \eqref{eqn:static:convex:reg11}
 & =\sum_{t=1}^{T} \left(\frac{1}{2\alpha_t}\|\m{x}_t - \m{x}^*\|^2 - \frac{1}{2\alpha_{t+1}}\|\m{x}_{t+1} -\m{x}^*\|^2\right)\\\nonumber
& +\sum_{t=1}^{T} \left( \frac{1}{2\alpha_{t+1}}\|\m{x}_{t+1} -\m{x}^*\|^2 - \frac{1}{2\alpha_t}\|\m{x}_{t+1} -\m{x}^*\|^2\right)
\\\nonumber &\leq \frac{D^2 }{2\alpha_1}-\frac{\|\m{x}_{T+1} - \m{x}^*\|^2 }{2\alpha_{T+1}}+D^2\sum_{t=1}^{T}\left(\frac{1}{2\alpha_{t+1}}-\frac{1}{2\alpha_{t}} \right)
\\&\leq \frac{D^2}{2\alpha_{T+1}}.
\end{align}

Using Lemma~\ref{lem:lips} and Assumption~\ref{assu:dom}, and following similar steps as in the derivation of \eqref{eqn:convex:reg3}, we obtain
\begin{align}
\eqref{eqn:static:convex:reg12} + \eqref{eqn:static:convex:reg13} \nonumber &\leq    D M_f \sum_{t=1}^{T}\norm{\m{y}_{t+1} -\m{y}^*_t(\m{x}_t)} + M_f^2 \sum_{t=1}^{T}\alpha_t\norm{\m{y}_{t+1} -\m{y}^*_t(\m{x}_t)}^2
\\&\leq    D M_f \sum_{t=1}^{T}\norm{\m{y}_{t+1} -\m{y}^*_t(\m{x}_t)} + M_f^2 \alpha_1\sum_{t=1}^{T}\norm{\m{y}_{t+1} -\m{y}^*_t(\m{x}_t)}^2.
\end{align}
Further, it follows from Assumption \ref{assu:f:a1} that
\begin{align}\label{sfs}
\eqref{eqn:static:convex:reg14s} =
\sum_{t=1}^{T} \alpha_{t}  \|\nabla f_{t}(\m{x}_{t},\m{y}^*_{t}(\m{x}_{t}))\|^2 \leq
 \ell_{f,0}^2\sum_{t=1}^{T} \alpha_{t}.
\end{align}
\end{subequations}
Substituting \eqref{eqn:static:convex:reg2}--\eqref{sfs} into \eqref{eqn:static:convex:reg1} gives
\begin{equation}
\begin{aligned}\label{eqn:static:convex:reg3}
& ~~ \sum_{t=1}^T\left( f_t(\m{x}_t,\m{y}^*_t(\m{x}_t)) - f_t(\m{x}^*,\m{y}^*_t(\m{x}^*)) \right)\\
&\leq    \frac{ D^2}{2\alpha_{T+1}}+  \sum_{t=1}^T \left(  M_f^2\alpha_1  \norm{\m{y}_{t+1} -\m{y}^*_t(\m{x}_t)}^2 + D M_f  \norm{\m{y}_{t+1} -\m{y}^*_t(\m{x}_t)}+\ell_{f,0}^2 \alpha_t \right),
\end{aligned}
\end{equation}
By Lemma~\ref{eq:lower:dynamic}--\ref{item:leml1},  \eqref{rr} and Assumption \ref{assu:dom}, we have
\begin{subequations}
\begin{equation}
\begin{aligned}\label{nop1}
\sum_{t=1}^T &\norm{\m{y}_{t+1} -\m{y}^*_t(\m{x}_t)}^2  \leq   \frac{\rho^2_1}{1-2\rho^2_1}  \|\m{y}_{1} - \m{y}_{1}^*(\m{x}_1)\|^2\\
& +  \frac{6}{1-2\rho^2_1}  \left( 2L_{\m{y}}^2\sum_{t=1}^{T}{\rho}_{t}^2\|\m{x}_{t}-\m{x}^*\|^2 +  \sum_{t=2}^T{\rho}_t^2 \|\m{y}^*_{t-1}(\m{x}^*)-\m{y}^*_{t}(\m{x}^*)\|^2 \right)\\
& \leq  \frac{1}{2} {D'}^2 + \frac{\pi^4}{30}L_{\m{y}}^2  D^2+  \frac{\pi^4}{60}L_{\m{y}}^2 \bar{Y}_{2,T},
\end{aligned}
\end{equation}
where $\bar{Y}_{2,T}=  \sum_{t=2}^{T}\norm{\m{y}_{t-1}^*(\m{x}^*) - \m{y}_{t}^*(\m{x}^*)}^2$.

Similarly, we obtain
\begin{equation}\label{eqn:regbc:33}
\begin{aligned}
\sum_{t=1}^T& \norm{\m{y}_{t+1} -\m{y}^*_t(\m{x}_t)} \leq   \frac{\rho_1}{1-\rho_1} \|\m{y}_{1} - \m{y}_{1}^*(\m{x}_1)\|\\
& +  \frac{1}{1-\rho_1}  \left(2 L_{\m{y}}\sum_{t=1}^{T}\rho_{t}\|\m{x}_{t}-\m{x}^*\| +  \sum_{t=2}^T\rho_t \|\m{y}^*_{t-1}(\m{x}^*)-\m{y}^*_{t}(\m{x}^*)\| \right) \\
&\leq D'+  \pi^2L_{\m{y}} D+  \frac{\pi^2}{2}L_{\m{y}} \bar{Y}_{1,T},
\end{aligned}
\end{equation}
\end{subequations}
where $\bar{Y}_{1,T}=  \sum_{t=2}^{T}\norm{\m{y}_{t-1}^*(\m{x}^*) - \m{y}_{t}^*(\m{x}^*)}$.

Let
$$\dot{E}_1 (\alpha_1):= \alpha_1 M_f^2\left( \frac{1}{2}{D'}^2+ \frac{\pi^4}{30} L_{\m{y}}^2  D^2\right)  + D M_f \left( D'+  \pi^2L_{\m{y}} D\right).$$
By substituting  \eqref{nop1} and \eqref{eqn:regbc:33} into \eqref{eqn:static:convex:reg3} and using our choice of the stepsize $\alpha_t = D/(\ell_{f,0} \sqrt{t})$, we obtain
\begin{equation}\label{eqn:static:convex:finreg}
\begin{aligned}
& ~~ \sum_{t=1}^T\left( f_t(\m{x}_t,\m{y}^*_t(\m{x}_t)) - f_t(\m{x}^*,\m{y}^*_t(\m{x}^*)) \right)\\
& \leq  \frac{1}{2} D \ell_{f,0} \sqrt{T} +\dot{E}_1 (\alpha_1) +\alpha_1 M_f^2 \frac{\pi^4}{60}L_{\m{y}}^2 \bar{Y}_{2,T} + D M_f \frac{\pi^2}{2}L_{\m{y}} \bar{Y}_{1,T}+D\ell_{f,0}\sum_{t=1}^T\frac{1}{\sqrt{t}},
\end{aligned}
\end{equation}
Let
\begin{equation*}
\begin{aligned}
\dot{e}_1 &:= \frac{3}{2} D \ell_{f,0},~~\dot{e}_2:=  D M_f \frac{\pi^2}{2}L_{\m{y}},~~~\dot{e}_3:= \alpha_1 M_f^2  \frac{\pi^4}{60}L_{\m{y}}^2.
\end{aligned}
\end{equation*}
From \eqref{eqn:static:convex:finreg} and using Lemma~\ref{lem:rimann}--\ref{itm:lem:b:rimann}, we get
\begin{equation*}
\begin{split}
\textnormal{BS-Reg}_T \leq \dot{e}_1\sqrt{T}+\dot{E}_1 (\alpha_1)+  \dot{e}_3 \bar{Y}_{2,T} +  \dot{e}_2 \bar{Y}_{1,T}.
\end{split}
\end{equation*}
This completes the proof of the theorem and gives \eqref{eqn:convex:sbound}.
\end{proof}
\subsubsection{Discussion on the number of inner iterations and the window size}
As mentioned before, by using inner gradient descent multiple times, we are able to get more information
from each inner function and obtain a tight bound for the dynamic regret in terms of $Y_{p,T}$. However, according to our analysis in Theorems~\ref{thm:dynamic:strong:B-OGD} and~\ref{thm:dynamic:convex:B-OGD}, even for sufficiently large $K_t$ and $w>1$, the dynamic regret bound can only be improved by a constant factor. 
A related question is whether we can reduce the value of $K_t$ by using, for example, the smoothness of $\nabla \m{y}_{t}(\m{x})$, similar to offline bilevel optimization~\citep{chen2021closing}, or by adopting more advanced optimization techniques, such as acceleration or momentum-type gradient methods for both inner and outer updates~\citep{nesterov2003introductory}. These are open problems for us and will be investigated as future work.


\subsection{Proof for Non-convex OBO with Partial Information}\label{sec:app:nonconv}
This section gives regret bounds for OBO in the non-convex setting.
\subsubsection{Auxiliary Lemmas}
\begin{lem}\label{lem:non:lip}
Under Assumption~\ref{assu:f}, for all $t \in [T]$ and $\m{x} \in\mathbb{R}^d$, we have
\begin{equation}
\begin{aligned}
 \left\| \tilde{\nabla} F_{t,\m{u}}(\m{x},\m{y}_{t+1}) -\nabla F_{t,\m{u}}(\m{x},\m{y}^*_{t}(\m{x}))\right\|^2
\leq M_f^2  \left\|\m{y}_{t+1}- \m{y}^*_{t}(\m{x})\right\|^2,
\end{aligned}
\end{equation}
where $\tilde{\nabla} F_{t,\m{u}}$ is  defined in \eqref{eqn:tave:grad} and  $M_f$ is given in Lemma \ref{lem:lips}.
\end{lem}
\begin{proof}
From  \eqref{eqn:tave:grad}, we get
\begin{align*}
&\quad \norm{ \tilde{\nabla} F_{t,\m{u}}(\m{x},\m{y}_{t+1}) -\nabla F_{t,\m{u}}(\m{x},\m{y}^*_{t}(\m{x}))}^2\\
&= \norm{ \frac{1}{W} \sum_{i=0}^{w-1} u_i \left (\tilde{\nabla} f_{t-i}(\m{x}, \m{y}_{t+1})- \nabla f_{t-i}(\m{x}, \m{y}^*_{t}(\m{x}))\right) }^2\\
 &\leq \frac{1}{2W^2} \sum_{i=0}^{w-1} \sum_{j=0}^{w-1} u_i u_j
 \norm{\tilde{\nabla} f_{t-i}(\m{x}, \m{y}_{t+1})- \nabla f_{t-i}(\m{x}, \m{y}^*_{t}(\m{x})) }^2
\\
&
+ \frac{1}{2W^2}\sum_{i=0}^{w-1} \sum_{j=0}^{w-1} u_i u_j  \left\|\tilde{\nabla} f_{t-j}(\m{x}, \m{y}_{t+1})- \nabla f_{t-j}(\m{x}, \m{y}^*_{t}(\m{x})) \right\|^2
\\
& =\frac{1}{W^2} \sum_{j=0}^{w-1} u_j \sum_{i=0}^{w-1}u_i
\left\|\tilde{\nabla} f_{t-i}(\m{x}, \m{y}_{t+1})- \nabla f_{t-i}(\m{x}, \m{y}^*_{t}(\m{x})) \right\|^2
\\
& \leq \frac{M_f^2}{W^2} \sum_{j=0}^{w-1} u_j \sum_{i=0}^{w-1}  u_i\left\|\m{y}_{t+1}- \m{y}^*_{t}(\m{x})\right\|^2\\
& = M_f^2\left\|\m{y}_{t+1}- \m{y}^*_{t}(\m{x})\right\|^2.
\end{align*}
Here, the first inequality uses Lemma \ref{lem:trig} with $c=1$;  the second inequality uses  Lemma \ref{lem:lips}; and the last equality follows since $(1/W)\sum_{i=0}^{w-1} u_i=1$.
\end{proof}

Similar to Lemma~\ref{eq:lower:dynamic}, the following lemma characterizes the inner estimation error  $\|\m{y}_{t+1} - \m{y}^*_t(\m{x}_t) \|$, where $ \m{y}_{t+1} $ is the inner variable update via Algorithm~\ref{alg:obgd}. In particular, it shows that by applying inner gradient descent at each round $t$, we are able to obtain an error bound in terms of the local regret $\|\nabla F_{t,\m{u}}(\m{x}_{t},\m{y}^*_{t}(\m{x}_{t}))\|^2$ and the inner solution variation $H_{2,T}=\sum_{t=2}^{T} \sup_{\m{x}\in \mb{R}^{d_1}} \| \m{y}^*_{t-1}(\m{x}) - \m{y}^*_{t}(\m{x})\|^2$.

\begin{lem}\label{lm:nn}
Suppose Assumption~\ref{assu:f} holds. If we choose the stepsizes as
\begin{align*}
\beta_t=\beta =\frac{2}{\ell_{g,1}+\mu_g},~~~\textnormal{and}~~~ \alpha_t=\alpha\leq \frac{1}{2\sqrt{2}L_{\m{y}} M_f(\kappa_g^2-1  )^{1/2}},   
\end{align*}
for all $t \in [T]$,
then the sequence $\{(\m{x}_t,\m{y}_{t})\}_{t=1}^T$ generated by Algorithm~\ref{alg:obgd} satisfy
\begin{equation}
\begin{aligned}\label{eqn:yt1-ystar}
 \sum_{t=1}^T  \|  \m{y}_{t+1} -  \m{y}^*_t(\m{x}_t) \|^2 
 & \leq  \frac{(\kappa_g-1)^2}{2(\kappa_g+1)}\|\m{y}_{1}-\m{y}^*_{1}(\m{x}_{1})\|^2
 \\&  +\frac{1}{2M_f^2}(\frac{\kappa_g-1}{\kappa_g+1})\sum_{t=1}^{T}\left\|\nabla F_{t,\m{u}}(\m{x}_{t},\m{y}^*_{t}(\m{x}_{t}))\right\|^2+2(\kappa_g-1  )^2H_{2,T} 
.  
\end{aligned}
\end{equation}
Here, $H_{2,T}=\sum_{t=2}^{T} \sup_{\m{x}\in \mb{R}^{d_1}} \| \m{y}^*_{t-1}(\m{x}) - \m{y}^*_{t}(\m{x})\|^2$.
\end{lem}
\begin{proof}
Since $\beta =2/(\ell_{g,1}+\mu_g)$, from Lemma \ref{lm:nes},  we have
\begin{align*}
 \|\m{y}_{t+1} - \m{y}_{t}^*(\m{x}_{t})\|^2&
\leq \left( 1 -  \frac{2}{\kappa_g+1}  \right)^2\left\|\m{y}_{t}-\m{y}_{t}^*(\m{x}_{t})\right\|^2,
\end{align*}
which implies that 
  \begin{align}\label{NM}
\nonumber
 \sum_{t=1}^{T}\|\m{y}_{t+1}-\m{y}^*_{t}(\m{x}_{t})\|^2
& \leq  \left( 1 -  \frac{2}{\kappa_g+1}  \right)^2\|\m{y}_{1}-\m{y}^*_{1}(\m{x}_{1})\|^2\\
&+\left( 1 -  \frac{2}{\kappa_g+1}  \right)^2 \sum_{t=2}^{T}\|\m{y}_{t}-\m{y}^*_{t}(\m{x}_{t})\|^2.
  \end{align}
From Lemma \ref{lem:trig}, we have
 \begin{align}\label{MLm}
 \nonumber \sum_{t=2}^{T}\|\m{y}_{t}-\m{y}^*_{t}(\m{x}_{t})\|^2 &\leq 
  \left( 1 +  \frac{1}{\kappa_g+1}  \right)\sum_{t=2}^{T}\|\m{y}_{t}-\m{y}^*_{t-1}(\m{x}_{t-1})\|^2\\
\nonumber &+\left(1+ \kappa_g+1\right)\sum_{t=2}^{T}\|\m{y}^*_{t}(\m{x}_{t})- \m{y}^*_{t-1}(\m{x}_{t-1}) \|^2
\\ \nonumber&\leq 
  \left( 1 +  \frac{1}{\kappa_g+1}  \right)\sum_{t=2}^{T}\|\m{y}_{t}-\m{y}^*_{t-1}(\m{x}_{t-1})\|^2\\
\nonumber &+2(2+ \kappa_g)\sum_{t=2}^{T}\|\m{y}^*_{t}(\m{x}_{t})- \m{y}^*_{t}(\m{x}_{t-1}) \|^2
\\\nonumber &+2(2+ \kappa_g)\sum_{t=2}^{T}\|\m{y}^*_{t}(\m{x}_{t-1})- \m{y}^*_{t-1}(\m{x}_{t-1}) \|^2
\\ &\leq 
  \left( 1 +  \frac{1}{\kappa_g+1}  \right)\sum_{t=1}^{T}\|\m{y}_{t+1}-\m{y}^*_{t}(\m{x}_{t})\|^2 \nonumber\\
&+2(2+ \kappa_g)\sum_{t=1}^{T} \|\m{y}^*_{t+1}(\m{x}_{t+1})- \m{y}^*_{t+1}(\m{x}_{t}) \|^2
+2(2+ \kappa_g)H_{2,T}.
 \end{align}
From Lemma \ref{lem:lips} and the update rule of $\m{x}_t$, we obtain 
\begin{align*}
\nonumber
\|\m{y}^*_{t+1}(\m{x}_{t+1})- \m{y}^*_{t+1}(\m{x}_{t}) \|^2
 &\leq  L_{\m{y}}^2 \|\m{x}_{t} -\m{x}_{t+1}\|^2 \\
 \nonumber
 &=L_{\m{y}}^2  \alpha^2  \left\|
\tilde{\nabla} F_{t,\m{u}}(\m{x}_{t},\m{y}_{t+1})\right\|^2
\\
\nonumber
& \leq  2  L_{\m{y}}^2  \alpha^2   \left\|\nabla F_{t,\m{u}}(\m{x}_{t},\m{y}^*_{t}(\m{x}_{t}))\right\|^2\\\nonumber
& + 2 L_{\m{y}}^2  \alpha^2 \left\| \tilde{\nabla} F_{t,\m{u}}(\m{x}_{t},\m{y}_{t+1}) - \nabla F_{t,\m{u}}(\m{x}_{t},\m{y}^*_{t}(\m{x}_{t}))\right\|^2 \\
& \leq  2  L_{\m{y}}^2   \alpha^2  \left(  \left\|\nabla F_{t,\m{u}}(\m{x}_{t},\m{y}^*_{t}(\m{x}_{t}))\right\|^2 +   M_f^2\left\|\m{y}_{t+1}- \m{y}^*_{t}(\m{x}_{t})\right\|^2 \right),
\end{align*}
where the second inequality holds due to Lemma \ref{lem:trig} and the last  inequality follows from Lemma~\ref{lem:non:lip}.

Now, substituting the above bound into  \eqref{MLm}, we obtain
\begin{align}\label{eew}
 \nonumber   \sum_{t=2}^{T}\|\m{y}_{t}-\m{y}^*_{t}(\m{x}_{t})\|^2 &\leq 
  \left( 1 +  \frac{1}{\kappa_g+1}+4(2+ \kappa_g)  L_{\m{y}}^2   \alpha^2M_f^2  \right)\sum_{t=1}^{T}\|\m{y}_{t+1}-\m{y}^*_{t}(\m{x}_{t})\|^2
\\&+4\left(2+ \kappa_g\right)  L_{\m{y}}^2   \alpha^2\sum_{t=1}^{T}\left\|\nabla F_{t,\m{u}}(\m{x}_{t},\m{y}^*_{t}(\m{x}_{t}))\right\|^2
+2\left(2+ \kappa_g\right)H_{2,T}.
\end{align}
Substituting \eqref{eew} into \eqref{NM}, we get
\begin{align}\label{gfr}
\nonumber \sum_{t=1}^T  \|  \m{y}_{t+1} -  \m{y}^*_t(\m{x}_t) \|^2 
  &\leq  \left( 1 -  \frac{2}{\kappa_g+1}  \right)^2\|\m{y}_{1}-\m{y}^*_{1}(\m{x}_{1})\|^2
  +A (\alpha)\sum_{t=1}^{T}\left\|\m{y}_{t+1}-\m{y}_{t}^*(\m{x}_{t})\right\|^2
\\&+B(\alpha)\sum_{t=1}^{T}\left\|\nabla F_{t,\m{u}}(\m{x}_{t},\m{y}^*_{t}(\m{x}_{t}))\right\|^2  
+2\left( 2 +  \kappa_g  \right)\left( 1 -  \frac{2}{\kappa_g+1}  \right)^2H_{2,T},
  \end{align}
where 
\begin{align*}
 A (\alpha)&:=  \left( 1 +  \frac{1}{\kappa_g+1}  \right)\left( 1 -  \frac{2}{\kappa_g+1}  \right)^2+4\left( 2 +  \kappa_g  \right)\left( 1 -  \frac{2}{\kappa_g+1}  \right)^2 L_{\m{y}}^2  \alpha^2 M_f^2,\quad \textnormal{and}
\\ 
B(\alpha)&:=4\left( 2 +  \kappa_g  \right)\left( 1 -  \frac{2}{\kappa_g+1}  \right)^2  L_{\m{y}}^2   \alpha^2. 
\end{align*}
We now proceed to bound terms $A(\alpha)$ and $B(\alpha)$, respectively. Let’s bound term $A(\alpha)$ first as
\begin{align*}
  A (\alpha)&=  
    \left( 1 -  \frac{2}{\kappa_g+1}  \right)\left( ( 1 +  \frac{1}{\kappa_g+1}  )( 1 -  \frac{2}{\kappa_g+1}  ) +4( 2 +  \kappa_g  )( 1 -  \frac{2}{\kappa_g+1}  ) L_{\m{y}}^2  \alpha^2 M_f^2  \right)
    \\&\leq
    \left( 1 -  \frac{2}{\kappa_g+1}  \right)\left(  1 -  \frac{1}{\kappa_g+1}  +8( 1 +  \kappa_g  )( 1 -  \frac{2}{\kappa_g+1}  ) L_{\m{y}}^2  \alpha^2 M_f^2  \right)
    \\&=
    \left( 1 -  \frac{2}{\kappa_g+1}  \right)\left(  1 -  \frac{1}{\kappa_g+1}  +8(\kappa_g-1) L_{\m{y}}^2  \alpha^2 M_f^2  \right)
    \\&\leq \left( 1 -  \frac{2}{\kappa_g+1}  \right),
\end{align*}
where the first inequality is by the inequality $(1+a/2)(1-a)\leq (1-a/2-a^2/2)\leq 1-a/2$ and the second inequality is due to the assumption that $\alpha^2\leq \frac{1}{8(\kappa_g^2-1  )L_{\m{y}}^2 M_f^2}$. 

Next, we bound $B(\alpha)$ as follows
\begin{align*}
B(\alpha)\leq 8( 1 +  \kappa_g  ) \left( 1 -  \frac{2}{\kappa_g+1}  \right)^2  L_{\m{y}}^2   \alpha^2=\frac{8(\kappa_g-1)^2}{\kappa_g+1} L_{\m{y}}^2   \alpha^2\leq \frac{\kappa_g-1}{(\kappa_g+1)^2M_f^2},
\end{align*}
where the last inequality holds because $\alpha^2\leq \frac{1}{8(\kappa_g^2-1  )L_{\m{y}}^2 M_f^2}$.

Inserting the above two bounds for $A(\alpha)$ and $B(\alpha)$ into  \eqref{gfr} gives
\begin{align*}
 \sum_{t=1}^T  \|  \m{y}_{t+1} -  \m{y}^*_t(\m{x}_t) \|^2 
 & \leq  \left(\frac{\kappa_g-1}{\kappa_g+1}\right)^2\|\m{y}_{1}-\m{y}^*_{1}(\m{x}_{1})\|^2
 +\left(1- \frac{2}{\kappa_g+1}\right)\sum_{t=1}^{T}\norm{\m{y}_{t+1}-\m{y}_{t}^*(\m{x}_{t})}^2
 \\&  +\frac{\kappa_g-1}{(\kappa_g+1)^2M_f^2}\sum_{t=1}^{T}\left\|\nabla F_{t,\m{u}}(\m{x}_{t},\m{y}^*_{t}(\m{x}_{t}))\right\|^2+\frac{4(\kappa_g-1  )^2}{\kappa_g+1  }H_{2,T} 
.  
\end{align*}
Rearranging the above terms gives \eqref{eqn:yt1-ystar} and completes the proof.
\end{proof}

The following lemma shows that the difference between the time-averaged function $F_{t,\m{u}}$ computed at  $(\m{x}_t,\m{y}^*_t(\m{x}_{t}))$ and $ (\m{x}_{t+1},\m{y}^*_{t}(\m{x}_{t+1})) $ is bounded. This extends the single-level setting to the generic weight sequence  $\{u_{i}\}_{i=0}^{w-1}$, and the proof utilizes the ideas from \cite[Lemmas 3.2,  3.3]{aydore2019dynamic} and \cite[
Theorem 3]{hazan2017efficient}.
\begin{lem}\label{lem:capdiff}
Let $\{(f_t,g_t)\}_{t=1}^T$ be the sequence of functions presented to Algorithm~\ref{alg:obgd}, satisfying Assumptions \ref{assu:f:a1} and \ref{assu:f:b}. Then, we have
\begin{equation}\label{eqn:diff:Ftu}
 \sum_{t=1}^{T} \left(F_{t,\m{u}}(\m{x}_t,\m{y}^*_t(\m{x}_{t})) - F_{t,\m{u}}(\m{x}_{t+1},\m{y}^*_{t}(\m{x}_{t+1}))\right)\leq \frac{2T M }{W} +2M
+ \ell_{f,0}  H_{1,T}.
\end{equation}
Here, $ H_{1,T} =  \sum_{t=2}^{T} \sup_{\m{x}\in \mb{R}^{d_1}} \| \m{y}^*_{t}(\m{x}) - \m{y}^*_{t-1}(\m{x})\|$ and $M$ is defined in Assumption~\ref{assu:f:b}.
\end{lem}
\begin{proof}
Observe that
\begin{subequations}\label{eqn:fu:diff}
\begin{align}
\nonumber &\quad \sum_{t=1}^{T} \left( F_{t,\m{u}}(\m{x}_t,\m{y}^*_t(\m{x}_{t})) - F_{t,\m{u}}(\m{x}_{t+1},\m{y}^*_{t}(\m{x}_{t+1}))\right)
\\\label{eqn:fu:diff1} &= \sum_{t=1}^{T} \left( F_{t,\m{u}}(\m{x}_t,\m{y}^*_t(\m{x}_{t})) - F_{t,\m{u}}(\m{x}_{t+1},\m{y}^*_{t+1}(\m{x}_{t+1}))\right)
 \\\label{eqn:fu:diff2}&+\sum_{t=1}^{T} \left( F_{t,\m{u}}(\m{x}_{t+1},\m{y}^*_{t+1}(\m{x}_{t+1}))- F_{t,\m{u}}(\m{x}_{t+1},\m{y}^*_{t}(\m{x}_{t+1}))\right).
\end{align}
\end{subequations}
In the following, we bound the terms \eqref{eqn:fu:diff1} and \eqref{eqn:fu:diff2} separately.

For \eqref{eqn:fu:diff2}, we have
\begin{align}\label{sdf}
\nonumber \eqref{eqn:fu:diff2} &=\sum_{t=1}^{T} \frac{1}{W} \sum_{i=0}^{w-1} u_i\left( f_{t-i}(\m{x}_{t+1},\m{y}^*_{t+1}(\m{x}_{t+1})) - f_{t-i}(\m{x}_{t+1},\m{y}^*_{t}(\m{x}_{t+1}))\right)
\\\nonumber&\leq \frac{\ell_{f,0}}{W}\sum_{i=0}^{w-1} u_i \sum_{t=1}^{T}\|\m{y}^*_{t+1}(\m{x}_{t+1})-\m{y}^*_{t}(\m{x}_{t+1}) \|
\\\nonumber &= \ell_{f,0} \sum_{t=1}^{T}\|\m{y}^*_{t+1}(\m{x}_{t+1})-\m{y}^*_{t}(\m{x}_{t+1}) \|
\\&\leq  \ell_{f,0} H_{1,T},
\end{align}
where the first inequality is due to Assumption \ref{assu:f:a1} and the second equality follows since $(1/W)\sum_{i=0}^{w-1} u_i=1$.

For the term \eqref{eqn:fu:diff1}, we have
\begin{align*}
\nonumber  \eqref{eqn:fu:diff1}&=\sum_{t=2}^{T} \Big( F_{t,\m{u}}(\m{x}_t,\m{y}^*_{t}(\m{x}_{t})) - F_{t-1,\m{u}}(\m{x}_{t},\m{y}^*_{t}(\m{x}_{t}))\Big)
\\\nonumber & +F_{1,\m{u}}(\m{x}_1,\m{y}^*_1(\m{x}_{1}))-F_{T,\m{u}}(\m{x}_{T+1},\m{y}^*_{T+1}(\m{x}_{T+1}))
\\\nonumber &=\sum_{t=2}^{T} \frac{1}{W}\sum_{i=0}^{w-1} u_i\left( f_{t-i}(\m{x}_{t},\m{y}^*_{t}(\m{x}_{t})) - f_{t-1-i}(\m{x}_{t},\m{y}^*_{t}(\m{x}_{t}))\right)
\\ & +f_{1}(\m{x}_1,\m{y}^*_1(\m{x}_{1}))-\frac{1}{W}\sum_{i=0}^{w-1} u_i  f_{T-i}(\m{x}_{T+1},\m{y}^*_{T+1}(\m{x}_{T+1}))
.
\end{align*}
Since $\{u_{i}\}_{i=0}^{w-1}$ is the weight sequence with $ 1=u_{0}\geq u_1 \ldots u_{w-1} >0$, given in Definition~\ref{def:avg:grad}, we have
\begin{align*}
 &\quad\sum_{i=0}^{w-1} u_i\left( f_{t-i}(\m{x}_{t},\m{y}^*_{t}(\m{x}_{t})) - f_{t-1-i}(\m{x}_{t},\m{y}^*_{t}(\m{x}_{t}))\right)
\\ &=u_0f_{t}(\m{x}_{t},\m{y}^*_{t}(\m{x}_{t})) +u_1 f_{t-1}(\m{x}_{t},\m{y}^*_{t}(\m{x}_{t}))+\cdots+u_{w-1} f_{t-w+1}(\m{x}_{t},\m{y}^*_{t}(\m{x}_{t}))
 \\&-u_0f_{t-1}(\m{x}_{t},\m{y}^*_{t}(\m{x}_{t})) -u_1 f_{t-2}(\m{x}_{t},\m{y}^*_{t}(\m{x}_{t}))-\cdots-u_{w-1}f_{t-w}(\m{x}_{t},\m{y}^*_{t}(\m{x}_{t}))
 \\&= u_0 f_{t}(\m{x}_{t},\m{y}^*_{t}(\m{x}_{t})) -u_{w-1} f_{t-w}(\m{x}_{t},\m{y}^*_{t}(\m{x}_{t}))
+\sum_{i=1}^{w-1} \left(u_i-u_{i-1}\right) f_{t-i}(\m{x}_{t},\m{y}^*_{t}(\m{x}_{t}))
\\&\leq u_0 f_{t}(\m{x}_{t},\m{y}^*_{t}(\m{x}_{t})) -u_{w-1} f_{t-w}(\m{x}_{t},\m{y}^*_{t}(\m{x}_{t}))
+\sum_{i=1}^{w-1} \left(u_{i-1}-u_i\right) |f_{t-i}(\m{x}_{t},\m{y}^*_{t}(\m{x}_{t}))|
\\&\leq u_0 f_{t}(\m{x}_{t},\m{y}^*_{t}(\m{x}_{t})) -u_{w-1} f_{t-w}(\m{x}_{t},\m{y}^*_{t}(\m{x}_{t}))
+\max_i |f_{t-i}(\m{x}_{t},\m{y}^*_{t}(\m{x}_{t}))|(u_0 -u_{w-1} )
,
\end{align*}
which implies that
\begin{align}\label{nn}
\nonumber  \eqref{eqn:fu:diff1}&\leq \sum_{t=2}^{T} \frac{1}{W} \left(u_0 f_{t}(\m{x}_{t},\m{y}^*_{t}(\m{x}_{t})) -u_{w-1} f_{t-w}(\m{x}_{t},\m{y}^*_{t}(\m{x}_{t}))
+\max_i |f_{t-i}(\m{x}_{t},\m{y}^*_{t}(\m{x}_{t}))|(u_0 -u_{w-1} )\right)
\\\nonumber & + f_{1}(\m{x}_1,\m{y}^*_1(\m{x}_{1}))-\frac{1}{W}\sum_{i=0}^{w-1} u_i f_{T-i}(\m{x}_{T+1},\m{y}^*_{T+1}(\m{x}_{T+1}))
\\\nonumber&\leq \frac{2TM \left(u_0 - u_{w-1} \right)}{W}
 +M+M
\\&\leq \frac{2TM }{W} +M+M,
\end{align}
where the second inequality is by
Assumption~\ref{assu:f:b}.

Combining \eqref{sdf} with \eqref{nn}, we get \eqref{eqn:diff:Ftu}.
\end{proof}
\subsubsection{Proof of Theorem~\ref{thm:dynamic:nonc}}\label{sec:app:reg:nonconvex}
\begin{proof}
Lemma~\ref{lem:lips} implies that \eqref{eqn:lip:cons3} still holds by replacing $f_t$ with $F_t$. Hence,
\begin{align}\label{eqn:prthm:non1}
\nonumber&\quad F_{t, \m{u}}(\m{x}_{t+1}, \m{y}^*_{t}(\m{x}_{t+1})) - F_{t, \m{u}}(\m{x}_t,\m{y}^*_t(\m{x}_{t}))\\\nonumber
&\leq \left\langle \nabla F_{t, \m{u}}(\m{x}_t, \m{y}^*_t(\m{x}_{t})), \m{x}_{t+1} - \m{x}_t \right\rangle+ \frac{L_f}{2} \| \m{x}_{t+1} - \m{x}_t \|^2   \\
 & \leq -\alpha \left\langle \nabla F_{t, \m{u}}(\m{x}_t, \m{y}^*_t(\m{x}_{t})),   \tilde{\nabla} F_{t, \m{u}}(\m{x}_t,\m{y}_{t+1})  \right\rangle + \frac{L_f \alpha^2 }{2} \|\tilde{\nabla} F_{t, \m{u}}(\m{x}_t,\m{y}_{t+1}) \|^2.
\end{align}
By Lemma \ref{lem:non:lip}, we have
\begin{subequations}
\begin{equation}\label{eqn:prthm:non2}
    \begin{split}
&- \left\langle \nabla F_{t, \m{u}}(\m{x}_t, \m{y}^*_t(\m{x}_{t})),   \tilde{\nabla} F_{t, \m{u}}(\m{x}_t,\m{y}_{t+1})  \right\rangle =  - \left\langle  \nabla F_{t, \m{u}}(\m{x}_t, \m{y}^*_t(\m{x}_{t})),   \nabla F_{t, \m{u}}(\m{x}_t, \m{y}^*_t(\m{x}_{t}))\right\rangle\\
&\qquad \qquad \qquad - \left\langle  \nabla F_{t, \m{u}}(\m{x}_t, \m{y}^*_t(\m{x}_{t})),     \tilde{\nabla} F_{t, \m{u}}(\m{x}_t,\m{y}_{t+1}) -\nabla F_{t, \m{u}}(\m{x}_t, \m{y}^*_t(\m{x}_{t})) \right\rangle \\
&\qquad \qquad \qquad  \leq  -\frac{1}{2} \left\| \nabla F_{t, \m{u}}(\m{x}_t, \m{y}^*_t(\m{x}_{t}))\right\|^2 + \frac{1}{2} \left\|\tilde{\nabla} F_{t, \m{u}}(\m{x}_t, \m{y}_{t+1})- \nabla F_{t, \m{u}}(\m{x}_t, \m{y}^*_t(\m{x}_{t}))\right\|^2 \\
&\qquad \qquad \qquad  \leq  -\frac{1}{2} \left\| \nabla F_{t, \m{u}}(\m{x}_t, \m{y}^*_t(\m{x}_{t}))\right\|^2 + \frac{M_f^2}{2} \left \|\m{y}_{t+1}- \m{y}^*_t(\m{x}_{t})\right\|^2,
    \end{split}
\end{equation}
and
\begin{align}\label{eqn:prthm:non3}   \nonumber\|\tilde{\nabla} F_{t, \m{u}}(\m{x}_t,\m{y}_{t+1}) \|^2 &\leq
   2 \|  \nabla F_{t,\m{u}}(\m{x}_t,\m{y}^*_t(\m{x}_{t}))\|^2+ 2  \left\|\tilde{\nabla} F_{t, \m{u}}(\m{x}_t, \m{y}_{t+1})- \nabla F_{t, \m{u}}(\m{x}_t, \m{y}^*_t(\m{x}_{t}))\right\|^2
   \\&\leq
   2 \|  \nabla F_{t,\m{u}}(\m{x}_t,\m{y}^*_t(\m{x}_{t}))\|^2+ 2 M_f^2 \left\|\m{y}_{t+1}- \m{y}^*_t(\m{x}_{t})\right\|^2.
\end{align}
\end{subequations}
Substituting \eqref{eqn:prthm:non2} and \eqref{eqn:prthm:non3} into \eqref{eqn:prthm:non1}, rearranging terms and summing up from $t = 1$ to $t = T$, we obtain
\begin{align}\label{eqn:grad:I13}
\nonumber &\quad \left(\frac{\alpha}{2}-L_f \alpha^2\right) \sum_{t=1}^{T}\| \nabla F_{t,\m{u}}(\m{x}_t,\m{y}^*_t(\m{x}_{t})) \|^2
\\\nonumber & \leq  \sum_{t=1}^{T} \left( F_{t,\m{u}}(\m{x}_t,\m{y}^*_t(\m{x}_{t}))  - F_{t,\m{u}}(\m{x}_{t+1},\m{y}^*_t(\m{x}_{t+1}))\right)
+M_f^2\left(\frac{\alpha}{2}+L_f \alpha^2\right)\sum_{t=1}^{T} \left\|\m{y}_{t+1}- \m{y}^*_t(\m{x}_{t})\right\|^2
\\&\leq    \frac{2T M }{W} +2M
+ \ell_{f,0} H_{1,T}+M_f^2\left(\frac{\alpha}{2}+L_f \alpha^2\right)\sum_{t=1}^{T} \left\|\m{y}_{t+1}- \m{y}^*_t(\m{x}_{t})\right\|^2,
\end{align}
where the second inequality  follows from Lemma \ref{lem:capdiff}.

Note that our choices of the stepsizes $\alpha_t$ and $\beta_t$ as 
\begin{align*}
\beta_t=\beta =\frac{2}{\ell_{g,1}+\mu_g},~~~\textnormal{and}~~~ \alpha_t=\alpha\leq \min \left\{ \frac{1}{8L_f}, \frac{1}{2\sqrt{2}L_{\m{y}} M_f(\kappa_g^2-1  )^{1/2}}\right\},   
\end{align*}
satisfy the condition of Lemma \ref{lm:nn}. Hence, from Lemma \ref{lm:nn}, we get
\begin{align*}
\nonumber &\quad\left(\frac{\alpha}{2}-L_f \alpha^2\right) \sum_{t=1}^{T}\| \nabla F_{t,\m{u}}(\m{x}_t,\m{y}^*_t(\m{x}_{t})) \|^2
\\&\leq 
\frac{2T M }{W} +2M
+ \ell_{f,0} H_{1,T}
+\frac{1}{2}\left(\frac{\alpha}{2}+L_f \alpha^2\right)\left(\frac{\kappa_g-1}{\kappa_g+1}\right)\sum_{t=1}^{T}\left\|\nabla F_{t,\m{u}}(\m{x}_{t},\m{y}^*_{t}(\m{x}_{t}))\right\|^2  
\\&+M_f^2\left(\frac{\alpha}{2}+L_f \alpha^2\right)\left(\kappa_g-1\right)^2\left( \frac{\|\m{y}_{1}-\m{y}^*_{1}(\m{x}_{1})\|^2}{2(\kappa_g+1)} 
+2H_{2,T}\right)
.
\end{align*}
Rearranging the terms leads to
\begin{align}\label{po}
\nonumber &\quad\frac{1}{2(\kappa_g+1)}\left( \frac{(\kappa_g+3)\alpha}{2}-3(\kappa_g+\frac{1}{3})L_f \alpha^2\right)\sum_{t=1}^T \| \nabla F_{t,\m{u}}(\m{x}_t, \m{y}^*_t(\m{x}_t)) \|^2
\\& \leq 
\frac{2T M }{W} +2M
+ \ell_{f,0} H_{1,T}
+M_f^2\left(\frac{\alpha}{2}+L_f \alpha^2\right)(\kappa_g-1)^2\left( \frac{\|\m{y}_{1}-\m{y}^*_{1}(\m{x}_{1})\|^2}{2(\kappa_g+1)} 
+2H_{2,T}\right)
.
\end{align}
Since $\alpha \leq 1/(8L_f)$, we have 
\begin{align*}
     &\frac{(\kappa_g+3)\alpha}{2}-3(\kappa_g+\frac{1}{3})L_f \alpha^2\geq\frac{(\kappa_g+1)\alpha}{2}-3(\kappa_g+1)L_f \alpha^2\geq \frac{\alpha}{8}(\kappa_g+1),~~~\textnormal{and}
     \\&\frac{\alpha}{2}+L_f \alpha^2\leq \frac{5}{8}\alpha .
\end{align*}
Substituting the above observations into \eqref{po} gives
\begin{align*}
\nonumber \frac{\alpha}{16}\sum_{t=1}^T \| \nabla F_{t,\m{u}}(\m{x}_t, \m{y}^*_t(\m{x}_t)) \|^2
& \leq      \frac{2T M }{W} +2M
+ \ell_{f,0} H_{1,T}
\\&+\frac{5}{8}\alpha  M_f^2(\kappa_g-1)^2\left( \frac{\|\m{y}_{1}-\m{y}^*_{1}(\m{x}_{1})\|^2}{2(\kappa_g+1)} 
+2H_{2,T}\right)
.
\end{align*}
Therefore, we get the following bound
\begin{align*}
\nonumber  \sum_{t=1}^T \| \nabla F_{t,\m{u}}(\m{x}_t, \m{y}^*_t(\m{x}_t)) \|^2
& \leq     \frac{16}{\alpha}\left( \frac{2T M }{W} +2M
+ \ell_{f,0} H_{1,T}\right)
\\& +  10M_f^2    (\kappa_g-1)^2\left( \frac{\|\m{y}_{1}-\m{y}^*_{1}(\m{x}_{1})\|^2}{2(\kappa_g+1)} 
+2H_{2,T}\right)
\\&=\mc{O}\left( \frac{ T}{ W}+  H_{1,T}+   H_{2,T}\right).
\end{align*}
This completes the proof.
\end{proof}

\section{Addendum to Section~\ref{sec:exp}: Implementation Details and Additional Experiments}\label{app:sec:experiments}

\subsection{Details on Online Hyperparameter Learning for Dynamic Regression}\label{app:sec:experiments:dynamic:regression}

The synthetic data are generated as follows: To simulate the distribution changes, we generate the output according to $b_t = \m{a}_t^\top \m{y}_s^*(\m{x}_s^*)+ \epsilon_t$, where $(\m{x}_s^*,\m{y}_s^*(\m{x}_s^*))\in \mb{R}^{d_1} \times  \mb{R}^{d_2}$ is the underlying model for $s$-th stage, and $\epsilon_t \in [0,0.1]$ is the random noise. We consider two setups for the underlying model: (i) there are three changes ($S=3$) in the minimizers $(\m{x}_s^*,\m{y}_s^*(\m{x}_s^*))$, and (ii) the underlying model is fixed ($S=1$), i.e., $(\m{x}^*,\m{y}^*(\m{x}^*))=(\m{x}_s^*,\m{y}_s^*(\m{x}_s^*))$ for all $t \in [T]$. The time horizon, the outer and inner dimensions are set to $T=5000$, $d_1=1$ and $d_2=5$, respectively.

We used a grid-search of parameters in our experiments in Subsection~\ref{sec:exp:dynamic:regression}. For the grid-search setting, we select the best performing parameters \( K_t, \alpha,\beta \) from a grid \( \{5, 10\} \times \{0.001, 0.01, 0.1, 0.5\} \times \{0.001, 0.01, 0.1, 0.5\} \).  The smoothing (averaging) parameter \( \delta \) is set to 0.9.   All algorithms have been run on a Mac machine equipped with a 1.8 GHz Intel Core i5 processor and 8 GB RAM.

\subsection{Details and Additional Experiments on Online Parametric Loss Tuning Implementation}\label{app:sec:detials:autobalance}
This subsection provides details on implementing online parametric loss tuning tailored for imbalanced data. Additionally, it includes extra experiments conducted on two other datasets: Tadpole and Adult.

\subsubsection{Dataset Specifications and Model Architectures}
\paragraph{MNIST} 
 The MNIST image dataset~\citep{lecun2010mnist} comprises $10$ classes of human-written numbers ranging from $0$ to $9$. The dataset contains a total of $60,000$ training images and $10,000$ testing images, each sized at 28$\times28$. Consequently, there are approximately $6,000$ training images and $1,000$ testing images for each class. To introduce imbalance into the training and validation data, we randomly selected
$
 5000\times0.6^i,~ i=0,1,\ldots 9   
$
samples from the original training data for each class. These samples were then divided into new training and validation datasets at a 4:1 ratio. We employed a 4-layer convolutional neural network (CNN) for all comparison algorithms. Each convolutional block in the network consists of a $3\times3$ convolution (with padding=1 and stride=1), batch normalization, ReLU activation, and $2\times2$ max pooling. 
The CNN has $64$ filters in every convolutional layer.

\paragraph{Tadpole} In the Appendix, we conduct additional experiments on the Tadpole dataset~\citep{marinescu2019tadpole}. The Tadople dataset is introduced in Grand Challenge~\footnote{\url{https://tadpole.grand-challenge.org/Data/}}, a platform for end-to-end development of machine learning solutions in biomedical imaging. Tadpole is an abbreviation for The Alzheimer's Disease Prediction Of Longitudinal Evolution (TADPOLE), a subset of the Alzheimer’s Disease Neuroimaging Initiative (ADNI)~\footnote{\url{https://adni.loni.usc.edu/}}, which constitutes an extensive data collection for Alzheimer’s disease (AD). Initially, Tadpole contains 12,741 samples and 1,907 features. 
Its classes include individuals classified as cognitively normal (CN), mild cognitive impairment (MCI) or Alzheimer's disease (AD). 
We only select 17 commonly-used features, including `CDRSB', `ADAS11', `MMSE', `RAVLT\_immediate', `Hippocampus', `WholeBrain', `Entorhinal', `MidTemp', `FDG', `AV45', `ABETA\_UPENNBIOMK9\_04\_19\_17', `TAU\_UPENNBIOMK9\_04\_19\_17', `PTAU\_UPENNBIOMK9\_04\_19\_17', `APOE4', `AGE', `ADAS13', `Ventricles'. We exclusively select classes MCI and AD to form the two-class classification task. 
The two classes are already imbalanced, with 2,106 samples in the AD class and 4,044 in the MCI class.
To further imbalance the dataset, we only select half of the samples from AD. We utilize a 2-layer multilayer perceptron (MLP) with ReLU as the activation function and employ Dropout for regularization.

\paragraph{Adult} We also conduct additional experiments on the Adult dataset~\citep{misc_adult_2}, 
 aiming to predict an individual’s annual income 
based on various factors, including the individual’s education level, age, gender, occupation and more.  The dataset originally comprises 48,842 samples and 15 features. After removing samples with missing values and duplicated features following the process introduced in Kaggle~\footnote{\url{https://www.kaggle.com/code/amirhosseinzinati/adult-income-k-nearest-neighbors-knn}}, we have 45,175 samples and 11 features. The two classes are defined as follows: income less than or equal to \$50K (class 0) and income greater than \$50K (class 1). The original distribution is already imbalanced (0 vs 1 is 3:1). Thus, we do not modify it further. We also employ a 2-layer multilayer perceptron (MLP) with ReLU as the activation function and Dropout for regularization.

\subsubsection{Baselines and Setting Details}
In our experiments, we compare our method to two baselines: one being AutoBalance~\citep{li2021autobalance}, and the other being Single-Level OGD~\citep{zinkevich2003online}.
\paragraph{AutoBalance} 
AutoBalance~\citep{li2021autobalance} is an offline bilevel gradient descent framework that updates hyperparameters $\mathbf{x}_t$ and the model $\mathbf{y}_t$ to address imbalance issues. We essentially adopt all the settings from the Autobalance study. Specifically, in all three datasets, the inner-level optimization trains the CNN model using a learning rate of 0.1, momentum of 0.9, and weight decay of $1e-4$. However, to adapt it to the online environment and ensure a fair comparison, AutoBalance will, at each timestep, utilize all the observed data until the current timestep to train the model instead of employing a fixed number of batches, as in the original setting of AutoBalance. This will allow AutoBalance to run quickly at the beginning but progressively slower as time passes. At the outer level, AutoBalance does not initiate training from the beginning. Instead, AutoBalance usually initiates the outer level after the network achieves near-zero loss. For MNIST, AutoBalance starts outer-level training at the 120th timestep, while it starts at the 80th and 40th timesteps for Tadpole and Adult, respectively. The learning rate for the outer level is 0.001 on all the three datasets. Our OAGD follows the same setting of AutoBalance for both the inner- and outer-level training.

\paragraph{Single-Level OGD}
The Single-Level OGD~\citep{zinkevich2003online} updates the model, $\mathbf{y}_t$, with fixed hyperparameters, $\mathbf{x}$, at each timestep  solely based on the newly observed data using gradient descent. Specifically, the hyperparameters include adjustments in multiplicative and additive logits, along with the inverse class weight. For Single-Level OGD, the multiplicative logits adjustment is 1, the additive logits adjustment is 0, and the inverse class weight is 1, resulting in a vanilla cross-entropy loss. The learning rate is 0.1 on all the three datasets.

\subsubsection{Additional Experiments}
\begin{figure*}[t]
\begin{center}
\includegraphics[scale=0.34]{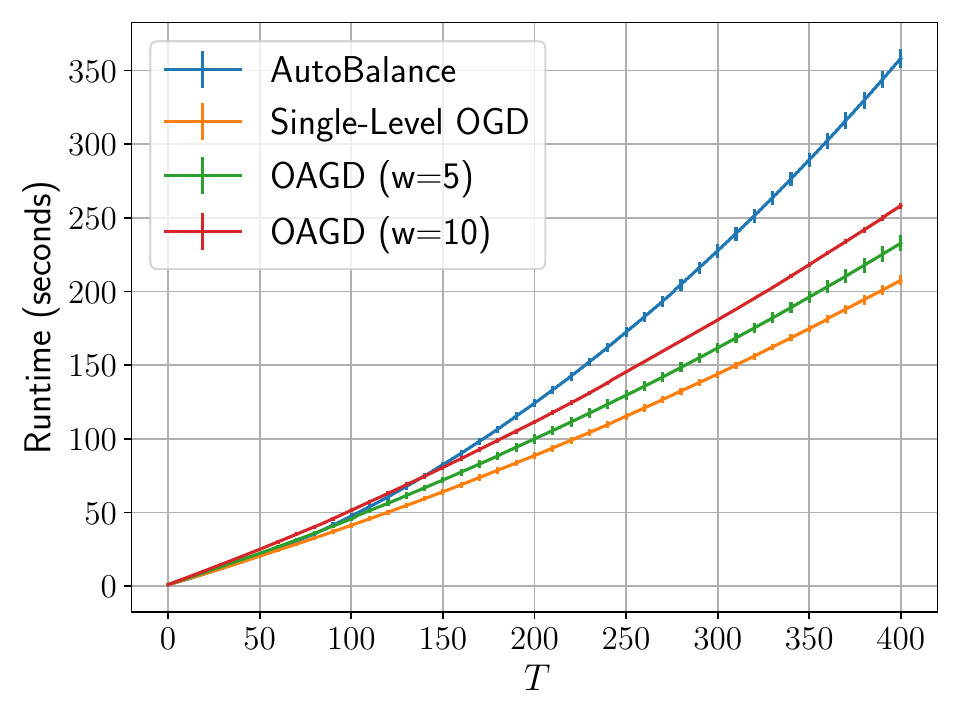}
\includegraphics[scale=0.34]{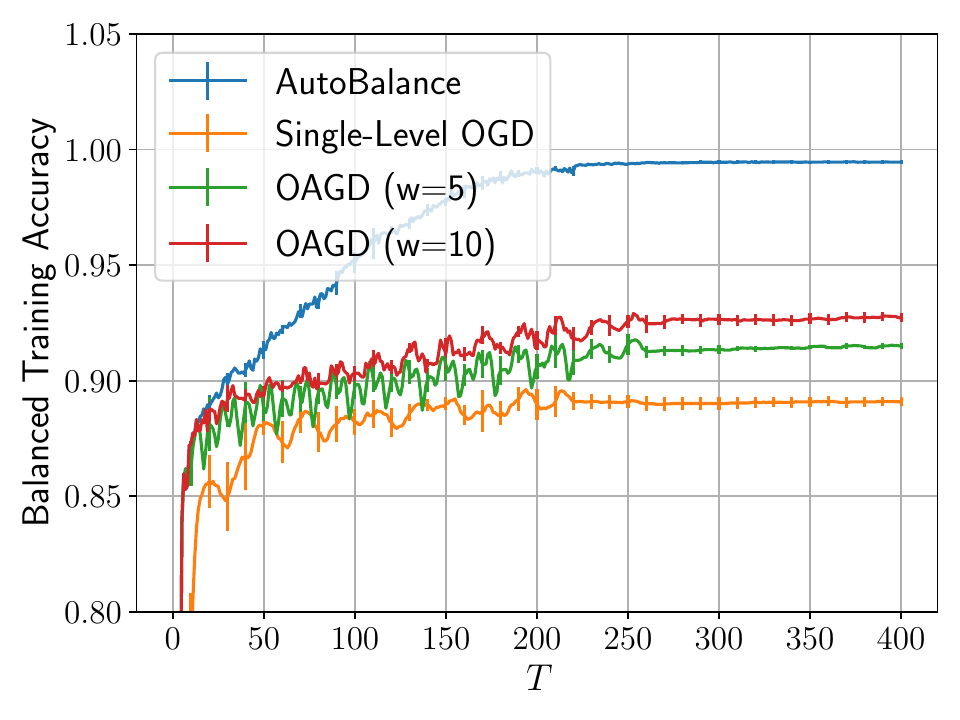} 
\includegraphics[scale=0.34]{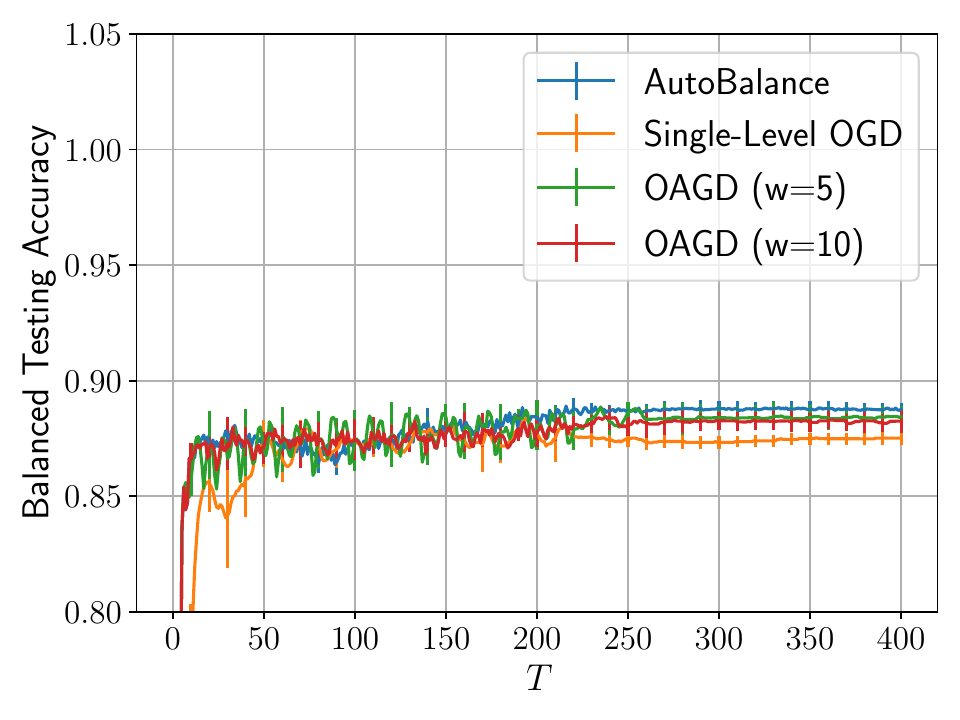}
\vspace{-.25cm}
\caption{\label{fig:tadpole_all}Performance comparison (mean$\pm$std) on parametric loss tuning for imbalanced \textbf{Tadpole} data over five runs. We compare our OAGD ($w=5, 10$)  with AutoBalance and Single-Level OGD. OAGD achieves comparable balanced testing accuracy to AutoBalance but with a reduced runtime.}  
\vspace{-.25cm}
\end{center}
\end{figure*}
\begin{figure*}[htbp]
\begin{center}
\includegraphics[scale=0.34]{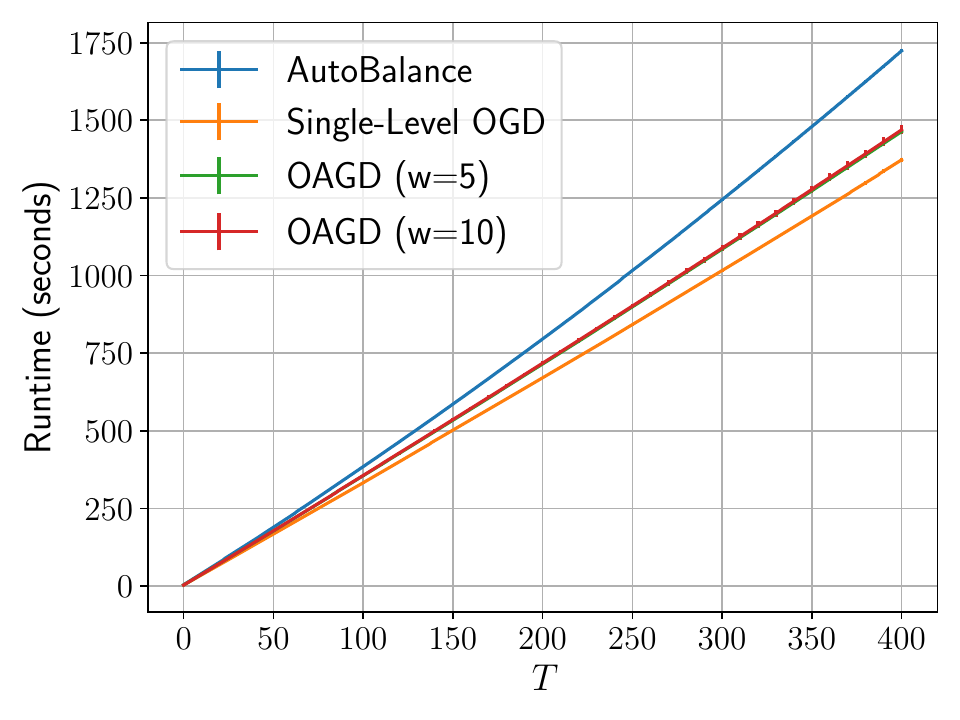}
\includegraphics[scale=0.34]{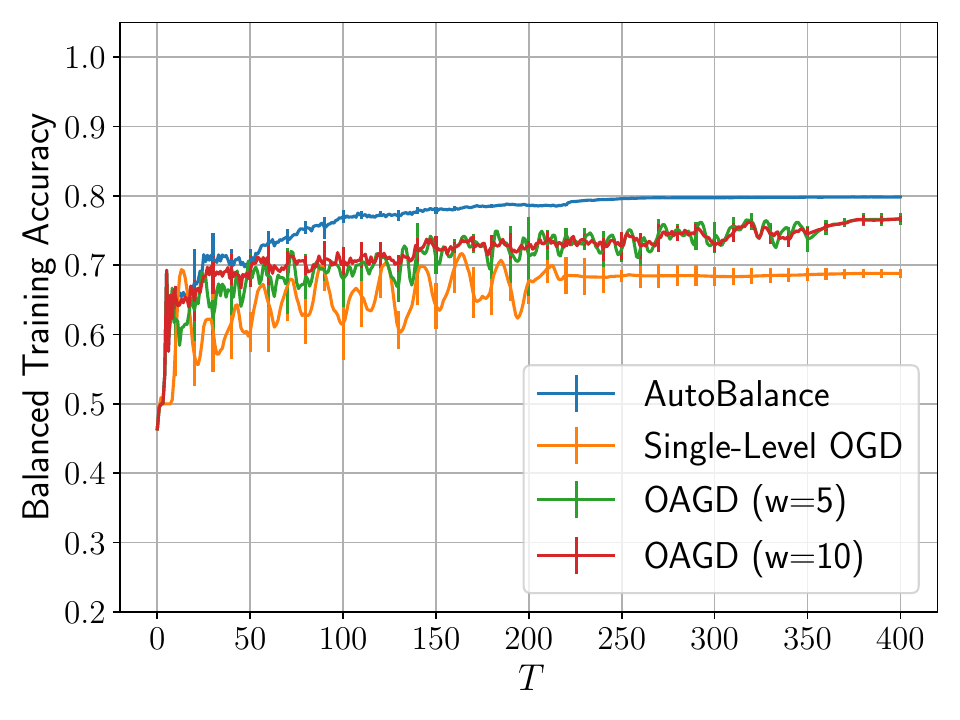} 
\includegraphics[scale=0.34]{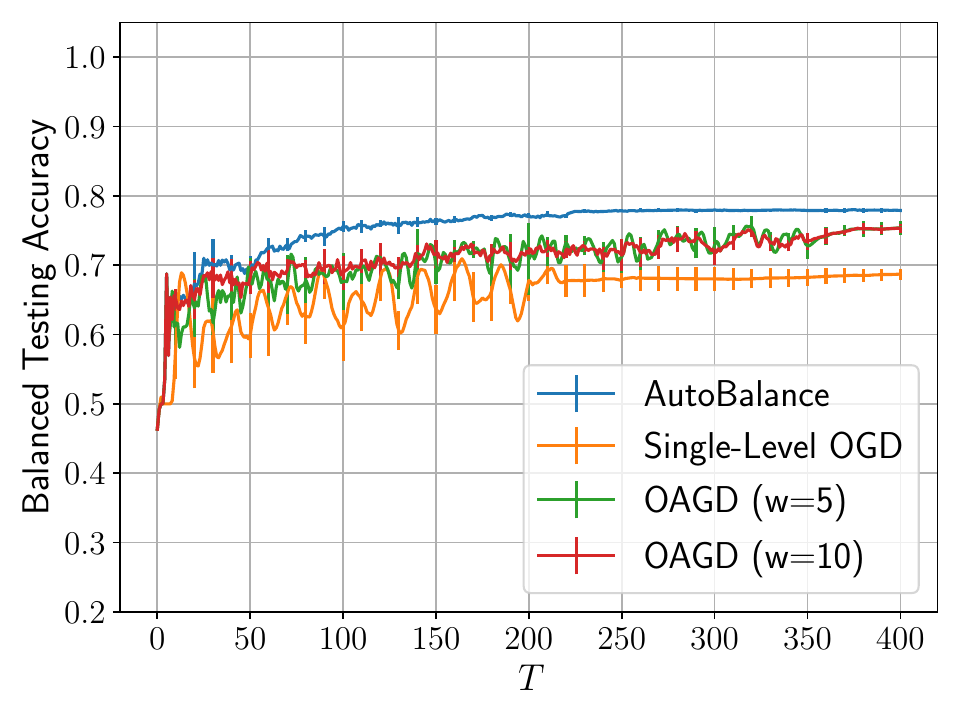}
\vspace{-.25cm}
\caption{\label{fig:adult_all}Performance comparison (mean$\pm$std) on parametric loss tuning for imbalanced \textbf{Adult} data over five runs. We compare our OAGD ($w=5, 10$)  with AutoBalance and Single-Level OGD. OAGD achieves comparable balanced testing accuracy to AutoBalance but with a reduced runtime.}  
\vspace{-.25cm}
\end{center}
\end{figure*}

Figure \ref{fig:tadpole_all} provides a performance comparison (mean$\pm$std) for parametric loss tuning on imbalanced \textbf{Tadpole} data across five runs. We compare our OAGD ($w=5, 10$)  with AutoBalance and Single-Level OGD. OAGD achieves a comparable balanced testing accuracy to AutoBalance but with a reduced runtime. AutoBalance outperforms OAGD and Single-Level OGD in terms of balanced training accuracy. This is because AutoBalance utilizes many more samples than OAGD and OGD. This allows AutoBalance to have a higher chance of overfitting the training data, resulting in high training accuracy, while still obtaining similar balanced testing accuracy compared to other methods.

Figure \ref{fig:adult_all} provides a performance comparison (mean$\pm$std) on parametric loss tuning for imbalanced \textbf{Adult} data over five runs. We compare our OAGD ($w=5, 10$)  with AutoBalance and Single-Level OGD. OAGD achieves comparable balanced accuracy to AutoBalance but with a reduced runtime.

\subsection{Details and Additional Experiments on Online Meta Learning}\label{app:sec:detials:oml}
This subsection provides the implementation details and additional experiments on miniImageNet of online meta learning.

\subsubsection{Datasets and Model Architectures}

\paragraph{FC100} FC100~\citep{oreshkin2018tadam} is a dataset derived from CIFAR100, containing 100 classes, with each class comprising 600 images of size 32. Following~\citep{oreshkin2018tadam}, the 100 classes are divided as follows: 60 classes for meta-training, 20 classes for meta-validation, and 20 classes for meta-testing. There are 36,000, 12,000, 12,000 samples in the original training, validation and testing datasets, respectively. We transform them into 20,000, 600 and 600 training, validation and testing tasks using the {\it TaskDataset} tool from learn2learn~\citep{arnold2020learn2learn}. For all comparison algorithms, we employ a 4-layer convolutional neural networks(CNN) comprising four convolutional blocks. Each block consists of a $3\times3$ convolution (padding=1, stride=2), batch normalization, ReLU activation, and $2\times2$ max pooling. Additionally, each convolutional layer contains 64 filters.

\paragraph{MiniImageNet}
The miniImageNet dataset~\citep{vinyals2016matching} is derived from ImageNet and comprises 100 classes, each containing 600 images sized at $84\times84$. Following the repository, we partition these classes into 64 classes for meta-training, 16 classes for meta-validation, and 20 classes for meta-testing. Following the repository, we use a four-layer CNN with four convolutional blocks, where each block sequentially consists of a $3\times3$ convolution, batch normalization, ReLU activation, and $2\times2$ max pooling. Each convolutional layer has 64 filters.

\subsubsection{Baselines and Setting Details}
In our experiments, we compare our method to three baselines: MAML~\citep{finn2017model}, ANIL~\citep{raghu2019rapid}, and ITD-BiO~\citep{ji2021bilevel}. Notably, these three methods are all implemented using iterative differentiation within the PyTorch framework. They leverage two modules, the {\it features} and the {\it head}, albeit in distinct ways. The {\it features} are used to process the raw input, such as the CNN which processes the image in our case, while the {\it head} is responsible for the final classification. To ensure a fair comparison, we set the inner learning rate $\beta=0.1$, the outer learning rate $\alpha$=0.001, and inner step $K=20$ for all the methods.

\paragraph{MAML}
MAML stands as the foundational work of meta-learning, in which meta-parameters are learned in the outer loop, while task-specific models are learned in the inner loop using only a small amount of data from the current task. In the implementation, MAML combines {\it features} and the {\it head} into the meta model, which is then cloned by the local model to perform local adaptation. We update all parameters of the meta model, including those of the {\it features} and {\it head}.

\paragraph{ANIL}
ANIL stands as a widely used meta-learning algorithm that simplifies MAML by eliminating the inner loop for all parts of the MAML-trained network except the task-specific head. In its implementation, only the {\it head} of the meta model, cloned by the local model, is employed for subsequent local adaptation. Nonetheless, the {\it features} remain utilized for data processing. Ultimately, both the parameters of the {\it features} and the {\it head} undergo updates.

\paragraph{ITD-BiO}
ITD-BiO is a gradient-based stochastic bilevel optimization framework relying on iterative differentiation (ITD). In its implementation, the meta model cloned by the local model comprises solely the {\it head} for subsequent local adaptation. Nevertheless, we continue to utilize the {\it features} for data processing. Ultimately, only the parameters of the {\it features} undergo updates.

It is worth noting that we do not adapt all the baselines to the online environment; instead, we directly utilize the implementation from \citep{ji2021bilevel}. In the online environment, data is observed in batches at a time, and the offline method updates the model by incorporating all the data observed up until the current timestep. This will initially speed up the training, but it will progressively slow down over time. Eventually, the entire process will become very time-consuming. Instead, the implementation uses only a fixed number of batches (tasks) randomly sampled from all available tasks at each timestep, which, in our case, is 32. We only use a specific window-size number of batches (tasks) and load them sequentially.

\begin{figure*}[t]
\begin{center}

\includegraphics[scale=0.34]{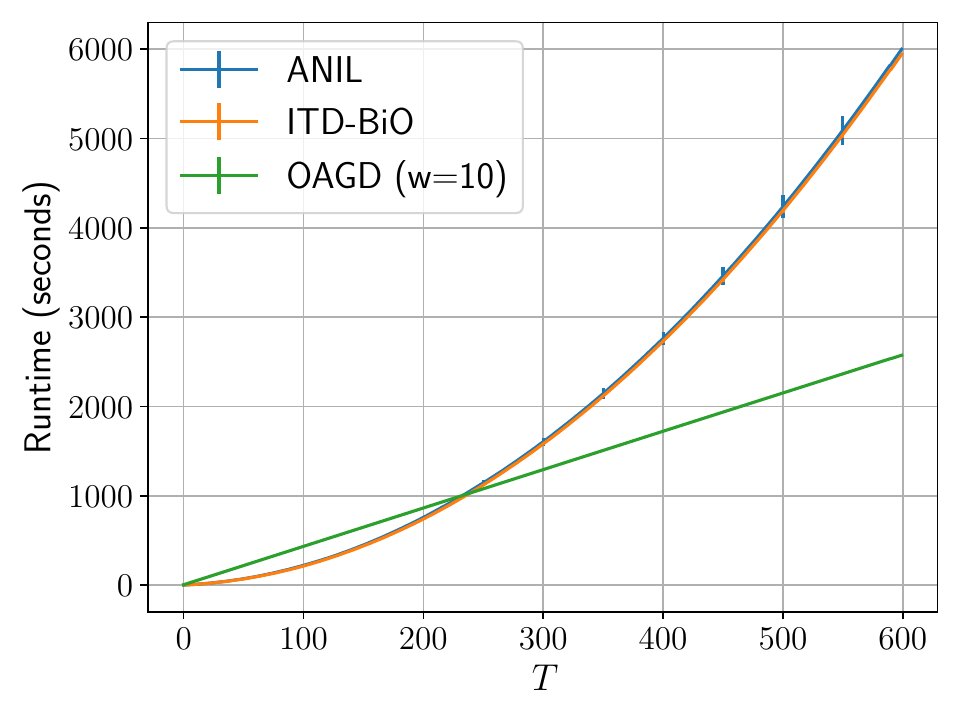}
\includegraphics[scale=0.34]{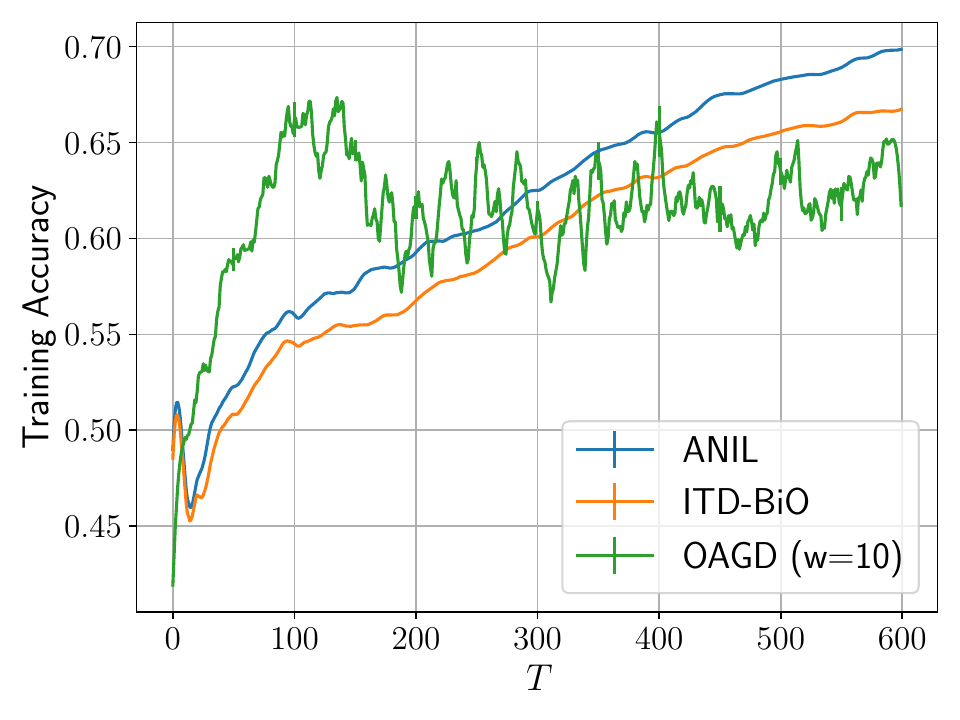} 
\includegraphics[scale=0.34]{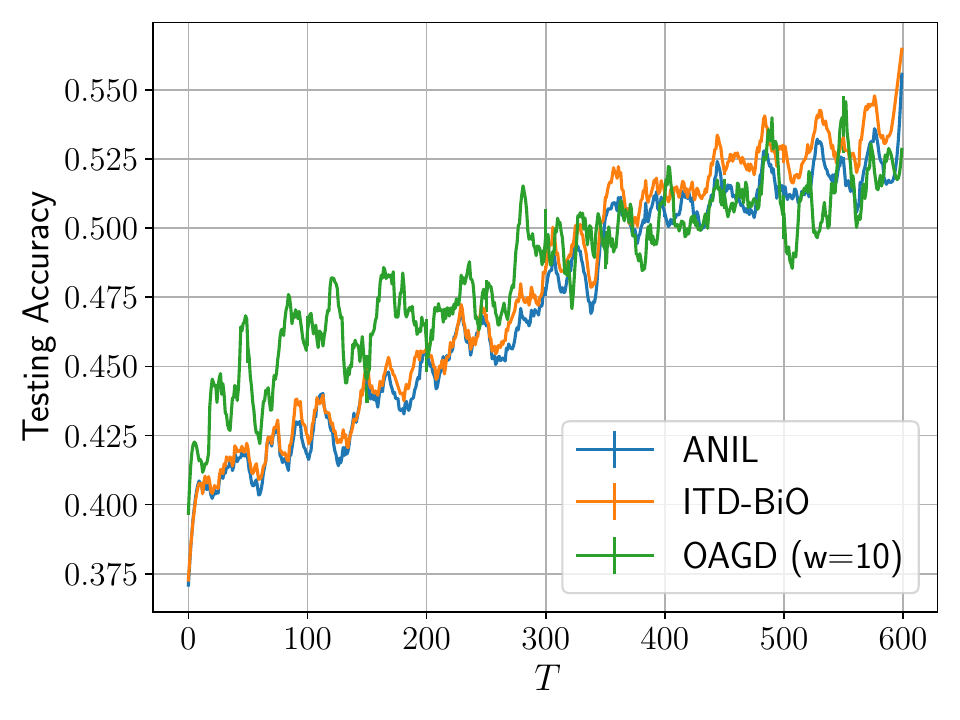}
\vspace{-3mm}
\caption{\label{fig:miniimage_all}Performance comparison (mean$\pm$std) for online meta-learning on the \textbf{miniImageNet} dataset across five runs. We compare our OAGD ($w=10$)  with ANIL and ITD-BiO. OAGD achieves comparable accuracy to the baselines while significantly reducing runtime.}
\end{center}
\vspace{-12pt}
\end{figure*}

\subsubsection{Results on Additional Dataset}\label{app:sec:add:datasets}

We conducted our experiments on an additional dataset, {\bf miniImageNet}. Figure \ref{fig:miniimage_all} provides a performance comparison (mean$\pm$std) for meta-learning on the miniImageNet dataset across five runs. We compare our OAGD ($w=10$) with ANIL and ITD-BiO. OAGD achieves comparable accuracy but with a shorter runtime.

\subsection{Sensitivity Analysis}\label{app:sec:sensetivity}
\begin{figure*}[htbp]
\begin{center}
\includegraphics[scale=0.34]{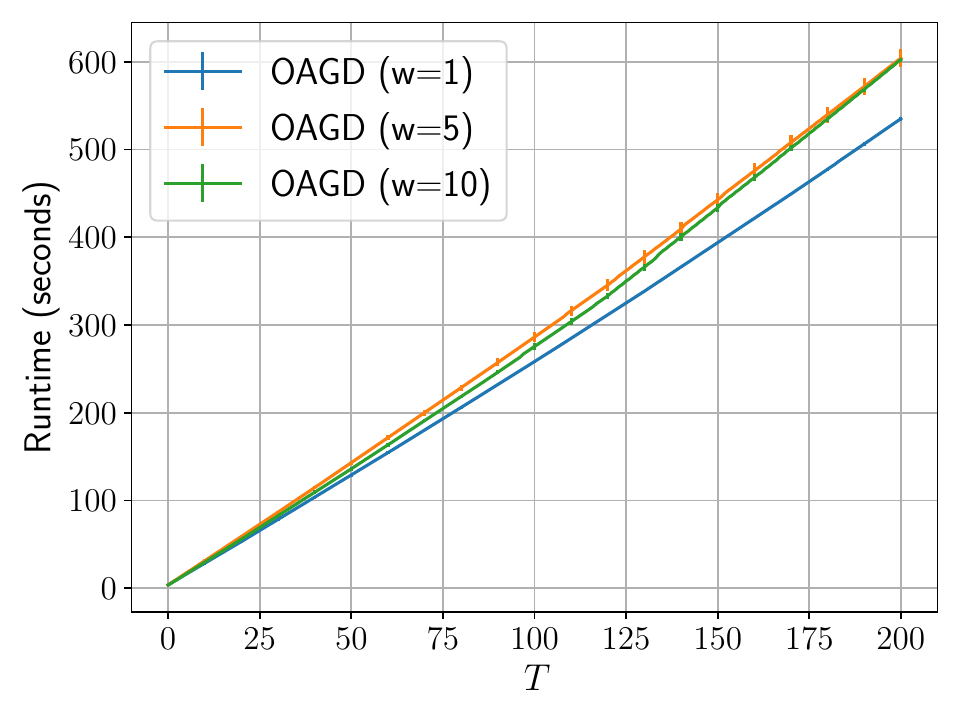}
\includegraphics[scale=0.34]{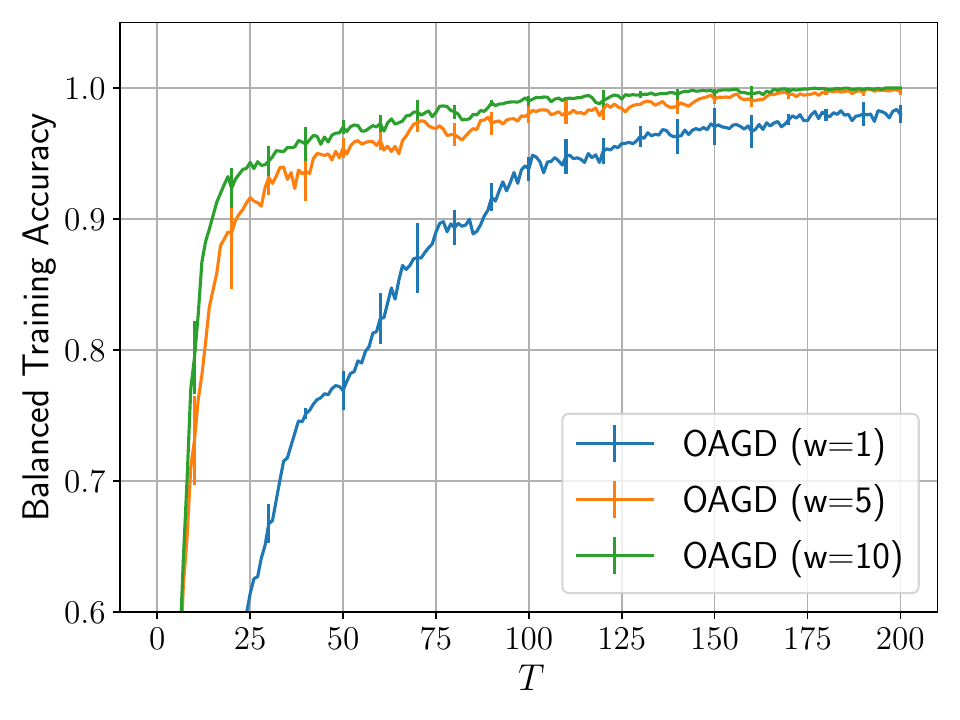} 
\includegraphics[scale=0.34]{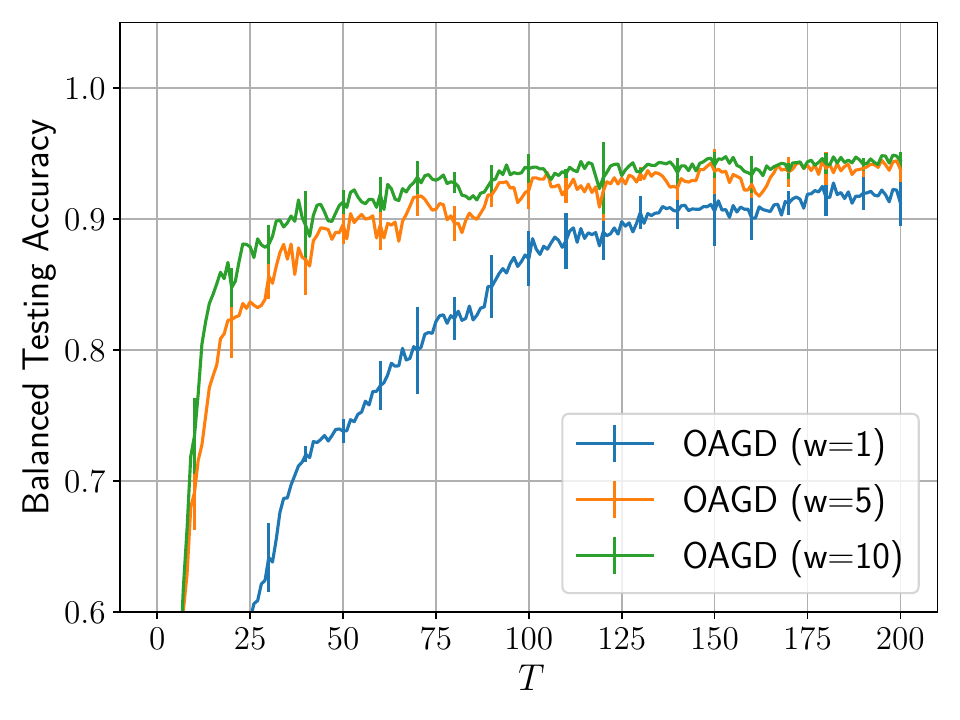}
\vspace{-.25cm}
\caption{\label{fig:mnist_win_sens}Performance comparison (mean$\pm$std) for parametric loss tuning on imbalanced \textbf{MNIST} data across five runs. We compare our OAGD across different window sizes ($w=1, 5, 10$). The larger the window size, the better the accuracy and the longer the runtime.}  
\vspace{-.25cm}
\end{center}
\end{figure*}
\begin{figure*}[htbp]
\begin{center}
\includegraphics[scale=0.34]{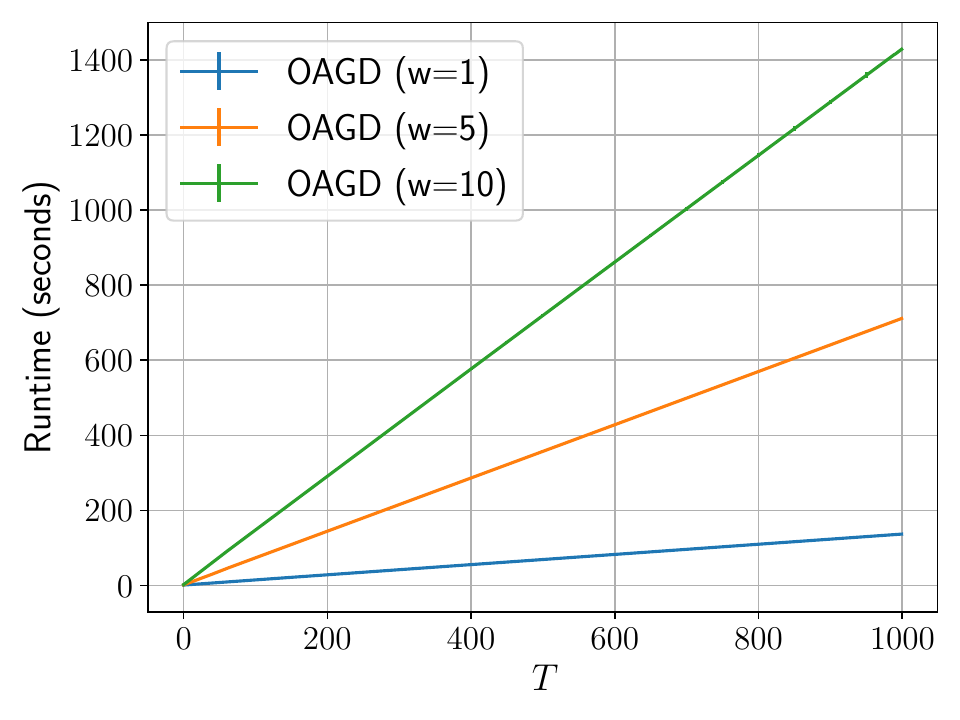}
\includegraphics[scale=0.34]{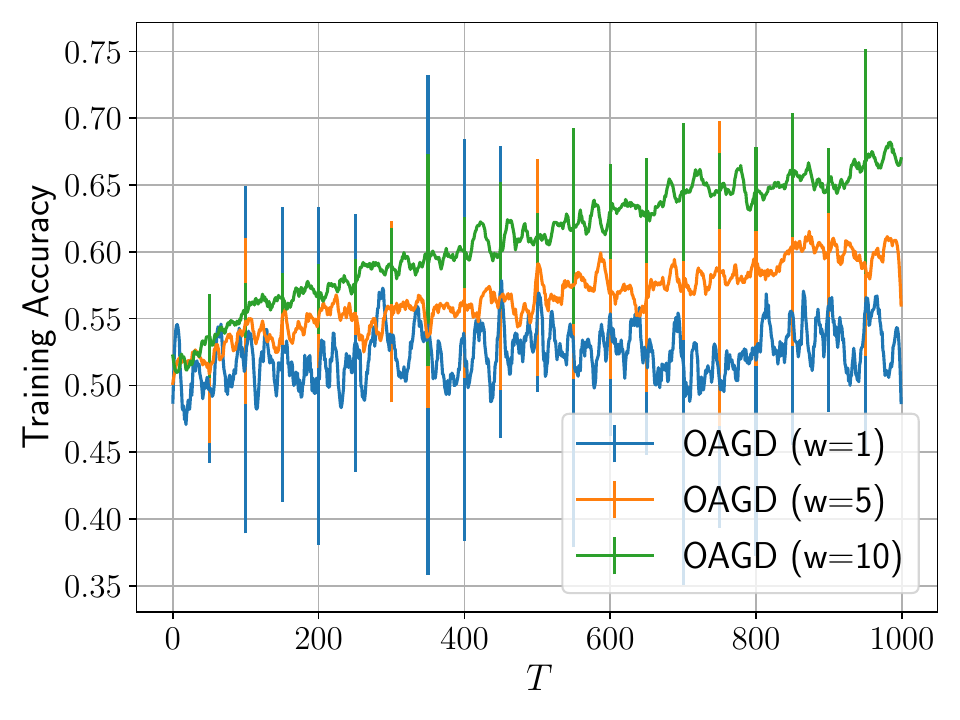} 
\includegraphics[scale=0.34]{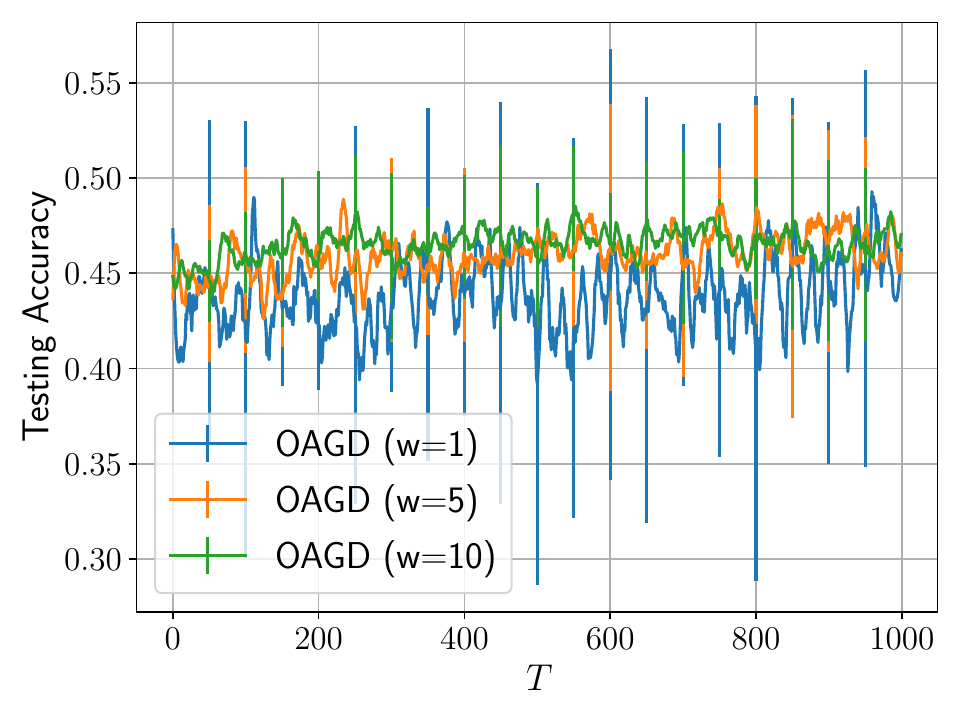}
\vspace{-.25cm}
\caption{\label{fig:fc100_win_sens}Performance comparison (mean$\pm$std) for meta-learning on \textbf{FC100} data across five runs. We compare our OAGD with different window sizes ($w=1, 5, 10$). The larger the window size, the better the accuracy and the longer the runtime.}  
\vspace{-.25cm}
\end{center}
\end{figure*}

In this section, we perform sensitivity analysis on the window size $w$, learning rate (both inner $\beta$ and outer $\alpha$), and inner optimization step $K$ to acquire a comprehensive understanding of our method.
\subsubsection{Sensitivity Analysis of Window Size $w$}
Figures~\ref{fig:mnist_win_sens} and \ref{fig:fc100_win_sens}  illustrate the results of our OAGD on MNIST and FC100 using different window sizes ($w=1,5,10$). As observed, the larger the window size, the higher the accuracy and the longer the runtime. This is expected, as a larger window size allows for the use of more information from previous timesteps. Consequently, the gradient can be approximated with greater accuracy, leading to improved results. However, this also requires more computations for gradient calculation. It is important to note that we do not merely record previous gradients. Instead, we leverage previous data and the current model to compute the gradients, an approach that has been proven to be more effective.


\subsubsection{Sensitivity Analysis of Learning Rates $\alpha$ and $\beta$}
Figures~\ref{fig:tadpole_compare_inner_lr} and \ref{fig:tadpole_compare_outer_lr} display the sensitivity results for the inner and outer learning rates, $\beta$ and $\alpha$, on the Tadpole dataset. Likewise, Figures~\ref{fig:adult_compare_inner_lr} and \ref{fig:adult_compare_outer_lr} exhibit the sensitivity results for the inner and outer learning rates on the Adult dataset. Specifically, when analyzing the sensitivity to the inner learning rate, we fix the outer learning rate at 0.001 and experiment with different inner learning rates ($\beta=0.9, 0.5, 0.1, 0.01$). Conversely, when examining sensitivity to the outer learning rate, we set the inner learning rate to 0.1 and test different outer learning rates ($\alpha=0.0001, 0.001, 0.01, 0.1$). Our observations indicate that our OAGD is not particularly sensitive to changes in either the inner or outer learning rates, as all cases demonstrate consistently high accuracy.
\begin{figure*}[htbp]
\begin{center}
\includegraphics[scale=0.34]{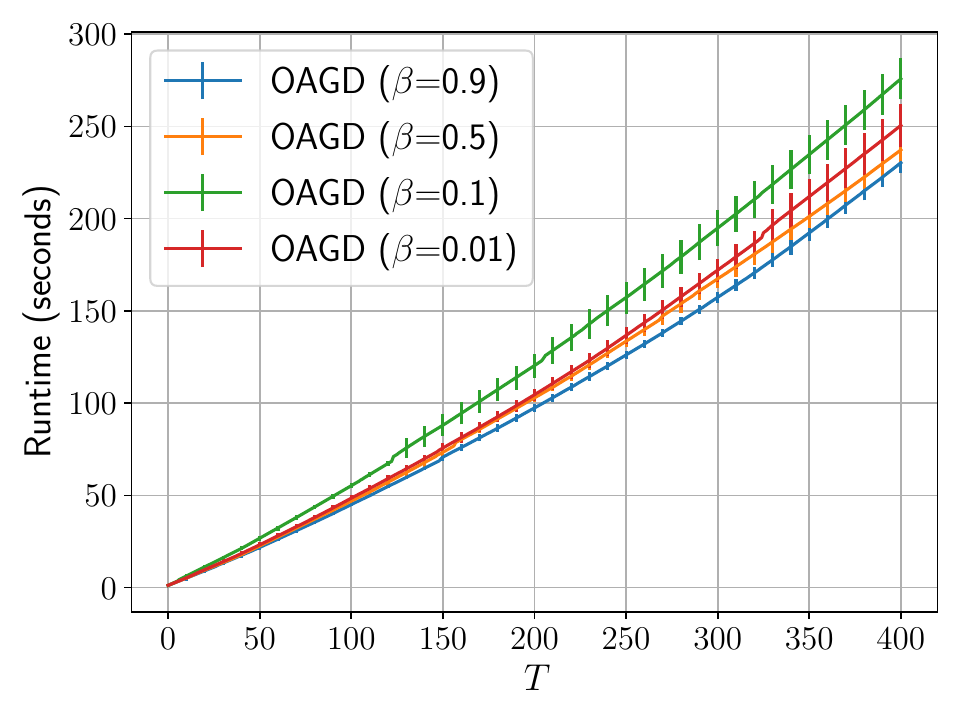}
\includegraphics[scale=0.34]{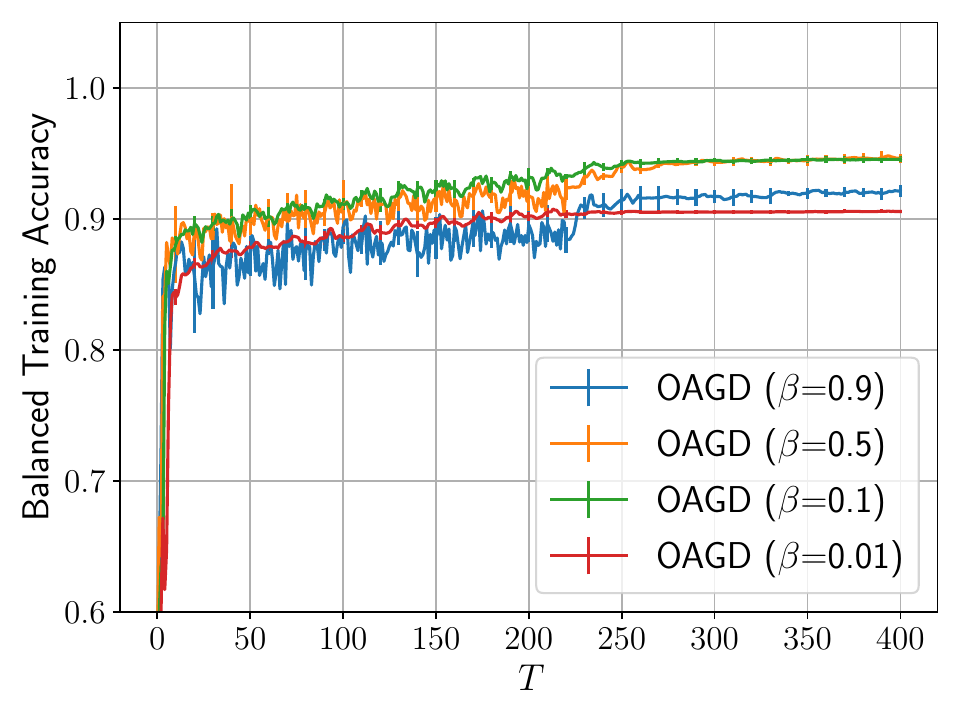} 
\includegraphics[scale=0.34]{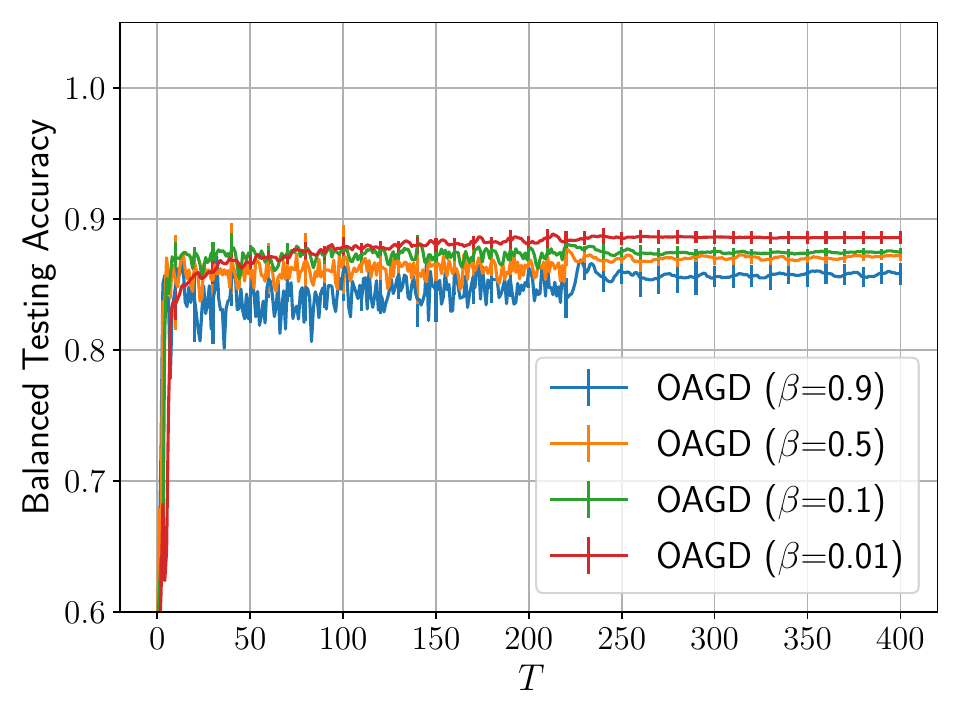}
\vspace{-.25cm}
\caption{\label{fig:tadpole_compare_inner_lr}Performance comparison (mean$\pm$std) on parametric loss tuning for imbalanced \textbf{Tadpole} data over five runs. We compare our OAGD ($w=10$) using various inner learning rates ($\beta=0.01, 0.1, 0.5, 0.9$) while keeping the outer learning rate fixed at $0.001$. Our OAGD is not significantly affected by changes in the inner learning rate.}  
\vspace{-.25cm}
\end{center}
\end{figure*}

\begin{figure*}[htbp]
\begin{center}
\includegraphics[scale=0.34]{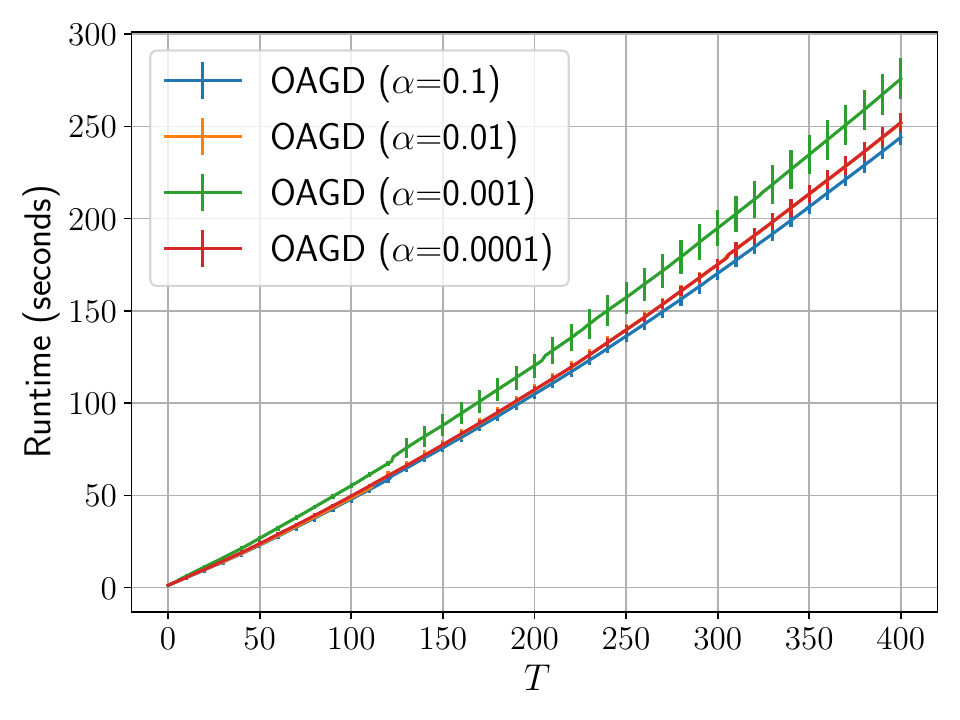}
\includegraphics[scale=0.34]{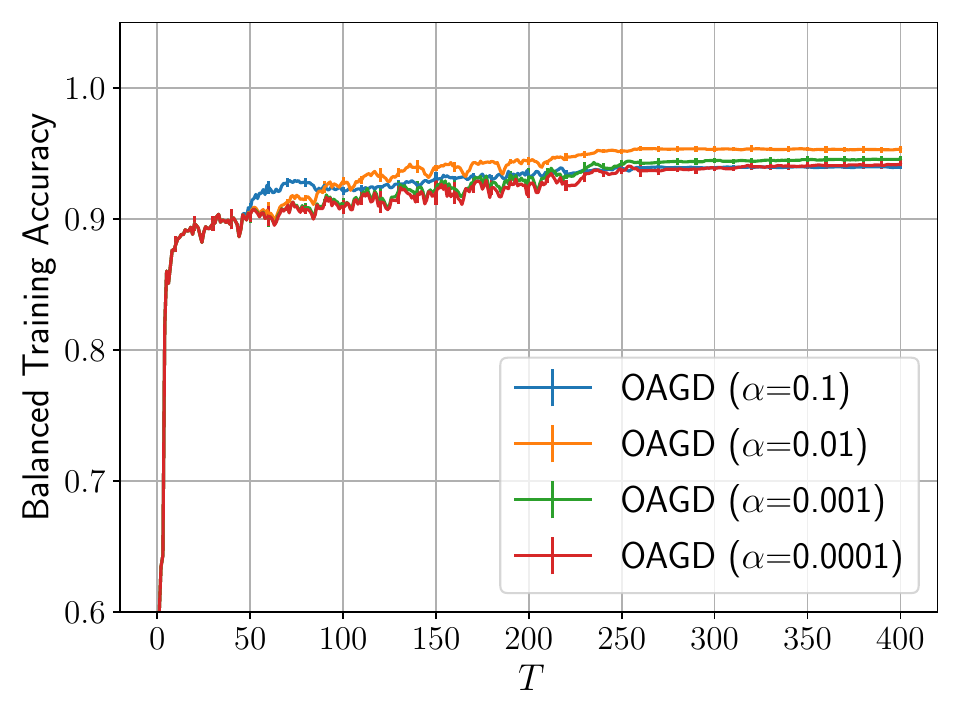} 
\includegraphics[scale=0.34]{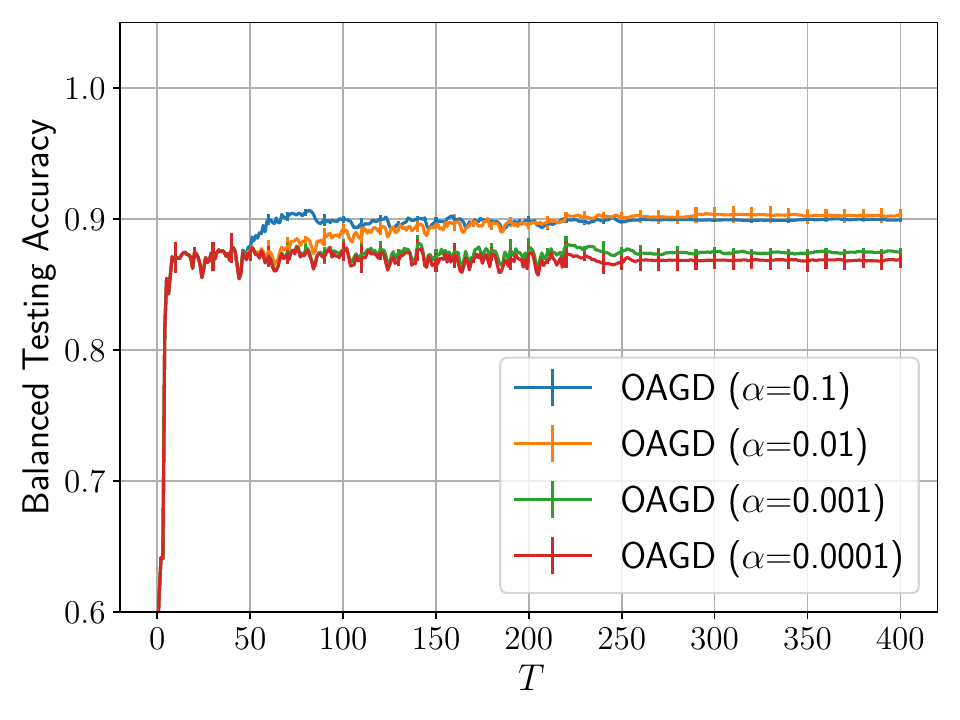}
\vspace{-.25cm}
\caption{\label{fig:tadpole_compare_outer_lr}Performance comparison (mean$\pm$std) on parametric loss tuning for imbalanced \textbf{Tadpole} data over five runs. We compare our OAGD ($w=10$) using various outer learning rates ($\alpha=0.0001, 0.001, 0.01, 0.1$) while maintaining a fixed inner learning rate of $0.1$. Our OAGD is not sensitive to the outer learning rate.}  
\vspace{-.25cm}
\end{center}
\end{figure*}

\begin{figure*}[htbp]
\begin{center}
\includegraphics[scale=0.34]{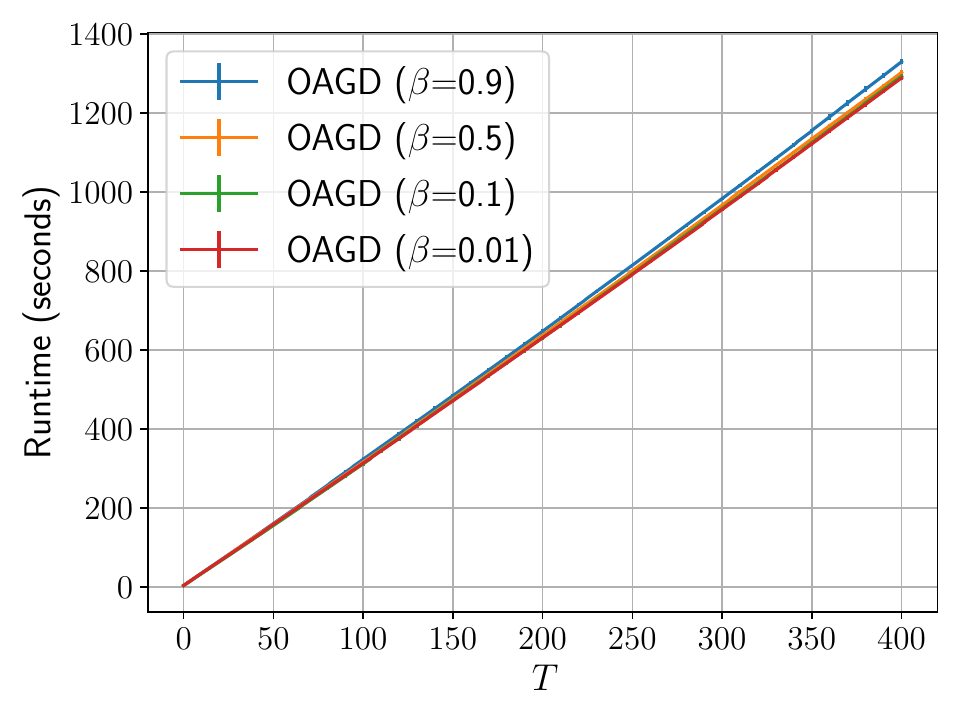}
\includegraphics[scale=0.34]{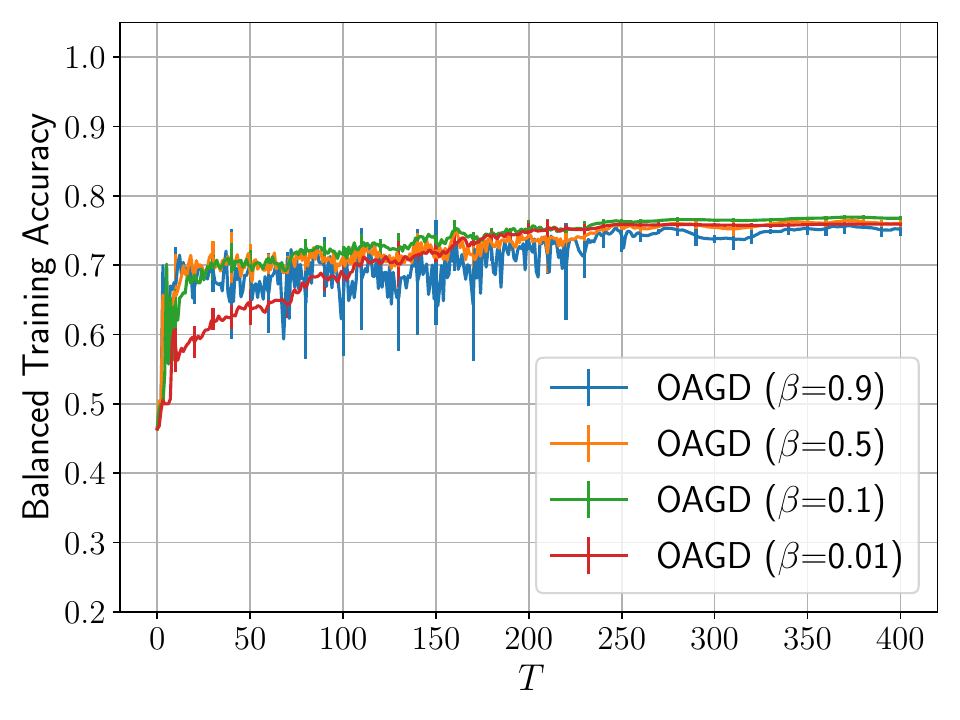} 
\includegraphics[scale=0.34]{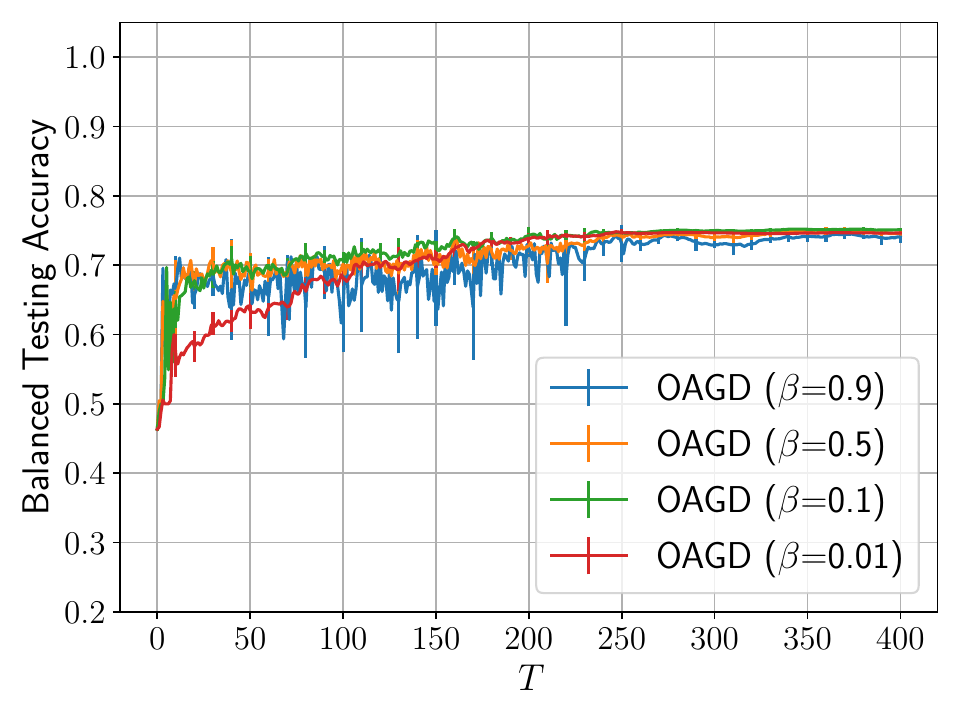}
\vspace{-.25cm}
\caption{\label{fig:adult_compare_inner_lr}Performance comparison (mean$\pm$std) on parametric loss tuning for imbalanced \textbf{Adult} data over five runs. We compare our OAGD ($w=10$) using various inner learning rates ($\beta=0.01, 0.1, 0.5, 0.9$) while keeping the outer learning rate fixed at $0.001$. Our OAGD is not sensitive to the inner learning rate.}  
\vspace{-.25cm}
\end{center}
\end{figure*}
\vspace{2cm}

\begin{figure*}[htbp]
\begin{center}
\includegraphics[scale=0.34]{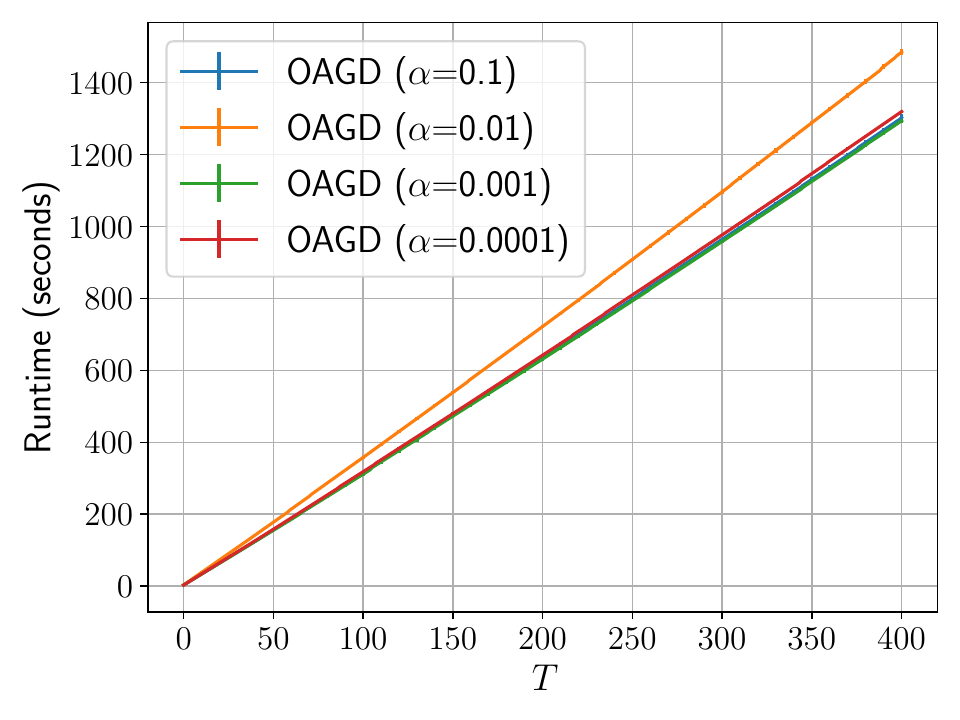}
\includegraphics[scale=0.34]{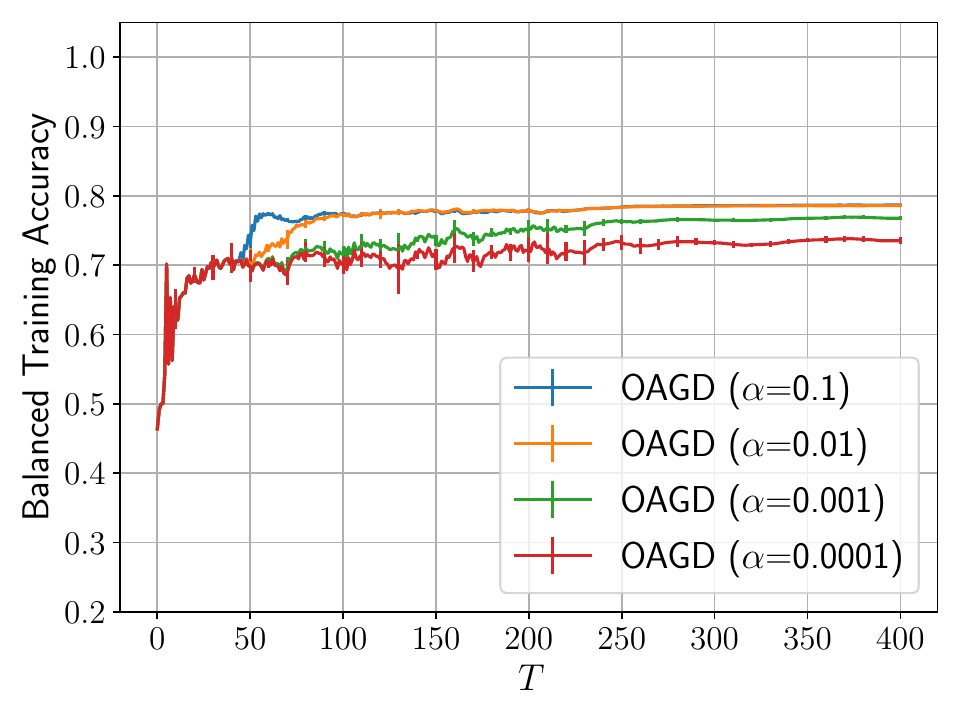} 
\includegraphics[scale=0.34]{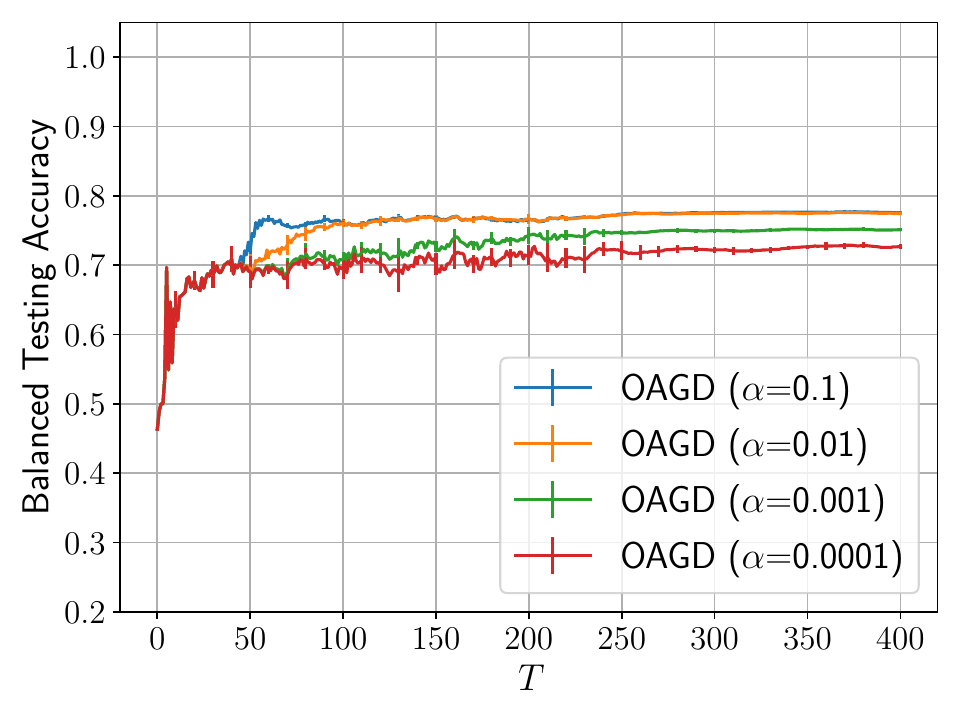}
\vspace{-.25cm}
\caption{\label{fig:adult_compare_outer_lr}Performance comparison (mean$\pm$std) on parametric loss tuning for imbalanced \textbf{Adult} data across five runs. We compare our OAGD ($w=10$) using different outer learning rates ($\alpha=0.0001, 0.001, 0.01, 0.1$) while keeping the inner learning rate fixed at $0.1$. Our OAGD is not sensitive to the outer learning rate.}  
\vspace{-.25cm}
\end{center}
\end{figure*}
\begin{figure*}[htbp]
\begin{center}
\includegraphics[scale=0.34]{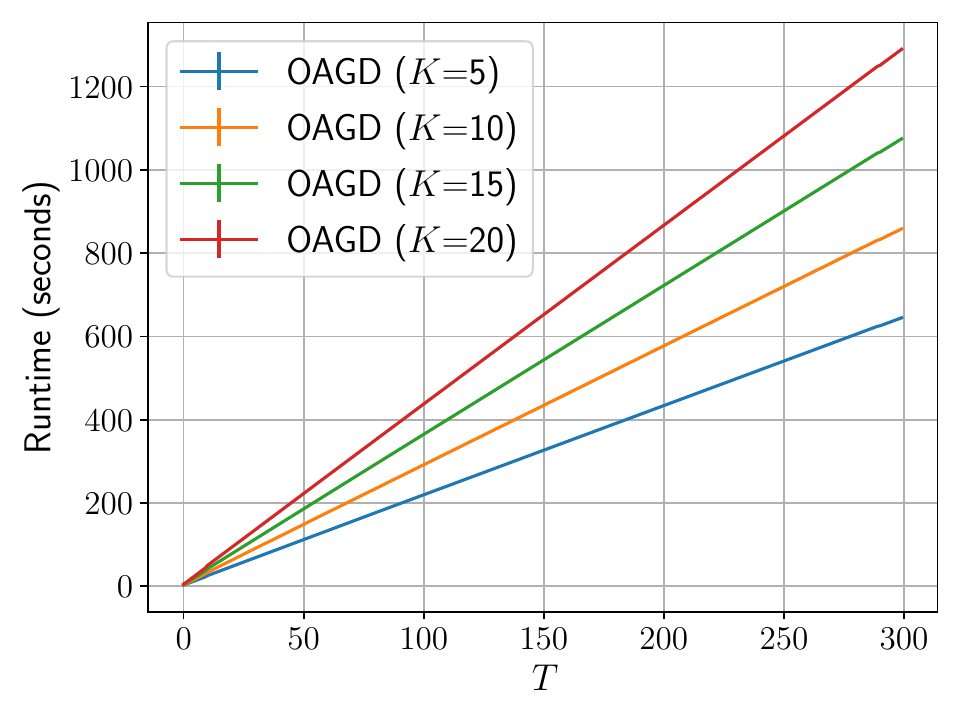}
\includegraphics[scale=0.34]{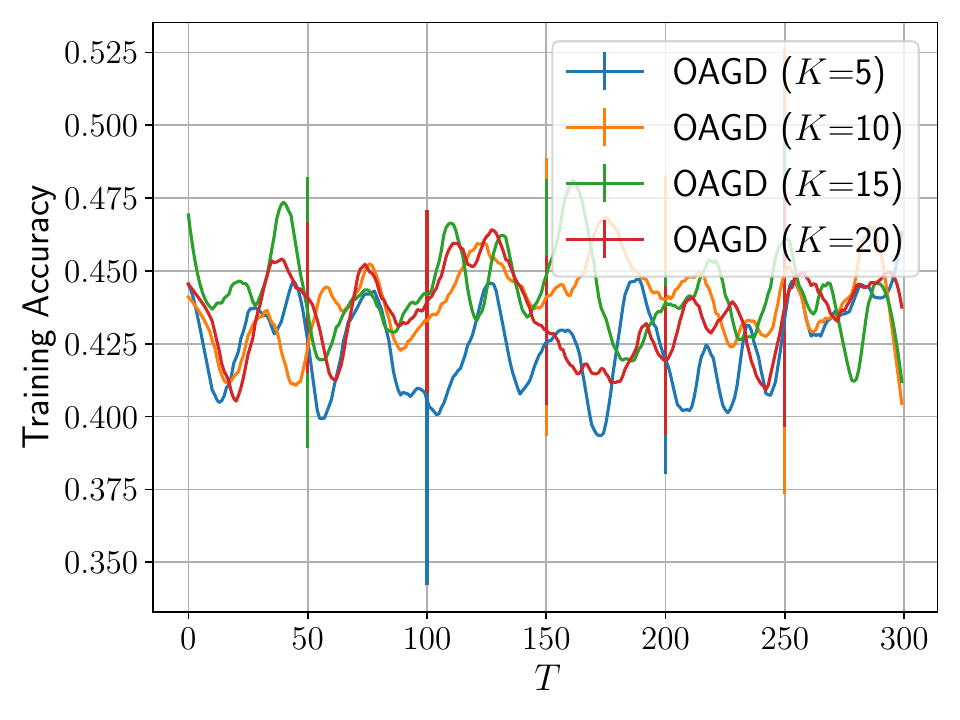} 
\includegraphics[scale=0.34]{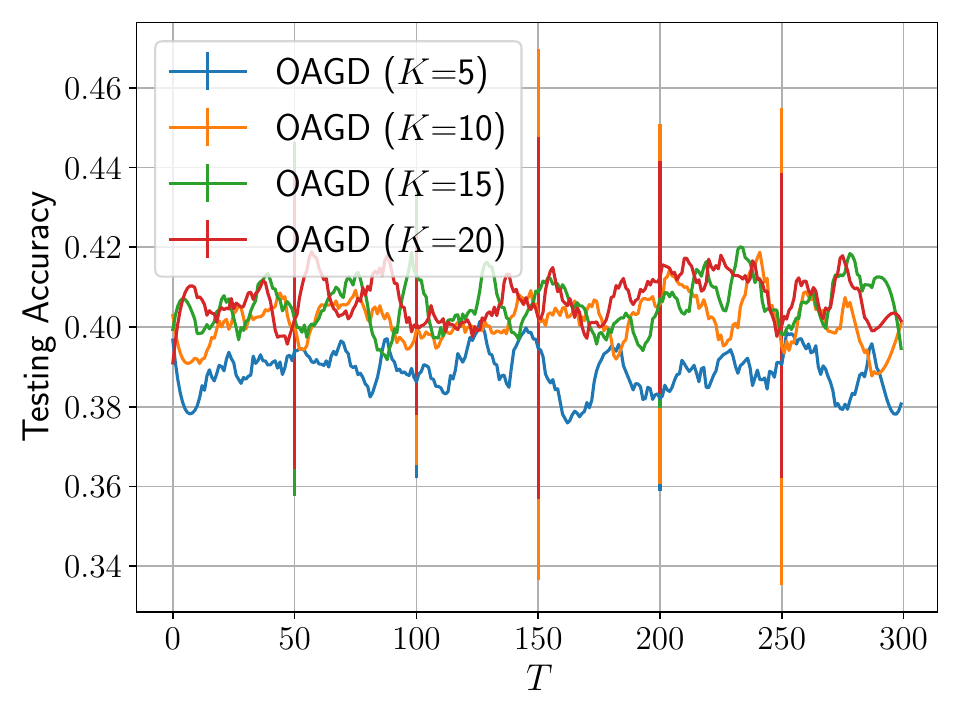}
\vspace{-3mm}
\caption{\label{fig:miniimage_compare_inner_stp}Performance comparison (mean$\pm$std) on online meta-learning for \textbf{miniImageNet} data across five runs. We compare our OAGD ($w=10$) with varying inner steps for inner optimization. A larger inner step may yield improved and more stable accuracy but leads to longer runtime.}
\end{center}
\vspace{-12pt}
\end{figure*}
\subsubsection{Sensitivity Analysis to Inner Step $K$}
Figure~\ref{fig:miniimage_compare_inner_stp} provides the sensitivity analysis results to different numbers of inner optimization steps $K_t=K$ for each round $t$. We can see that a larger number of inner optimization steps can make the accuracy better and more stable. However, this will take longer runtime.

\end{document}